\documentclass[12pt]{amsart}
\usepackage{amssymb}
\usepackage{amsmath,enumerate}
\usepackage{graphicx}

\usepackage{color}

\theoremstyle{definition}
\newtheorem{Def}{Definition}[section]

\theoremstyle{remark}
\newtheorem{Rem}[Def]{Remark}

\newtheorem{Claim}[Def]{Claim}
\newtheorem*{Ack}{Acknowledgments}
\theoremstyle{theorem}
\newtheorem{Th}[Def]{Theorem}
\newtheorem{Prop}[Def]{Proposition}
\newtheorem{Lem}[Def]{Lemma}
\newtheorem{Cor}[Def]{Corollary}
\newtheorem{Fact}[Def]{Fact}

\newcommand{\R}{\mathbb{R}}
\newcommand{\Z}{\mathbb{Z}}

\newcommand{\C}{\mathbb{C}}
\renewcommand{\P}{\mathbb{P}}

\newcommand{\CL}{\mathcal{L}}

\newcommand{\al}{\alpha }
\newcommand{\be}{\beta }
\newcommand{\ga}{\gamma }

\newcommand{\vep}{\varepsilon }
\newcommand{\vph}{\varphi }
\newcommand{\la}{\lambda }

\newcommand{\im}{{\rm Im}}

\newcommand{\ztwo}{\{ 0,1\}}
\newcommand{\tal}{\tilde{\alpha} }
\newcommand{\ttal}{\tilde{\tilde{\alpha}}}
\newcommand{\talp}{{\tilde{\alpha}'} }
\newcommand{\ttalp}{{\tilde{\tilde{\alpha}}'}}
\newcommand{\tga}{\tilde{\gamma} }

\newcommand{\ol}[1]{\overline{#1} }

\makeatletter
\@addtoreset{equation}{section}

\makeatother
\title[Fundamental group for Lauricella's $F_C$]{
  The fundamental group of the complement of 
  the singular locus of Lauricella's $F_C$
}
\author[Y. Goto]{Yoshiaki Goto}
\address[Goto]{
  General Education,
  Otaru University of Commerce,
  Otaru, Hokkaido, 047-8501, Japan 
}
\email{goto@res.otaru-uc.ac.jp}
\author[J. Kaneko]{Jyoichi Kaneko}
\address[Kaneko]{
   Department of Mathematical Science,
   University of the Ryukyus,
   Nishihara, Okinawa, 903-0213, Japan
}
\email{kaneko@math.u-ryukyu.ac.jp}

\date{\today}

\begin{document}
\maketitle

\begin{abstract}
  We study the fundamental group of the complement of 
  the singular locus of Lauricella's hypergeometric function $F_C$
  of $n$ variables. 
  The singular locus consists of $n$ hyperplanes and 
  a hypersurface of degree $2^{n-1}$ in the complex $n$-space. 
  We derive some relations that holds for general $n\geq 3$. 
  We give an explicit presentation of the fundamental group
  in the three-dimensional case.
  We also consider a presentation of the fundamental group of 
  $2^3$-covering of this space. 
\end{abstract}

\section{Introduction}\label{section-intro}
In the study of the monodromy representation of Lauricella's hypergeometric function $F_C$
(see, e.g., \cite{G-FC-monodromy}), 
we consider the fundamental group of the complement of the following hypersurfaces: 
\begin{align*}
  (x_1 =0) ,\ \ldots ,(x_n=0),\ 
  S^{(n)}=(F_n(x) =0 ) \subset \C^n , 
\end{align*}
where
\begin{align*}
  F_n(x)= \prod _{(a_1 ,\ldots ,a_n)\in \ztwo^n} 
    \left( 1-\sum_{k=1}^n (-1)^{a_k} \sqrt{x_k} \right) .
\end{align*}
Note that $F_n (x)$ is an irreducible polynomial in $x_k$'s of degree $2^{n-1}$. 
$S^{(3)}$ is known as a Steiner surface (see, e.g., \cite{Mukai}). 

Throughout this paper, we assume $n\geq 2$. 
Let $X^{(n)}$ be the complement of these hypersurfaces. 
We consider $n+1$ loops $\ga_0 , \ga_1 ,\ldots , \ga_n$ in $X^{(n)}$; 
$\ga_k \ (1\leq k \leq n)$ turns the divisor $(x_k=0)$, 
and $\ga_0$ turns the divisor $S^{(n)}$
around the point $\left( \frac{1}{n^2},\ldots, \frac{1}{n^2} \right) \in S^{(n)}$.
Explicit definitions are given in Section \ref{section-preliminaries}. 
\begin{Fact}[\cite{G-FC-monodromy}]
  The fundamental group $\pi_1 (X^{(n)})$ is generated by 
  $\ga_0 ,\ \ga_1 ,\ldots ,\ \ga_n$. 
  Further, they satisfy the following relations: 
  \begin{align}
    \tag{$R_{ij}$}\label{rel1}
    &[\ga_i ,\ga_j] =1 \quad (1\leq i<j \leq n), \\
    \tag{$R_k$}\label{rel2}
    &(\ga_0 \ga_k)^2 =(\ga_k \ga_0)^2  \quad (1\leq k \leq n),
  \end{align}
  where $[\al, \be ]=\al \be \al^{-1} \be^{-1}$. 
\end{Fact}
In this paper, we discuss the following:
\begin{itemize}
\item another relation in $\pi_1 (X^{(n)})$ for $n\geq 3$,
\item precise calculation of $\pi_1 (X^{(3)})$, 
\item a presentation of $\pi_1 (\tilde{X}^{(3)})$, 
  where $\tilde{X}^{(3)}$ is a $2^3$-covering of $X^{(3)}$. 
\end{itemize}
The main part of this paper is calculation for $n=3$  
(the second and third topics). 
In the following, we explain each topic. 

First, we give another relation in $\pi_1 (X^{(n)})$, by using
similar methods to \cite{G-FC-monodromy}: 
\begin{Th}\label{th-rel}
  For $I=\{ i_1 ,\ldots, i_p\}$, $J=\{ j_1 ,\ldots, j_q\} \subset \{ 1,\ldots ,n\}$ 
  with $p,q \geq 1$, $p+q \leq n-1$ and $I \cap J=\emptyset$, 
  we have
  \begin{align}
    \tag{$R_{IJ}$}\label{rel3}
    [ (\ga_{i_1} \cdots \ga_{i_p})\ga_0 (\ga_{i_1} \cdots \ga_{i_p})^{-1}, 
    (\ga_{j_1} \cdots \ga_{j_q})\ga_0 (\ga_{j_1} \cdots \ga_{j_q})^{-1}]=1 .
  \end{align}
\end{Th}
Note that if $n=2$, this relation does not appear, 
and it is shown in \cite{Kaneko} that the relations (\ref{rel1}) and (\ref{rel2})
generate all relations in $\pi_1 (X^{(2)})$, that is, 
\begin{align*}
  \pi_1(X^{(2)})=\langle \ga_0 , \ga_1 , \ga_2  \mid 
  [\ga_1,\ga_2]=1, (\ga_0\ga_1)^2=(\ga_1\ga_0)^2 ,(\ga_0\ga_2)^2=(\ga_2\ga_0)^2 \rangle .
\end{align*}

Second, we prove that if $n=3$, 
the relations (\ref{rel1}), (\ref{rel2}) and (\ref{rel3}) 
generate all relations in $\pi_1 (X^{(3)})$, that is, 
\begin{Th}
  \label{th-3dim}
  \begin{align*}
    \pi_1(X^{(3)})=\left\langle \ga_0,\ga_1,\ga_2,\ga_3  \left| 
        \begin{array}{l}
          [\ga_i,\ga_j]=1, \ [\ga_i \ga_0 \ga_i^{-1}, \ga_j \ga_0 \ga_j^{-1}]=1 
          \ (1\leq i<j \leq 3) \\
          (\ga_0\ga_k)^2=(\ga_k\ga_0)^2 \ (1\leq k \leq 3)
        \end{array}
      \right. \right\rangle .
  \end{align*}
\end{Th}
To prove this theorem, 
we compute $\pi_1(X^{(3)})$ in detail by using the theorem of van Kampen-Zariski. 
We cut $X^{(3)}$ by a plane and consider a pencil of lines. 
Then we obtain many monodromy relations, and we reduce them 
to those in the theorem. 

Finally, we consider 
a covering space $\tilde{X}^{(n)}$ of $X^{(n)}$
(especially, the case of $n=3$). 
If we put $x_k=\xi_k^2$, then $F_n$ is decomposed into 
$2^n$ linear forms in $\xi_k$'s. 
This means that there exists a $2^n$-covering space $\tilde{X}^{(n)}$ of $X^{(n)}$ 
which is a complement of hyperplanes. 
By using our presentation of $\pi_1(X^{(3)})$ and 
the Reidemeister-Schreier method, 
we also obtain the presentation of $\pi_1 (\tilde{X}^{(3)})$. 
There are several studies for the fundamental group of 
the complements of hyperplane arrangements (see, e.g., \cite{Orlik-Terao}, \cite{Randell}). 
However, it seems difficult to present $\pi_1 (\tilde{X}^{(n)})$ explicitly, 
even if we apply these results. 

The relations (\ref{rel1}), (\ref{rel2}) and (\ref{rel3}) 
generate all relations in $\pi_1 (X^{(n)})$ for $n=2,3$, 
while we do not know if the same claim holds for $n\geq 4$ or not.
As in Section \ref{section-3dim}, a plane cut of $X^{(3)}$ has three nodes, 
and we can interpret that these nodes correspond to three relations (\ref{rel3}). 
However, a plane cut of $X^{(4)}$ has 20 nodes, and 
the number of relations (\ref{rel3}) are 18. 
Thus, it seems that
relations different from the above ones hold.

Since our detailed calculations are very long, we omit some of them
in this paper.
For all the calculations, refer to 
\begin{center}
  \texttt{https://arxiv.org/abs/1710.09594v1}.
\end{center}

\begin{Ack}
  The authors are grateful to Susumu Tanab\'e  
  for helpful discussions. 
\end{Ack}

\section{Preliminaries}\label{section-preliminaries}
We give explicit definitions of the loops $\ga_0 , \ga_1 ,\ldots , \ga_n$. 

Let $\dot{x}=(\frac{1}{2n^2},\ldots ,\frac{1}{2n^2})\in X^{(n)}$ be a base point. 
For $1\leq k \leq n$, let $\ga_k$ be the loop in $X^{(n)}$ defined by
$$
\ga_k :[0,1] \ni \theta \mapsto 
\Bigg(  \frac{1}{2n^2} ,\ldots ,
\underset{k\textrm{-th}}{\frac{e^{2\pi \sqrt{-1}\theta}}{2n^2}},
\ldots ,\frac{1}{2n^2} \Bigg) \in X^{(n)}.
$$
We take a positive real number $\vep_0$ so that 
$\vep_0 <\min \left\{ \frac{1}{2n^2}, \frac{1}{(n-2)^2}-\frac{1}{n^2} \right\}$, 
and we define the loop $\ga_0$ in $X^{(n)}$ as $\ga_0 =\tau_0 \ga'_0 \overline{\tau_0}$, 
where 
\begin{align*}
  \tau_0 &:[0,1] \ni \theta \mapsto 
  \left( (1-\theta) \cdot \frac{1}{2n^2}+\theta \cdot \left( \frac{1}{n^2}-\vep_0 \right) \right) 
  (1,\ldots ,1) \in X^{(n)}, \\
  \ga'_0 &:[0,1] \ni \theta \mapsto 
  \left( \frac{1}{n^2} -\vep_0 e^{2\pi \sqrt{-1}\theta}  \right)
  (1,\ldots ,1) \in X^{(n)},
\end{align*}
and $\overline{\tau_0}$ is the reverse path of $\tau_0$. 
\begin{Rem}
  The loop $\ga_k \ (1\leq k \leq m)$ turns the hyperplane $(x_k=0)$, 
  and $\ga_0$ turns the hypersurface $S^{(n)}$ around the point 
  $\left( \frac{1}{n^2},\ldots, \frac{1}{n^2} \right)$, positively. 
  Note that $\left( \frac{1}{n^2},\ldots ,\frac{1}{n^2} \right)$ is the nearest to the origin 
  in $S^{(n)}\cap (x_1=x_2=\cdots =x_m)=\left\{ \frac{1}{n^2}(1,\ldots ,1),
  \frac{1}{(n-2)^2}(1,\ldots ,1),\ldots  \right\}$. 
\end{Rem}

\section{Proof of Theorem \ref{th-rel}}\label{section-rel}
In this section, we assume $n\geq 3$ and prove Theorem \ref{th-rel}. 
We use similar methods to \cite[Appendix]{G-FC-monodromy}. 
However, we change some notations for our convenience. 

We regard $\C^n$ as a subset of $\P^n$ and put $L_{\infty} =\P^n -\C^n$. 
Then we can consider that $S^{(n)}$ is a hypersurface in $\P^n$, and 
\begin{align*}
  X^{(n)}&=\C^n-\Big( (x_1\cdots x_n=0)\cup S^{(n)} \Big) \\
  &=\P^n -\Big( (x_1\cdots x_n=0)\cup S^{(n)} \cup L_{\infty } \Big).
\end{align*}
By a special case of the Zariski theorem of Lefschetz type 
(see, e.g., \cite[Chapter 4 (3.1)]{Dimca}), 
the inclusion $L\cap X^{(n)} \hookrightarrow X^{(n)}$ 
induces an epimorphism 
$$
\eta : 
\pi_1 \left( L\cap X^{(n)} \right) \to 
\pi_1 (X^{(n)}) ,
$$
for a line $L$ in $\P^n$ which intersects $\P^n -X^{(n)}$ transversally 
and avoids its singular parts. 
Note that generators of 
$\pi_1 (L\cap X^{(n)})$ 
are given by $n+2^{n-1}$ loops going once around
each of the intersection points in $L\cap ((x_1\cdots x_n=0)\cup S^{(n)}) \subset \C^n$. 
To define loops in $X^{(n)}$ explicitly, 
we specify such a line $L$ 
in the following way. 
Let $r_1 ,\ldots ,r_{n-1}$ 
be positive real numbers satisfying 
$$
r_{n-1} <\frac{1}{4} ,\quad r_k <\frac{r_{k+1}}{4} \ (1 \leq k \leq n-2),  
$$
and let $\vep=(\vep_1 ,\ldots ,\vep_{n-1})$ be sufficiently small positive real numbers 
such that $\vep_1 >\cdots >\vep_{n-1}$. 
We consider lines 
\begin{align*}
  L_0 : &(x_1 ,\ldots ,x_{n-1},x_n)
  =(r_1 ,\ldots ,r_{n-1} ,0) 
  +t(0 ,\ldots ,0 ,1) \quad (t\in \C), \\
  L_{\vep} : &(x_1 ,\ldots ,x_{n-1},x_n)
  =(r_1 ,\ldots ,r_{n-1} ,0) 
  +t(\vep_1 ,\ldots ,\vep_{n-1} ,1) \quad (t\in \C)
\end{align*}
in $\C^n$. 
We identify $L_{\vep}$ with $\C$ by the coordinate $t$. 
The intersection point 
$L_{\vep} \cap (x_k=0)$ is coordinated by 
$t=-\frac{r_k}{\vep_k} <0$, for $1 \leq k \leq n-1$. 
The intersection point $L_{\vep} \cap (x_n=0)$
is coordinated by $t=0$.
$L_{\vep}$ and $S^{(n)}$ intersect at $2^{n-1}$ points. 
We coordinate the intersection points $L_{\vep} \cap S^{(n)}$ by 
$t=t_{a_1 \cdots a_{n-1}}, \ (a_1 ,\ldots ,a_{n-1}) \in \ztwo^{n-1}$. 
The correspondence is as follows. 
We denote the coordinates of the intersection points $L_0 \cap S^{(n)}$ by 
$$
t_{a_1 \cdots a_{n-1}} ^{(0)}=
\Bigg( 1-\sum_{k=1}^{n-1} (-1)^{a_k} \sqrt{r_k} \Bigg)^2 .
$$
By this definition, we have 
\begin{align}
  \nonumber
  &t_{a_1 \cdots a_{n-1}}^{(0)} <t_{a'_1 \cdots a'_{n-1}}^{(0)} \\
  \label{eq-order}
  &\Longleftrightarrow  
  \exists r \textrm{ s.t. }a_{i} -a'_{i}=0 \ (i=r+1 ,\ldots ,n-1),\ 
  a_r=0,\ a'_r=1 .
\end{align}
For $(a_1 ,\ldots ,a_{n-1}),(a'_1 ,\ldots ,a'_{n-1})\in \ztwo^{n-1}$, 
we denote $(a_1 ,\ldots ,a_{n-1}) \prec (a'_1 ,\ldots ,a'_{n-1})$ 
when (\ref{eq-order}) holds. 
For example, if $n=4$, then 
$$
t_{000}^{(0)}<t_{100}^{(0)}<t_{010}^{(0)}<t_{110}^{(0)}<
t_{001}^{(0)}<t_{101}^{(0)}<t_{011}^{(0)}<t_{111}^{(0)}. 
$$
Since $L_{\vep}$ is sufficiently close to $L_0$, 
$t_{a_1 \cdots a_{n-1}}$ is supposed to be arranged near to 
$t_{a_1 \cdots a_{n-1}} ^{(0)}$. 

Since $L_0$ does not pass the singular part of $S^{(n)}$, 
for sufficiently small $\vep_k$'s, 
$L_{\vep}$ also avoids the singular parts of $\P^n -X^{(n)}$. 
Thus, 
$\eta_{\vep} : 
\pi_1 \left( L_{\vep}\cap X^{(n)} \right) 
\to \pi_1 (X^{(n)})$ 
is an epimorphism. 

Let $\ell_k$ be the 
loop in $L_{\vep}\cap X^{(n)}$ going once around the intersection point 
$L_{\vep} \cap (x_k=0)$, and let 
$\ell_{a_1 \cdots a_{n-1}}$ be the loop going 
once around the intersection point $t_{a_1 \cdots a_{n-1}}$. 
Each loop approaches the intersection point through 
the upper half-plane of the $t$-space.

As in \cite{G-FC-monodromy}, we have 
\begin{align*}
  &\eta_{\vep} (\ell_k )=\ga_k \ (1\leq k \leq n) ,\quad
  \eta_{\vep} (\ell_{0\cdots 0}) =\ga_0 , \\
  &\ga_i \ga_j =\ga_j \ga_i \quad (1\leq i,j \leq n) . 
\end{align*}
To investigate relations among the $\eta_{\vep}(\ell_{a_1 \cdots a_{n-1}})$'s, 
we consider these loops in 
$L_0 \cap X^{(n)}$. 
By the above definition, we can define 
the $\ell_{a_1 \cdots a_{n-1}}$'s as loops in 
$L_0 \cap X^{(n)}$. 
Since $L_0$ is sufficiently close to $L_{\vep}$, 
the image of $\ell_{a_1 \cdots a_{n-1}}$ under 
$$
\eta : 
\pi_1 \left( L_0 \cap X^{(n)} \right) 
\to \pi_1 (X^{(n)})
$$
coincides with $\eta_{\vep}(\ell_{a_1 \cdots a_{n-1}})$ 
as elements in $\pi_1 (X^{(n)})$. 
Though $\eta$ is not an epimorphism, 
relations among the $\eta (\ell_{a_1 \cdots a_{n-1}})$'s in $\pi_1 (X^{(n)})$ 
can be regarded as those among the $\eta_{\vep}(\ell_{a_1 \cdots a_{n-1}})$'s. 

In \cite{G-FC-monodromy}, we move $L_0$ as follows. 
For $\theta \in [0,1]$, 
let $L(\theta)$ be the line defined by 
\begin{align*}
  L(\theta) :&(x_1 ,\ldots ,x_k ,\ldots ,x_{n-1} ,x_n) \\
  &=(r_1 ,\ldots ,e^{2\pi \sqrt{-1}\theta }r_k ,\ldots ,r_{n-1} ,0) 
  +t(0 ,\ldots ,0 ,1) \quad (t\in \C) .
\end{align*}
Note that $L(0)=L(1)=L_0$. 
We identify $L(\theta)$ with $\C$ by the coordinate $t$. 
It is easy to see that the intersection points of 
$L(\theta)$ and $S^{(n)}$ are given by the following $2^{n-1}$ elements:
$$
t_{a_1 \cdots a_{n-1}} ^{(\theta )}=
\Bigg( 1-\sum_{\substack{j=1\\j \neq k}}^{n-1} (-1)^{a_j} \sqrt{r_j} 
  -(-1)^{a_k} \sqrt{r_k}e^{\pi \sqrt{-1}\theta} \Bigg)^2 .
$$
The points $1-\sum_{j \neq k} (-1)^{a_j} \sqrt{r_j} 
-(-1)^{a_k} \sqrt{r_k}e^{\pi \sqrt{-1}\theta}$ are in the right half-plane 
for any $\theta \in [0,1]$, 
since $\sum_{j=1}^{n-1} \sqrt{r_j} <\sum_{j=1}^{n-1} 2^{-j}<1$. 
Let $\theta$ move from $0$ to $1$, then 
\begin{enumerate}[(a)]
\item $t_{a_1 \cdots a_{k-1} 0 a_{k+1} \cdots a_{m-1}}^{(1)}
  =t_{a_1 \cdots a_{k-1} 1 a_{k+1} \cdots a_{m-1}}^{(0)}$,  
  $t_{a_1 \cdots a_{k-1} 0 a_{k+1} \cdots a_{m-1}}^{(1)}
  =t_{a_1 \cdots a_{k-1} 1 a_{k+1} \cdots a_{m-1}}^{(0)}$, 
\item $t_{a_1 \cdots a_{k-1} 1 a_{k+1} \cdots a_{m-1}}^{(\theta)}$ 
  moves in the upper half-plane, 
\item $t_{a_1 \cdots a_{k-1} 0 a_{k+1} \cdots a_{m-1}}^{(\theta)}$ 
  moves in the lower half-plane. 
\end{enumerate}
For example, the $t_{a_1 a_2 a_3}$'s move as Figure \ref{point-change}, 
for $n=4$ and $k=2$. 
\begin{figure}[h]
  \centering{
  \includegraphics[scale=1.0]{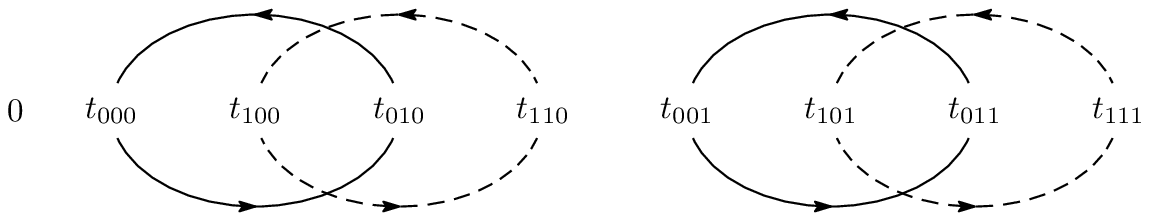} }
  \caption{$t_{a_1 a_2 a_3}$ for $n=4,\ k=2$. \label{point-change}}
\end{figure}

We put $P(\theta) =\C -\{ t_{a_1 \cdots a_{n-1}}^{(\theta)} \mid a_j \in \{ 0,1 \} \}$ 
that is regarded as a subset of $L(\theta )$. 
Let $\vep'$ be a sufficiently small positive real number, 
and we consider the fundamental group $\pi_1 (P(\theta),\vep')$. 
As mentioned above, the $\ell_{a_1 \cdots a_{n-1}}$'s are defined as 
elements in $\pi_1 (P(0),\vep')=\pi_1 (P(1),\vep')$. 
If we move $\theta$ from $0$ to $1$, 
then the $\ell_{a_1 \cdots a_{n-1}}$'s 
define the elements in each $\pi_1 (P(\theta),\vep')$ naturally. 

Note that by this variation, the base point moves around the divisor $(x_k=0)$, 
since the base point $\vep' \in P(\theta)$ corresponds to the point 
$(r_1 ,\ldots ,e^{2\pi \sqrt{-1}\theta }r_k ,\ldots ,r_{m-1} ,\vep') \in L(\theta)$. 
It implies the conjugation by $\ga_k$ in $\pi_1 (X^{(n)})$. 

In \cite{G-FC-monodromy}, we investigate the loops
$\ell_{a_1 \cdots a_{k-1} 1 a_{k+1} \cdots a_{m-1}}$ under this variation, 
and we obtain the following. 
\begin{Fact}[{\cite[Lemma A.1 (i)]{G-FC-monodromy}}]\label{fact-ell}
  We have
  \begin{align*}
    &\eta(\ell_{a_1 \cdots a_{k-1} 1 a_{k+1} \cdots a_{m-1}})
    =\ga_k \cdot \eta(\ell_{a_1 \cdots a_{k-1} 0 a_{k+1} \cdots a_{m-1}}) \cdot \ga_k^{-1} .
  \end{align*}
  Furthermore, we obtain 
  \begin{align*}
    \eta(\ell_{a_1 \cdots a_{n-1}}) 
    =(\ga_1^{a_1} \cdots \ga_{n-1}^{a_{n-1}}) \cdot 
    \ga_0 \cdot
    (\ga_1^{a_1} \cdots \ga_{n-1}^{a_{n-1}})^{-1} .
  \end{align*}
\end{Fact}
By considering the case 
$k=1$ and $\ell_{0\cdots 0} \in \pi_1 (P(0),\vep')$,  
we also obtain the following
\begin{Fact}[{\cite[Lemma A.1 (ii)]{G-FC-monodromy}}]
  We have
  \begin{align*}
    \eta(\ell_{0 \cdots 0}) 
    =\ga_1 \cdot \eta(\ell_{0\cdots 0} \ell_{10\cdots 0} \ell_{0\cdots 0}^{-1}) 
    \cdot \ga_1^{-1}, 
  \end{align*}
  and this implies $(\ga_0 \ga_1)^2 =(\ga_1 \ga_0)^2$.
\end{Fact}
\begin{Rem}\label{rem-change-index}
  Changing the definitions of $L_0$ and $L_{\vep}$, 
  we obtain the relations 
  $$
  (\ga_0 \ga_k)^2 =(\ga_k \ga_0)^2 \quad (1\leq k \leq n).
  $$
  For example, if we put 
  $$
  L_{\vep} : (x_1 ,x_2, \ldots ,x_n)
  =(0,r_1 ,\ldots ,r_{n-1} ) 
  +t(1, \vep_1 ,\ldots ,\vep_{n-1} ) \quad (t\in \C) ,
  $$
  then a similar argument shows 
  $(\ga_0 \ga_2)^2 =(\ga_2 \ga_0)^2$. 
  By the same reason, 
  for the proof of Theorem \ref{th-rel}, 
  it suffices to show that  
  \begin{align}
    \label{eq-simple-rel}
    [ (\ga_{1} \cdots \ga_{p-1}\ga_{p+q})\ga_0 (\ga_{1} \cdots \ga_{p-1}\ga_{p+q})^{-1}, 
    (\ga_{p} \cdots \ga_{p+q-1})\ga_0 (\ga_{p} \cdots \ga_{p+q-1})^{-1}]=1 ,
  \end{align}
  for any $p,q\geq 1$, $p+q\leq n-1$. 
  (Though the indices are complicated, this formulation is convenient for our proof.)
  Note that if $p=1$, we regard $\ga_{1}\cdots \ga_{p-1}=1$. 
\end{Rem}

Now, we investigate changes of other loops to prove Theorem \ref{th-rel}. 
For $p,q\geq 1$, $p+q\leq n-1$, we consider 
$$
\ell_{1\cdots 1 0 \cdots 0 0 \cdots 0}=
\ell_{\underbrace{1\cdots 1}_{p-1}\underbrace{0\cdots 0}_{q}\underbrace{0\cdots 0}_{n-p-q}} ,
$$
and its change for $k=p+q$. 
By the properties (a), (b), (c) and the proof of \cite[Lemma A.2]{G-FC-monodromy} 
show the following. 
\begin{Lem}\label{lem-variation}
  \label{lem-loop-change}
  $\ell_{1\cdots 1 0 \cdots 0 0 \cdots 0}$ in $\pi_1 (P(0),\vep')$ 
  changes into 
  \begin{align*}
    \left(\prod_{(a_1,\ldots ,a_{p+q-1})\in \ztwo^{p+q-1}}^{\prec} 
      \ell_{a_1 \cdots a_{p+q-1} 0\cdots 0} \right)
    \ell_{1\cdots 1 0 \cdots 0 1 \cdots 0}
    \left(\prod_{(a_1,\ldots ,a_{p+q-1})\in \ztwo^{p+q-1}}^{\prec} 
      \ell_{a_1 \cdots a_{p+q-1} 0\cdots 0} \right)^{-1}
  \end{align*}
  in $\pi_1 (P(1),\vep')$, where the notation $\displaystyle \prod^{\prec}$ means 
  the product multiplying in ascending order of indices with respect to $\prec$. 
\end{Lem}
For example, 
$$
\prod_{(a_1,a_2)\in \ztwo^{2}}^{\prec} \ell_{a_1 a_2 0 0}
=\ell_{0000}\ell_{1000}\ell_{0100}\ell_{1100}. 
$$ 
Since this variation corresponds to the conjugation by $\ga_k=\ga_{p+q}$, 
Fact \ref{fact-ell} and Lemma \ref{lem-variation} implies 
\begin{align*}
  &(\ga_{1}\cdots \ga_{p-1})\ga_0 (\ga_{1}\cdots \ga_{p-1})^{-1}\\
  &=\ga_{p+q}
  \Bigg(\prod_{(a_1,\ldots ,a_{p+q-1})\in \ztwo^{p+q-1}}^{\prec} 
  (\ga_{1}^{a_1}\cdots \ga_{p+q-1}^{a_{p+q-1}})\ga_0 
  (\ga_{1}^{a_1}\cdots \ga_{p+q-1}^{a_{p+q-1}})^{-1} \Bigg) \\
  &\quad \cdot (\ga_{1} \cdots \ga_{p-1}\ga_{p+q})\ga_0 (\ga_{1} \cdots \ga_{p-1}\ga_{p+q})^{-1} \\
  &\quad \cdot \Bigg(\prod_{(a_1,\ldots ,a_{p+q-1})\in \ztwo^{p+q-1}}^{\prec} 
  (\ga_{1}^{a_1}\cdots \ga_{p+q-1}^{a_{p+q-1}})\ga_0 
  (\ga_{1}^{a_1}\cdots \ga_{p+q-1}^{a_{p+q-1}})^{-1} \Bigg)^{-1} 
  \ga_{p+q}^{-1} .
\end{align*}
Note that the first factor of $\displaystyle \prod^{\prec}$
is $\ga_0$. 
Multiplying $\ga_0^{-1} \ga_{p+q}^{-1}$ by left and $\ga_{p+q} (\cdots)$ by right, we obtain 
\begin{align}
  \label{eq-change-ga}
  &\ga_0^{-1} \ga_{p+q}^{-1}(\ga_{1}\cdots \ga_{p-1})\ga_0 (\ga_{1}\cdots \ga_{p-1})^{-1} 
  \cdot \ga_{p+q} \cdot \ga_0 \\
  \nonumber
  &\cdot \Bigg(\prod_{\substack{(a_1,\ldots ,a_{p+q-1})\in \ztwo^{p+q-1}\\ 
      (a_1,\ldots ,a_{p+q-1}) \neq (0,\ldots ,0)}}^{\prec} 
  (\ga_{1}^{a_1}\cdots \ga_{p+q-1}^{a_{p+q-1}})\ga_0 
  (\ga_{1}^{a_1}\cdots \ga_{p+q-1}^{a_{p+q-1}})^{-1} \Bigg)\\
  \nonumber
  &=
  \Bigg(\prod_{\substack{(a_1,\ldots ,a_{p+q-1})\in \ztwo^{p+q-1}\\ 
      (a_1,\ldots ,a_{p+q-1}) \neq (0,\ldots ,0)}}^{\prec} 
  (\ga_{1}^{a_1}\cdots \ga_{p+q-1}^{a_{p+q-1}})\ga_0 
  (\ga_{1}^{a_1}\cdots \ga_{p+q-1}^{a_{p+q-1}})^{-1} \Bigg) \\
  \nonumber
  &\quad \cdot (\ga_{1} \cdots \ga_{p-1}\ga_{p+q})\ga_0 (\ga_{1} \cdots \ga_{p-1}\ga_{p+q})^{-1} . 
\end{align}
We prove Theorem \ref{th-rel} by using this equality. 
Before starting the proof, we give a useful equality:
\begin{align}
  \label{eq-rel2-2}
  \ga_k^{-1} \ga_0 \ga_k \ga_0 =\ga_0 \ga_k \ga_0 \ga_k^{-1} , 
\end{align}
which is equivalent to the relation (\ref{rel2}). 
We also note that the relations (\ref{rel3}) is equivalent to
\begin{align*}
  [ (\ga_{i_1} \cdots \ga_{i_p})^{-1} \ga_0 (\ga_{i_1} \cdots \ga_{i_p}), 
  (\ga_{j_1} \cdots \ga_{j_q})^{-1} \ga_0 (\ga_{j_1} \cdots \ga_{j_q})]=1, 
\end{align*}
by (\ref{rel1}).

\begin{proof}[Proof of Theorem \ref{th-rel}]
  We show the theorem by induction on $p+q\geq 2$. 
  As mentioned in Remark \ref{rem-change-index}, 
  it is sufficient to show (\ref{eq-simple-rel}) for each $p,q$. 
  Considering the conjugation by $\ga_l$'s, we have the following lemma. 
  \begin{Lem}\label{lem-proof}
    Assume that we have proved (\ref{eq-simple-rel}) for any $p,q$ with $p+q\leq k-1$. 
    Then we obtain the relation (\ref{rel3}) for 
    $$
    I=\{ i_1,\ldots ,i_r \} ,\ J=\{ j_1 ,\ldots ,j_s \} \subset \{ 1,\ldots ,n\}
    $$
    which satisfy $I\not\subset J$, $I\not\supset J$ and $\#(I\triangle J)\leq k-1$. 
    Here, $I\triangle J$ is the symmetric difference of $I$ and $J$. 
  \end{Lem}
  First, we show the case $p+q=2$. 
  In this case, we have only to show that 
  \begin{align}
    \label{eq-proof-2}
    [ \ga_1\ga_0 \ga_{1}^{-1}, \ga_2\ga_0 \ga_2^{-1}]=1 .
  \end{align}
  The equality (\ref{eq-change-ga}) for $p=q=1$ is 
  \begin{align*}
    \ga_0^{-1} \ga_2^{-1} \cdot 1 \cdot \ga_0 \cdot 1^{-1} \cdot \ga_2 
    \cdot \ga_0 \cdot (\ga_1\ga_0 \ga_1^{-1} ) 
    =\ga_1\ga_0 \ga_1^{-1} \cdot \ga_2\ga_0 \ga_2^{-1} .
  \end{align*}
  By (\ref{eq-rel2-2}), the left-hand side equals to 
  $$
  \ga_0^{-1} \ga_2^{-1} \ga_0 \ga_2 \ga_0 \ga_1\ga_0 \ga_1^{-1} 
  =\ga_0^{-1} \ga_0 \ga_2 \ga_0 \ga_2^{-1}  \ga_1\ga_0 \ga_1^{-1}
  =\ga_2 \ga_0 \ga_2^{-1} \cdot \ga_1\ga_0 \ga_1^{-1} .
  $$
  Thus (\ref{eq-proof-2}) is proved. 
  
  Next, we assume that we have proved (\ref{eq-simple-rel}) for any $p,q$ with $p+q\leq k-1$
  (recall Lemma \ref{lem-proof}), and prove (\ref{eq-simple-rel}) in the case $p+q=k$. 
  \begin{Claim}\label{claim-1}
    If $(1,\ldots ,\overset{p-1}{1} ,0,\ldots ,0)\prec (a_1,\ldots ,a_{p+q-1})\preceq (1,\ldots ,1)$ and 
    $(a_1,\ldots ,a_{p+q-1}) \neq (0,\ldots ,\overset{p-1}{0} ,1,\ldots ,1)$, 
    then we have
    \begin{align*}
      [(\ga_{1}^{a_1}\cdots \ga_{p+q-1}^{a_{p+q-1}})\ga_0 
      (\ga_{1}^{a_1}\cdots \ga_{p+q-1}^{a_{p+q-1}})^{-1},
      (\ga_{1} \cdots \ga_{p-1}\ga_{p+q})\ga_0 (\ga_{1} \cdots \ga_{p-1}\ga_{p+q})^{-1}]=1 .
    \end{align*}
  \end{Claim}
  \begin{proof}[Proof of Claim]
    We put $I=\{ i\mid a_i=1 \}$ and $J=\{ 1 ,\ldots ,p-1,p+q\}$, 
    and we show $I$ and $J$ satisfy the conditions in Lemma \ref{lem-proof}. 
    Clearly, $p+q \in J-I$ and hence $I\not\supset J$. 
    Since $(1,\ldots ,\overset{p-1}{1} ,0,\ldots ,0)\prec (a_1,\ldots ,a_{p+q-1})$, 
    there exists $p\leq i \leq p+q-1$ such that $a_i=1$. 
    This implies that $I\not\subset J$. 
    Because of $I\cup J \subset \{1 ,\ldots ,p+q\}$ and 
    $(a_1,\ldots ,a_{p+q-1}) \neq (0,\ldots ,\overset{p-1}{0} ,1,\ldots ,1)$,  
    we obtain $\#(I\triangle J)\leq p+q-1=k-1$, and hence 
    $I$ and $J$ satisfy the conditions in Lemma \ref{lem-proof}.
    Thus, the claim is proved. 
  \end{proof}
  By applying this claim to the right-hand side of (\ref{eq-change-ga}), we have
  \begin{align}
    \label{eq-change-ga-2}
    &\ga_0^{-1} \ga_{p+q}^{-1}(\ga_{1}\cdots \ga_{p-1})\ga_0 (\ga_{1}\cdots \ga_{p-1})^{-1} 
    \cdot \ga_{p+q} \cdot \ga_0 \\
    \nonumber
    &\cdot
    \Bigg(\prod_{(0,\ldots ,0) \prec (a_1,\ldots ,a_{p+q-1}) \prec (0,\ldots ,0,1,\ldots ,1)}^{\prec} 
    (\ga_{1}^{a_1}\cdots \ga_{p+q-1}^{a_{p+q-1}})\ga_0 
    (\ga_{1}^{a_1}\cdots \ga_{p+q-1}^{a_{p+q-1}})^{-1} \Bigg)\\
    \nonumber
    &\cdot (\ga_{p}\cdots \ga_{p+q-1})\ga_0 (\ga_{p}\cdots \ga_{p+q-1})^{-1} \\   
    \nonumber
    &=
    \Bigg(\prod_{(0,\ldots ,0) \prec (a_1,\ldots ,a_{p+q-1}) \prec (0,\ldots ,0,1,\ldots ,1)}^{\prec} 
    (\ga_{1}^{a_1}\cdots \ga_{p+q-1}^{a_{p+q-1}})\ga_0 
    (\ga_{1}^{a_1}\cdots \ga_{p+q-1}^{a_{p+q-1}})^{-1} \Bigg) \\
    \nonumber
    &\quad \cdot (\ga_{p}\cdots \ga_{p+q-1})\ga_0 (\ga_{p}\cdots \ga_{p+q-1})^{-1}
    \cdot (\ga_{1} \cdots \ga_{p-1}\ga_{p+q})\ga_0 (\ga_{1} \cdots \ga_{p-1}\ga_{p+q})^{-1} .
  \end{align}
  We rewrite the first line:
  \begin{align}
    \label{eq-left-LHS1}
    &\ga_0^{-1} \ga_{p+q}^{-1}(\ga_{1}\cdots \ga_{p-1})\ga_0 (\ga_{1}\cdots \ga_{p-1})^{-1} 
    \cdot \ga_{p+q} \cdot \ga_0 \\
    \nonumber
    &=\ga_0^{-1} (\ga_{1}\cdots \ga_{p-1}) \cdot \ga_{p+q}^{-1} \ga_0 \ga_{p+q}
    \cdot (\ga_{1}\cdots \ga_{p-1})^{-1} \ga_0 (\ga_{1}\cdots \ga_{p-1})  \\
    \nonumber
    &\quad \cdot (\ga_{1}\cdots \ga_{p-1})^{-1} 
    \quad (\textrm{we can use the induction hypothesis.})\\
    \nonumber
    \nonumber
    &=(\ga_{1}\cdots \ga_{p-1})
    \cdot \ga_{p+q}^{-1} \ga_0 \ga_{p+q} \cdot (\ga_{1}\cdots \ga_{p-1})^{-1} .
  \end{align}
  \begin{Claim}\label{claim-2}
    This product commutes with
    $(\ga_{1}^{a_1}\cdots \ga_{p+q-1}^{a_{p+q-1}})\ga_0 
    (\ga_{1}^{a_1}\cdots \ga_{p+q-1}^{a_{p+q-1}})^{-1}$, 
    for $(0,\ldots ,0) \prec (a_1,\ldots ,a_{p+q-1}) \prec (1\ldots ,\overset{p-1}{1},0,\ldots ,0)$. 
  \end{Claim}
  \begin{proof}[Proof of Claim]
    Since $(0,\ldots ,0) \prec (a_1,\ldots ,a_{p+q-1}) \prec (1\ldots ,1,0,\ldots ,0)$, 
    we have $a_i=0$ for $i\geq p$ and 
    \begin{align*}
      &(\ga_{1}^{a_1}\cdots \ga_{p+q-1}^{a_{p+q-1}})\ga_0 
      (\ga_{1}^{a_1}\cdots \ga_{p+q-1}^{a_{p+q-1}})^{-1} \\
      &=(\ga_{1}^{a_1}\cdots \ga_{p-1}^{a_{p-1}})\ga_0 
      (\ga_{1}^{a_1}\cdots \ga_{p-1}^{a_{p-1}})^{-1} \\
      &=(\ga_{1}\cdots \ga_{p-1})
      (\ga_{1}^{1-a_1}\cdots \ga_{p-1}^{1-a_{p-1}})^{-1} \ga_0 
      (\ga_{1}^{1-a_1}\cdots \ga_{p-1}^{1-a_{p-1}})
      (\ga_{1}\cdots \ga_{p-1})^{-1} .
    \end{align*}
    Thus, the claim is equivalent to 
    $$
    [\ga_{p+q}^{-1} \ga_0 \ga_{p+q}, 
    (\ga_{1}^{1-a_1}\cdots \ga_{p-1}^{1-a_{p-1}})^{-1} \ga_0 
    (\ga_{1}^{1-a_1}\cdots \ga_{p-1}^{1-a_{p-1}})]=1 . 
    $$
    This relation follows from the induction hypothesis and 
    $(1-a_1,\ldots ,1-a_{p-1})\neq (0,\ldots ,0)$. 
  \end{proof}
  Since (\ref{eq-left-LHS1}) is changed into 
  \begin{align*}
    &(\ga_{1}\cdots \ga_{p-1})
    \cdot \ga_{p+q}^{-1} \ga_0 \ga_{p+q} \cdot (\ga_{1}\cdots \ga_{p-1})^{-1} \\
    &=(\ga_{1}\cdots \ga_{p-1})
    \cdot \ga_{p+q}^{-1} \ga_0 \ga_{p+q} 
    \cdot \ga_0 \ga_0^{-1} 
    \cdot (\ga_{1}\cdots \ga_{p-1})^{-1} \\
    &=(\ga_{1}\cdots \ga_{p-1})
    \cdot \ga_0 \ga_{p+q} \ga_0 \ga_{p+q}^{-1} \ga_0^{-1}
    \cdot (\ga_{1}\cdots \ga_{p-1})^{-1} \\
    &=(\ga_{1}\cdots \ga_{p-1})\ga_0 (\ga_{1}\cdots \ga_{p-1})^{-1} \cdot 
    (\ga_{1}\cdots \ga_{p-1}\ga_{p+q}) \ga_0 (\ga_{1}\cdots \ga_{p-1}\ga_{p+q})^{-1} \\
    &\quad \cdot \left( (\ga_{1}\cdots \ga_{p-1}) \ga_0(\ga_{1}\cdots \ga_{p-1})^{-1} \right)^{-1} ,
  \end{align*}
  the left-hand side of (\ref{eq-change-ga-2}) is 
  \begin{align}
    \label{eq-LHS}
    &\Bigg(\prod_{(0,\ldots ,0) \prec (a_1,\ldots ,a_{p+q-1}) \prec (1\ldots ,1,0,\ldots ,0)}^{\prec} 
    (\ga_{1}^{a_1}\cdots \ga_{p+q-1}^{a_{p+q-1}})\ga_0 
    (\ga_{1}^{a_1}\cdots \ga_{p+q-1}^{a_{p+q-1}})^{-1} \Bigg) \\
    \nonumber
    &\cdot (\ga_{1}\cdots \ga_{p-1}) \ga_0 (\ga_{1}\cdots \ga_{p-1})^{-1}
    \cdot (\ga_{1}\cdots \ga_{p-1}\ga_{p+q}) \ga_0 (\ga_{1}\cdots \ga_{p-1}\ga_{p+q})^{-1} \\
    \nonumber
    &\cdot \Bigg(\prod_{(1\ldots ,1,0,\ldots ,0) \prec (a_1,\ldots ,a_{p+q-1}) \prec (0,\ldots ,0,1,\ldots ,1)}^{\prec} 
    (\ga_{1}^{a_1}\cdots \ga_{p+q-1}^{a_{p+q-1}})\ga_0 
    (\ga_{1}^{a_1}\cdots \ga_{p+q-1}^{a_{p+q-1}})^{-1} \Bigg) \\
    \nonumber
    &\cdot (\ga_{p}\cdots \ga_{p+q-1})\ga_0 (\ga_{p}\cdots \ga_{p+q-1})^{-1} .
  \end{align}
  By Claim \ref{claim-1}, $(\ga_{1}\cdots \ga_{p-1}\ga_{p+q}) \ga_0 (\ga_{1}\cdots \ga_{p-1}\ga_{p+q})^{-1}$
  commutes with the third line. 
  Then (\ref{eq-LHS}) is equal to 
  \begin{align*}
    &\Bigg(\prod_{(0,\ldots ,0) \prec (a_1,\ldots ,a_{p+q-1}) \prec (0,\ldots ,0,1,\ldots ,1)}^{\prec} 
    (\ga_{1}^{a_1}\cdots \ga_{p+q-1}^{a_{p+q-1}})\ga_0 
    (\ga_{1}^{a_1}\cdots \ga_{p+q-1}^{a_{p+q-1}})^{-1} \Bigg) \\
    &\cdot (\ga_{1}\cdots \ga_{p-1}\ga_{p+q}) \ga_0 (\ga_{1}\cdots \ga_{p-1}\ga_{p+q})^{-1} 
    \cdot (\ga_{p}\cdots \ga_{p+q-1})\ga_0 (\ga_{p}\cdots \ga_{p+q-1})^{-1} .
  \end{align*}
  Therefore, (\ref{eq-change-ga-2}) implies the commutativity (\ref{eq-simple-rel}).
\end{proof}
\section{Presentation of $\pi_1 (X^{(3)})$}
\label{section-3dim}
Hereafter, we mainly consider the case of $n=3$. 
In this section, we prove Theorem \ref{th-3dim}. 

To prove the theorem, 
we consider a plane cut of $X^{(3)}$. 
In the projective space $\P^3$, the defining equation of $S^{(3)}$ is
\begin{align*}
  &(\sqrt{x_0}-\sqrt{x_1}-\sqrt{x_2}-\sqrt{x_3})(\sqrt{x_0}+\sqrt{x_1}-\sqrt{x_2}-\sqrt{x_3})\\
  &\cdot (\sqrt{x_0}-\sqrt{x_1}+\sqrt{x_2}-\sqrt{x_3})(\sqrt{x_0}+\sqrt{x_1}+\sqrt{x_2}-\sqrt{x_3})\\
  &\cdot (\sqrt{x_0}-\sqrt{x_1}-\sqrt{x_2}+\sqrt{x_3})(\sqrt{x_0}+\sqrt{x_1}-\sqrt{x_2}+\sqrt{x_3})\\
  &\cdot (\sqrt{x_0}-\sqrt{x_1}+\sqrt{x_2}+\sqrt{x_3})(\sqrt{x_0}+\sqrt{x_1}+\sqrt{x_2}+\sqrt{x_3})\\
  &=\left( 2(x_0^2 +x_1^2 +x_2^2 +x_3^2) -(x_0 +x_1 +x_2 +x_3)^2\right)^2 -64x_0 x_1 x_2 x_3 .
\end{align*}
By \cite[Chapter XVII, \S 3, Ex.~11]{Hilton}, 
a plane cut (substituting $x_i$'s for linear forms)
of $S^{(3)}$ is a quartic with four bitangents 
$$
x_i=0 \quad (i=0,1,2,3),
$$
and with three nodes 
$$
x_i=x_j,\ x_k=x_l \quad (\{ i,j,k,l\} =\{ 0,1,2,3\}).
$$
We cut $S^{(3)}$ by $H\simeq\P^2$ with coordinates $(x,y,z)$ as 
\begin{align}
  \label{eq-plane-parameter}
  x_0=x-4z,\ x_1=-x-y,\ x_2=y-x,\ x_3=-x+z. 
\end{align}
Then, the defining equations of the components of $(\P^3-X^{(3)})\cap H$ are as follows:
\begin{align*}
  L_0&=(x_0=0)\cap H :x-4z=0, \\
  L_1&=(x_1=0)\cap H :-x-y=0, \\
  L_2&=(x_2=0)\cap H :y-x=0, \\
  L_3&=(x_3=0)\cap H :-x+z=0, \\
  Q&=S^{(3)}\cap H : 
  (4x^2+4y^2-32xz+25z^2)^2=64(y^2-x^2)(x-z)(x-4z).
\end{align*}
By using dehomogenized coordinate $(x,y)$ (put $z=1$), 
their expressions in $\C^2$ are given as 
\begin{align*}
  L_0&:x-4=0, \qquad
  L_1:x+y=0, \\
  L_2&:y-x=0, \qquad
  L_3:x-1=0, \\
  Q&:(4x^2+4y^2-32x+25)^2=64(y^2-x^2)(x-1)(x-4).
\end{align*}
Note that the line at infinity $(z=0)\subset H$ is not 
a component of $(\P^3-X^{(3)})\cap H$. 
By Zariski theorem of Lefschetz type (see, e.g., \cite[Chapter 4 (1.17)]{Dimca}), 
the inclusion $X^{(3)}\cap H \hookrightarrow X^{(3)}$
induces an isomorphism
\begin{align}
  \label{eq-isom-plane}
  \pi_1( X^{(3)}\cap H) \overset{\sim}{\longrightarrow} \pi_1(X^{(3)}) .
\end{align}

\subsection{Preliminary}
To compute $\pi_1( X^{(3)}\cap H)$, 
we consider $\{ \CL_{\la}: y=\la (x+1)\}_{\la\in \C} \subset H$ which is a pencil of lines 
through $(-1,0)\in \C^2$.  

We summarize some numerical data. See also Figure \ref{fig-sing3}. 
\begin{itemize}
\item $Q$ has three nodes $(\frac{5}{2},0)$, $(\frac{3}{2},\pm 1)$. 
\item $L_0$ is tangent to $Q$ at $(4, \pm \frac{\sqrt{39}}{2})$. 
\item $L_1$ is tangent to $Q$ at $(2+\frac{\sqrt{14}}{4},-2-\frac{\sqrt{14}}{4})$ 
  and $(2-\frac{\sqrt{14}}{4}, -2+\frac{\sqrt{14}}{4})$. 
\item $L_2$ is tangent to $Q$ at $(2+\frac{\sqrt{14}}{4}, 2+\frac{\sqrt{14}}{4})$ 
  and $(2-\frac{\sqrt{14}}{4}, 2-\frac{\sqrt{14}}{4})$. 
\item $L_3$ is tangent to $Q$ at $(1, \pm \frac{\sqrt{3}}{2})$. 
\item The intersection points of $\CL_0$ (: $y=0$) and $Q$ are 
  $(\frac{5}{2},0)$ (double root), $(\frac{11}{10}\pm \frac{\sqrt{-1}}{5},0)$. 
\item The line $\CL_{\la}$ is not generic for $X^{(3)}\cap H$ if and only if 
  $\la$ coincides with $0$, $\pm a_1$, $\ldots$, $\pm a_{10}$ or $a_{11}=\infty$ 
  which is given in Table \ref{table:ai}. 
  Note that each of $\pm a_1$, $\ldots$, $\pm a_{10}$ is a real number. 
\end{itemize}
\begin{figure}[h]
  \centering{
  \includegraphics[scale=0.8]{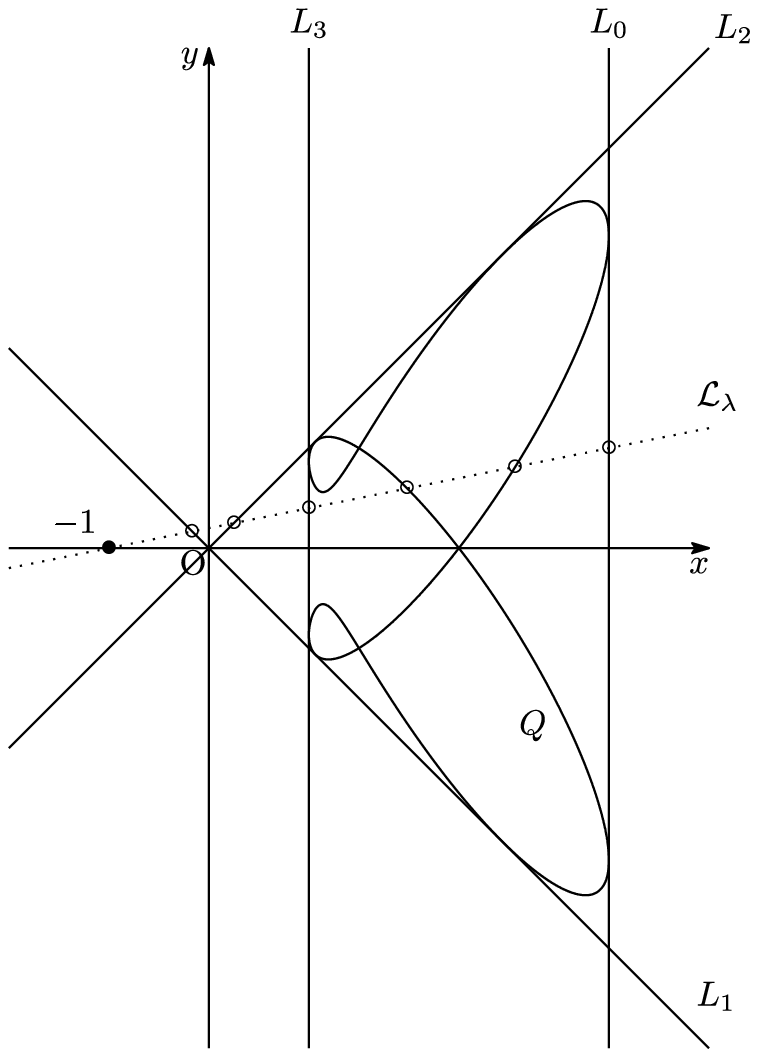} 
  }
  \caption{$X^{(3)}\cap H \subset \R^2$}
  \label{fig-sing3}
\end{figure}
\begin{table}
  \begin{tabular}{|l|l|}
    \hline
    &$\CL_{0}$ passes through \\ 
    & the node $(\frac{5}{2},0)\in Q$ 
    and the intersection point $(0,0)= L_1 \cap L_2$. \\ \hline
    $a_1 \fallingdotseq 0.2607431304$ 
    & $\CL_{a_1}$ is tangent to $Q$. \\ \hline
    $a_2 = 0.4 $ 
    & $\CL_{a_2}$ passes through the node $(\frac{3}{2},1)\in Q$.  \\ \hline  
    $a_3 \fallingdotseq 0.4330127020$ 
    & $\CL_{a_3}$ passes through the tangent point $(1, \frac{\sqrt{3}}{2})\in L_3 \cap Q$. \\ \hline
    $a_4 = 0.5$ 
    & $\CL_{a_4}$ passes through the intersection point $(1, 1)= L_2 \cap L_3$.\\ \hline  
    $a_5 \fallingdotseq 0.5156413111$ 
    & $\CL_{a_5}$ passes through the tangent point 
    $(2-\frac{\sqrt{14}}{4}, 2-\frac{\sqrt{14}}{4})\in L_2 \cap Q$. \\ \hline
    $a_6 \fallingdotseq 0.5196653275$ 
    & $\CL_{a_6}$ is tangent to $Q$. \\ \hline  
    $a_7 \fallingdotseq 0.6244997998$ 
    & $\CL_{a_7}$ passes through the tangent point $(4,\frac{\sqrt{39}}{2})\in L_0 \cap Q$. \\ \hline
    $a_8 \fallingdotseq 0.7458971504$ 
    & $\CL_{a_8}$ passes through the tangent point 
    $(2+\frac{\sqrt{14}}{4}, 2+\frac{\sqrt{14}}{4})\in L_2 \cap Q$. \\ \hline  
    $a_9 \fallingdotseq 0.7574500843$ 
    & $\CL_{a_9}$ is tangent to $Q$. \\ \hline
    $a_{10} = 0.8 $ 
    & $\CL_{a_{10}}$ passes through the intersection point $(4, 4)= L_0 \cap L_2$. \\ \hline  
    $a_{11} =\infty $ 
    & $\CL_{a_{11}}(:x=-1)$ passes through the intersection point $L_0 \cap L_3$.  \\ \hline
  \end{tabular}
  \caption{List of $a_i$'s}\label{table:ai}
\end{table}

\subsection{Computation of $\pi_1( X^{(3)}\cap H)$}
We compute $\pi_1( X^{(3)}\cap H)$ precisely. 
By the theorem of van Kampen-Zariski (see, e.g., \cite[Chapter 4 (3.15)]{Dimca}), 
all relations in $\pi_1( X^{(3)}\cap H) \simeq \pi_1(X^{(3)})$ are 
obtained from the monodromy relations around the 21 points 
$0$, $\pm a_1$, $\ldots$, $\pm a_{10}$ 
(note that the relation around $a_{11}=\infty$ follows from the others). 

Since $X^{(3)}\cap H$ is invariant under $[x:y:z]\mapsto [x:-y:z]$, 
the monodromy relations around $-a_1$, $\ldots$, $-a_{10}$ 
are obtained by a discussion parallel to those around $a_1$, $\ldots$, $a_{10}$. 

We fix a positive real number $a_0$ such that $0<a_0 <a_1$. 
First, we move $\la$ from $a_0$ to $a_{11}=\infty$. 
\begin{enumerate}[(1)]
\setcounter{enumi}{-1}
\item At $\la =a_0$, $\CL_{a_0}$ is a generic line for $X^{(3)}\cap H$.  
  We put $\{ h_1, \ldots ,h_8 \}=(\P^3 -X^{(3)})\cap H \cap \CL_{\la}$, which are indexed as follows.
  \begin{align*}
    \begin{array}{|l|c|c|c|c|c|}
      \hline 
      &h_1 & h_2 & h_3 & h_4,h_5,h_6,h_7 & h_8 \\ \hline
      \textrm{component} 
      &L_1 & L_2 & L_3 & Q & L_0 \\ \hline
    \end{array}
  \end{align*}
  Here, we suppose that $h_6 < h_7$ are real numbers, and $h_4$, $h_5$ are complex numbers 
  satisfying $\im (h_5) < 0< \im (h_4)$. 
  We take generators $\al_1, \ldots ,\al_8$ 
  of $\pi_1 (X^{(3)}\cap H \cap \CL_{a_0}) \simeq \pi_1 (\P^1 -\{ 8\textrm{ points}\})$ 
  as Figure \ref{fig-a0} (for simplicity, we consider $\sqrt{-1}\infty$ 
  as the base point in our pictures, 
  though we should take $(-1,0)\in \CL_{\la}$ as a base point). 
  Note that $\al_i$ is a loop going once around $h_i$ via the upper half-plane. 
  By the definition, we have a relation
  \begin{align*}
    \al_1 \cdots \al_8 =1.
  \end{align*}
  \begin{figure}[h]
    \centering{
      \includegraphics[scale=0.8]{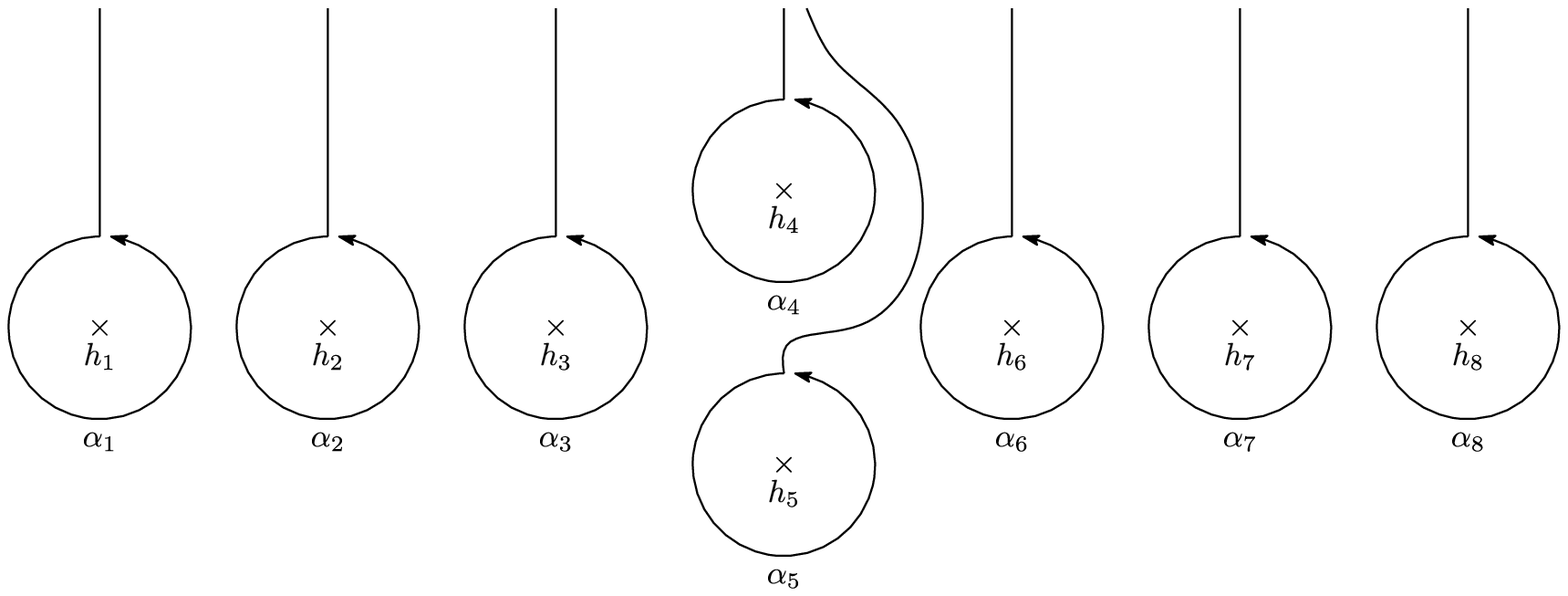} 
    }
    \caption{Loops in $\CL_{a_0}$}
    \label{fig-a0}
  \end{figure}
\item At $\la =a_1$, $\CL_{a_1}$ is tangent to $Q$. 
  If we move $\la$ around $a_1$, then $h_4$ and $h_5$ interchange
  counterclockwisely. 
  This implies a monodromy relation
  \begin{align*}
    \al_4 =\al_5 .
  \end{align*}
  By considering the half-turn of this move, 
  we obtain a picture of $X^{(3)}\cap H \cap \CL_{\la}$ with $a_1 < \la <a_2$; 
  see Figure \ref{fig-a12}. 
  \begin{figure}[h]
    \centering{
      \includegraphics[scale=0.8]{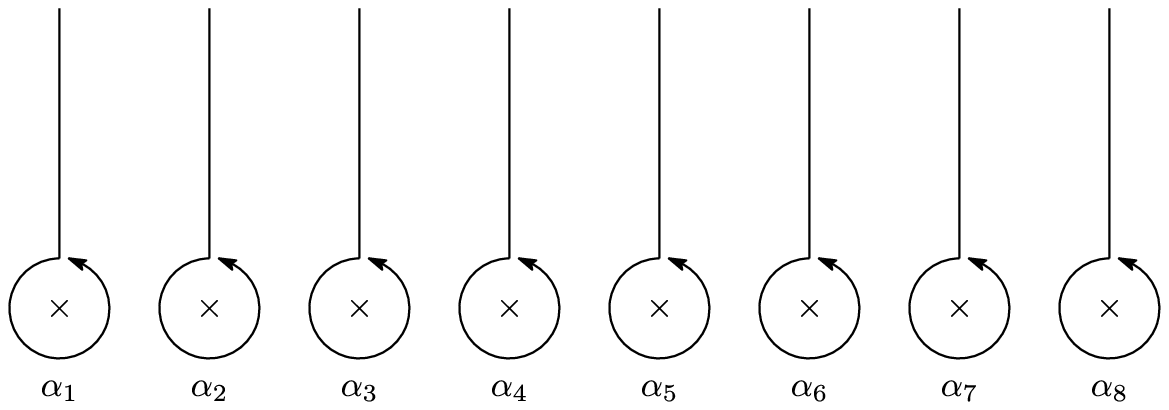} 
    }
    \caption{Loops in $\CL_{\la}$ with $a_1 < \la <a_2$}
    \label{fig-a12}
  \end{figure}
\item At $\la =a_2$, $\CL_{a_2}$ passes through the node $(\frac{3}{2},1)\in Q$. 
  When the line $\CL_{\la}$ approaches $\CL_{a_2}$, 
  the points $h_5$ and $h_6$ merge together 
  and we get a monodromy relation
  \begin{align*}
    [\al_5, \al_6 ]=1 ,\ \textrm{that is, } [\al_4, \al_6 ]=1.
  \end{align*}
  By considering the half-turn of this move, 
  we obtain a picture of $X^{(3)}\cap H \cap \CL_{\la}$ with $a_2 < \la <a_3$; 
  see Figure \ref{fig-a23}. 
  \begin{figure}[h]
    \centering{
      \includegraphics[scale=0.8]{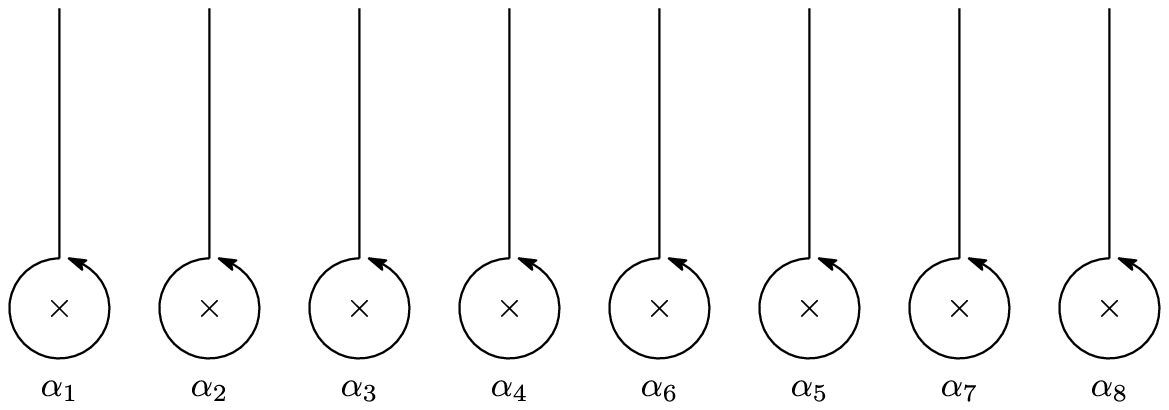} 
    }
    \caption{Loops in $\CL_{\la}$ with $a_2 < \la <a_3$}
    \label{fig-a23}
  \end{figure}
\item At $\la =a_3$, $\CL_{a_3}$ passes through the tangent point $(1, \frac{\sqrt{3}}{2})\in L_3 \cap Q$. 
  When the line $\CL_{\la}$ approaches $\CL_{a_3}$, 
  the points $h_3$ and $h_4$ merge together 
  and we get a monodromy relation
  (see also Figure \ref{fig-a3})
  \begin{align*}
    (\al_3 \al_4)^2 =(\al_4 \al_3)^2 .
  \end{align*}
  \begin{figure}[h]
    \centering{
      \includegraphics[scale=0.8]{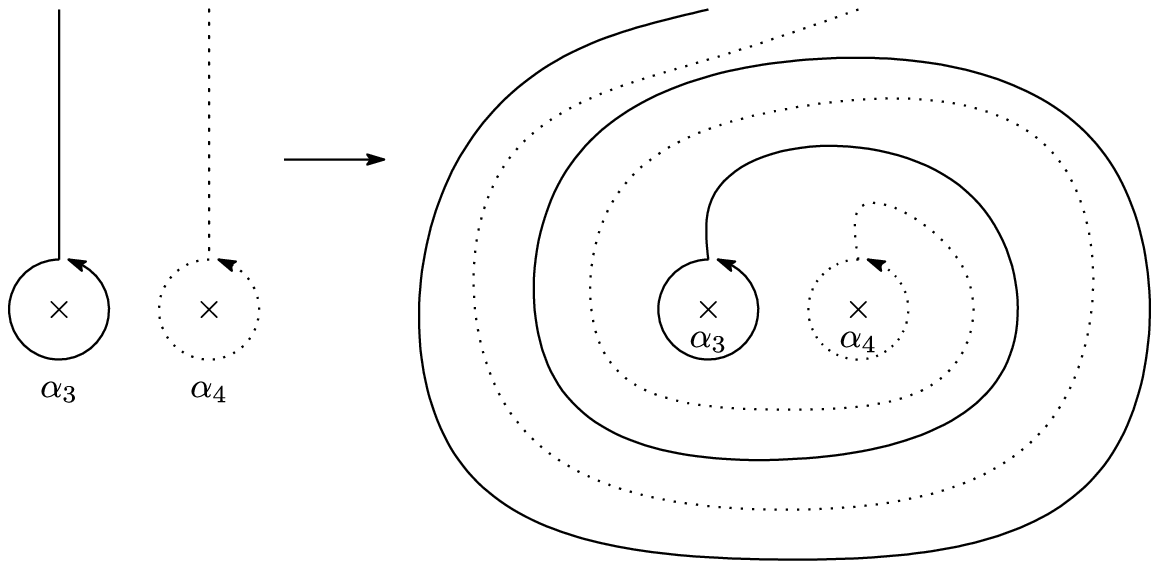} 
    }
    \caption{Loops in $\CL_{\la}$ obtained by moving $\la$ around $a_3$}
    \label{fig-a3}
  \end{figure}

  By considering the half-turn of this move, 
  we obtain a picture of $X^{(3)}\cap H \cap \CL_{\la}$ with $a_3 < \la <a_4$; 
  see Figure \ref{fig-a34}. 
  We retake loops around $h_3$ and $h_4$ by
  $\tal_3$ and $\tal_4$ in Figure \ref{fig-a34}, respectively. Note that 
  \begin{align*}
    \tal_3 &=\al_3 \al_4 \al_3 (\al_3 \al_4)^{-1}
    =\al_3 \al_4 \al_3 \al_4^{-1} \al_3^{-1}=\al_4^{-1} \al_3 \al_4 ,\\ 
    \tal_4 &=\al_3 \al_4 \al_4 (\al_3 \al_4)^{-1}=\al_3 \al_4 \al_3^{-1} .
  \end{align*}
  \begin{figure}[h]
    \centering{
      \includegraphics[scale=0.8]{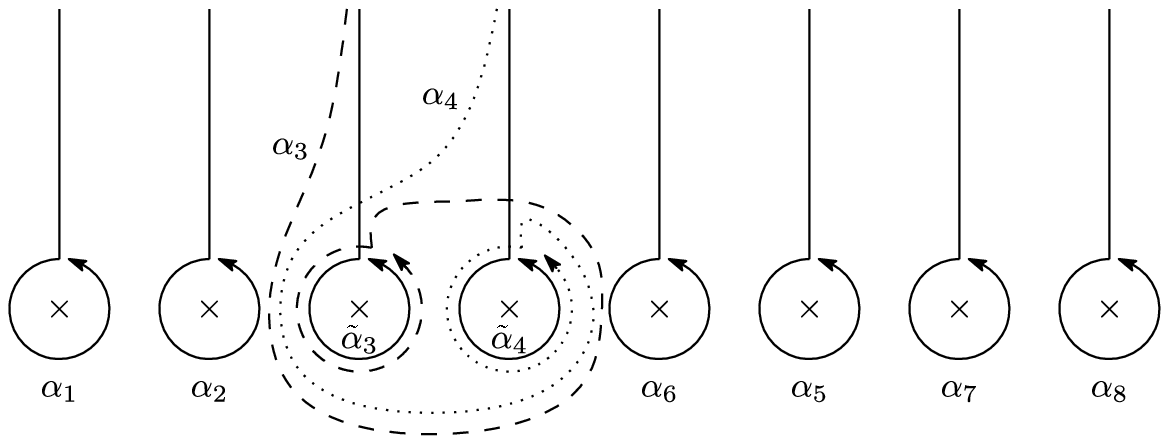} 
    }
    \caption{Loops in $\CL_{\la}$ with $a_3 < \la <a_4$}
    \label{fig-a34}
  \end{figure}
\item At $\la =a_4$, $\CL_{a_4}$ passes through the intersection point $(1, 1)= L_2 \cap L_3$.
  When the line $\CL_{\la}$ approaches $\CL_{a_4}$, 
  the points $h_2$ and $h_3$ merge together 
  and we get a monodromy relation
  \begin{align*}
    [\al_2, \tal_3 ]=1 ,\ \textrm{that is, } [\al_2, \al_4^{-1} \al_3 \al_4 ]=1.
  \end{align*}
  By considering the half-turn of this move, 
  we obtain a picture of $X^{(3)}\cap H \cap \CL_{\la}$ with $a_4 < \la <a_5$; 
  see Figure \ref{fig-a45}. 
  \begin{figure}[h]
    \centering{
      \includegraphics[scale=0.8]{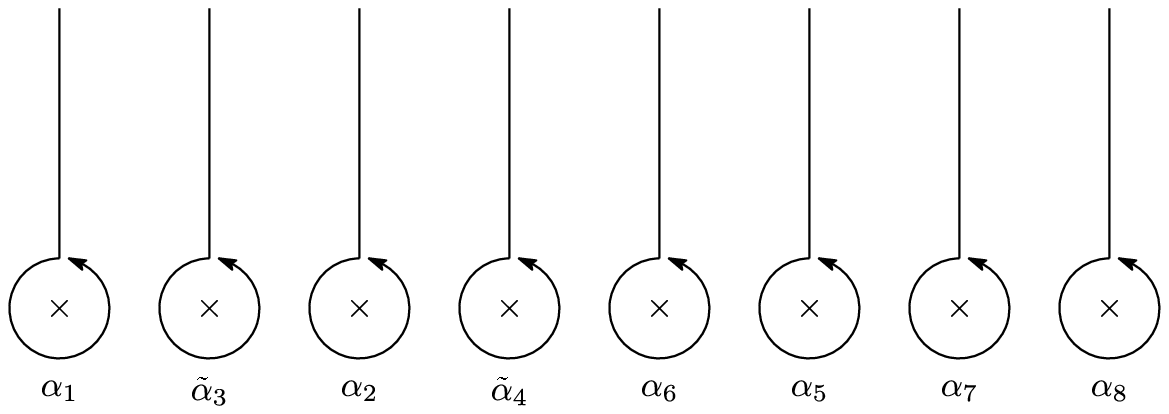} 
    }
    \caption{Loops in $\CL_{\la}$ with $a_4 < \la <a_5$}
    \label{fig-a45}
  \end{figure}
\item At $\la =a_5$, $\CL_{a_5}$ passes through the tangent point 
  $(2-\frac{\sqrt{14}}{4}, 2-\frac{\sqrt{14}}{4})\in L_2 \cap Q$. 
  When the line $\CL_{\la}$ approaches $\CL_{a_5}$, 
  the points $h_2$ and $h_4$ merge together 
  and we get a monodromy relation
  \begin{align*}
    (\al_2 \tal_4)^2 =(\tal_4 \al_2)^2 ,\ \textrm{that is, } 
    (\al_2 \al_3 \al_4 \al_3^{-1})^2 =(\al_3 \al_4 \al_3^{-1} \al_2)^2.
  \end{align*}
  By considering the half-turn of this move, 
  we obtain a picture of $X^{(3)}\cap H \cap \CL_{\la}$ with $a_5 < \la <a_6$; 
  see Figure \ref{fig-a56}. 
  We retake loops around $h_2$ and $h_4$ by
  $\tal_2$ and $\ttal_4$ in Figure \ref{fig-a56}, respectively. Note that 
  \begin{align*}
    \tal_2 &=\tal_4^{-1} \al_2 \tal_4 
    =(\al_3 \al_4 \al_3^{-1})^{-1} \al_2 (\al_3 \al_4 \al_3^{-1})
    =\al_3 \al_4^{-1} \al_3^{-1} \al_2 \al_3 \al_4 \al_3^{-1}\\
    &=\al_4^{-1} \al_3^{-1} \al_3 \al_4 \al_3 \al_4^{-1} \al_3^{-1} \al_2 \al_3 \al_4 \al_3^{-1}
    =(\al_3 \al_4)^{-1} \al_2 (\al_3 \al_4) ,\\ 
    \ttal_4 &=\al_2 \tal_4 \al_2^{-1} =(\al_2 \al_3)\al_4 (\al_2 \al_3)^{-1}.
  \end{align*}
  Here, we use the relations obtained in (3) and (4). 
  \begin{figure}[h]
    \centering{
      \includegraphics[scale=0.8]{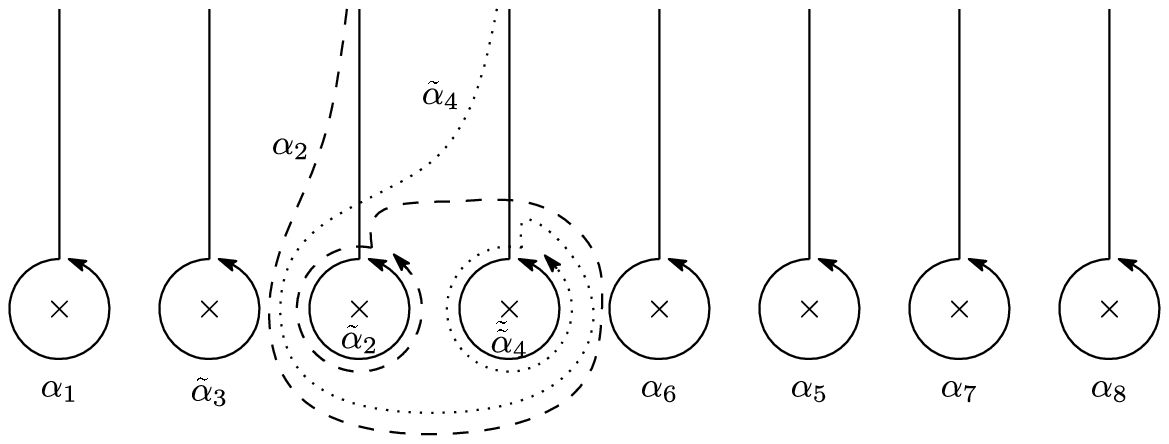} 
    }
    \caption{Loops in $\CL_{\la}$ with $a_5 < \la <a_6$}
    \label{fig-a56}
  \end{figure}
\item At $\la =a_6$, $\CL_{a_6}$ is tangent to $Q$. 
  If we move $\la$ around $a_6$, then $h_4$ and $h_6$ interchange 
  counterclockwisely. 
  This implies a monodromy relation 
  \begin{align*}
    \al_6 =\ttal_4 ,\ \textrm{that is, } 
    \al_6 =(\al_2 \al_3)\al_4 (\al_2 \al_3)^{-1}.
  \end{align*}
  By considering the half-turn of this move, 
  we obtain a picture of $X^{(3)}\cap H \cap \CL_{\la}$ with $a_6 < \la <a_7$; 
  see Figure \ref{fig-a67}. 
  Note that in $\CL_{\la}$ with $\la >a_6$, two points $h_4$ and $h_6$ are
  not in the real axis, and satisfy $\im (h_4)<0<\im (h_6)$.
  \begin{figure}[h]
    \centering{
      \includegraphics[scale=0.8]{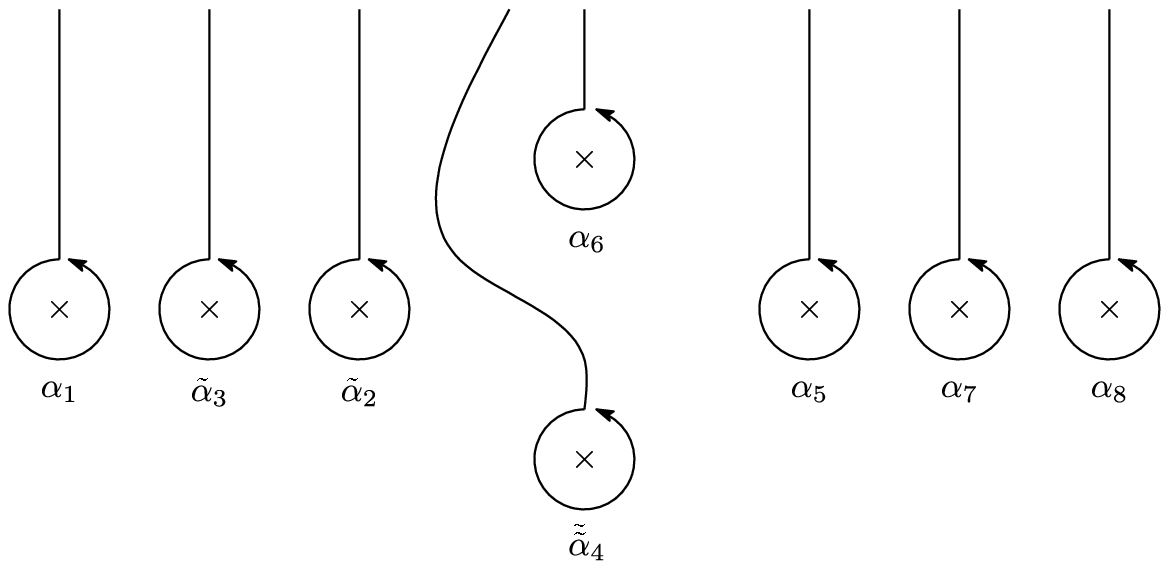} 
    }
    \caption{Loops in $\CL_{\la}$ with $a_6 < \la <a_7$}
    \label{fig-a67}
  \end{figure}
\item At $\la =a_7$, $\CL_{a_7}$ passes through the tangent point $(4,\frac{\sqrt{39}}{2})\in L_0 \cap Q$. 
  When the line $\CL_{\la}$ approaches $\CL_{a_7}$, 
  the points $h_7$ and $h_8$ merge together 
  and we get a monodromy relation
  \begin{align*}
    (\al_7 \al_8)^2 =(\al_7 \al_8)^2.
  \end{align*}
  By considering the half-turn of this move, 
  we obtain a picture of $X^{(3)}\cap H \cap \CL_{\la}$ with $a_7 < \la <a_8$; 
  see Figure \ref{fig-a78}. 
  We retake loops around $h_7$ and $h_8$ by
  $\tal_7$ and $\tal_8$ in Figure \ref{fig-a78}, respectively. Note that 
  \begin{align*}
    \tal_7 =\al_8^{-1} \al_7 \al_8 ,\quad 
    \tal_8 =\al_7 \al_8 \al_7^{-1}.
  \end{align*}
  \begin{figure}[h]
    \centering{
      \includegraphics[scale=0.8]{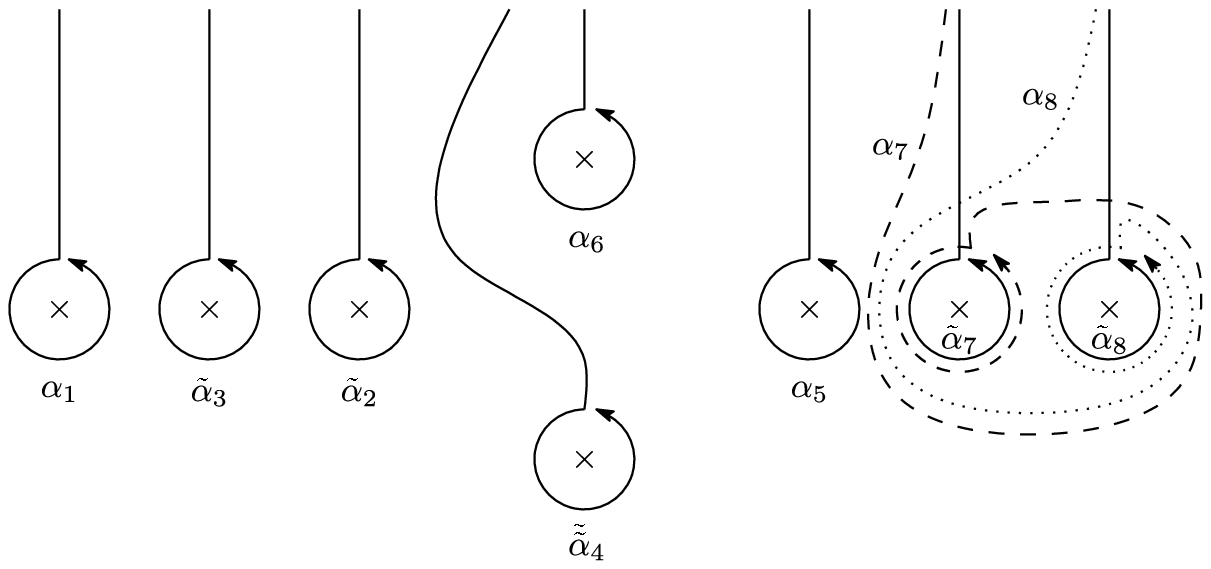} 
    }
    \caption{Loops in $\CL_{\la}$ with $a_7 < \la <a_8$}
    \label{fig-a78}
  \end{figure}
\item At $\la =a_8$, $\CL_{a_8}$ passes through the tangent point 
  $(2+\frac{\sqrt{14}}{4}, 2+\frac{\sqrt{14}}{4})\in L_2 \cap Q$. 
  When the line $\CL_{\la}$ approaches $\CL_{a_8}$, 
  the points $h_2$ and $h_5$ merge together. 
  To write down a monodromy relation, we retake loops around $h_2$ and $h_4$ by
  $\al_6^{-1} \tal_2 \al_6$ and $\tal_2 \ttal_4 \tal_2^{-1}$, respectively 
  (see Figure \ref{fig-a8}). 
  \begin{figure}[h]
    \centering{
      \includegraphics[scale=0.8]{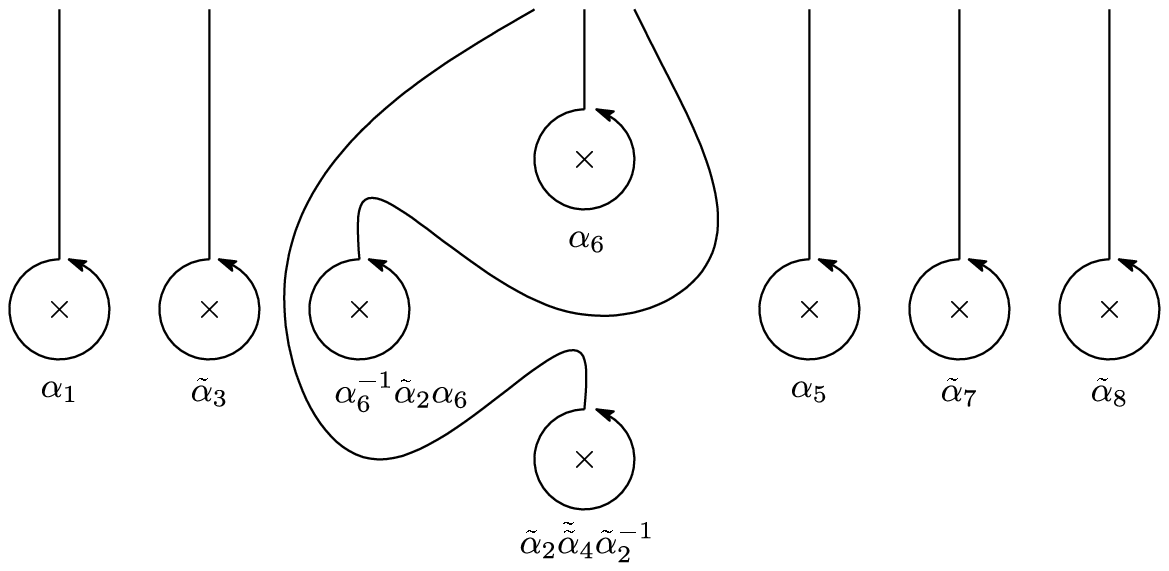} 
    }
    \caption{Retaking loops at $\CL_{\la}$ with $a_7 < \la <a_8$.}
    \label{fig-a8}
  \end{figure}

  By using these generators, we obtain a monodromy relation 
  \begin{align*}
    (\al_6^{-1} \tal_2 \al_6 \al_5)^2 =(\al_5 \al_6^{-1} \tal_2 \al_6)^2.
  \end{align*}
  Because of $[\al_5 ,\al_6]=1$, we can reduce this relation to
  \begin{align*}
    (\tal_2 \al_5)^2 =(\al_5 \tal_2 )^2 \textrm{, or equivalently, } \ 
    (\tal_2 \al_4)^2 =(\al_4 \tal_2 )^2 .
  \end{align*}
  By considering the half-turn of this move, 
  we obtain a picture of $X^{(3)}\cap H \cap \CL_{\la}$ with $a_8 < \la <a_9$; 
  see Figure \ref{fig-a89}. 
  We retake loops around $h_2$ and $h_5$ by
  $\ttal_2$ and $\tal_5$ in Figure \ref{fig-a89}, respectively. Note that 
  \begin{align*}
    \ttal_2 &=\al_5^{-1}( \al_6^{-1} \tal_2 \al_6) \al_5 
    =\al_5^{-1} \al_6^{-1} (\al_3 \al_4)^{-1} \al_2 (\al_3 \al_4) \al_6 \al_5 \\
    &=(\al_3 \al_4 \al_5 \al_6)^{-1} \al_2 (\al_3 \al_4 \al_5 \al_6)  ,\\
    \tal_5 &=\al_6^{-1} \tal_2 \al_6 \al_5 (\al_6^{-1} \tal_2 \al_6)^{-1}
    =\al_6^{-1} \tal_2 \al_5  \tal_2^{-1} \al_6 \\
    &=\al_6^{-1} (\al_3 \al_4)^{-1} \al_2 (\al_3 \al_4) \al_4 
    (\al_3 \al_4)^{-1} \al_2^{-1} (\al_3 \al_4) \al_6 \\
    &=\al_6^{-1} \al_4^{-1} \al_3^{-1} \al_6 \al_3 \al_4 \al_6 . \\
    &=\al_6^{-1} \al_4^{-1} \al_3^{-1} \al_4 \al_6 
    \al_4^{-1} \al_3 \al_4 \al_6  \\
    &=\al_6^{-1}  \al_2 \al_4^{-1} \al_3^{-1} \al_4 \al_2^{-1} \al_6 \al_2
    \al_4^{-1} \al_3 \al_4 \al_2^{-1} \al_6 \\
    &=\al_2 \al_3 \al_4^{-1} \al_3^{-1} \cdot \al_4^{-1} \al_3^{-1}  \al_4 
    \cdot \al_3 \al_4 \al_3^{-1} \cdot 
    \al_4^{-1} \al_3 \al_4 \cdot \al_3 \al_4 \al_3^{-1} \al_2^{-1}  \\
    &=\al_2 \al_4 \al_2^{-1}.
  \end{align*}
  Here, we use $[\al_5 ,\al_6]=1$, $\al_5=\al_4$, 
  $(\al_3 \al_4)^2 =(\al_4 \al_3)^2$, $[\al_2, \al_4^{-1} \al_3 \al_4 ]=1$, 
  $[\al_4 ,\al_6]=1$ and $\al_2^{-1} \al_6 \al_2 =\al_3 \al_4 \al_3^{-1}$.
  \begin{figure}[h]
    \centering{
      \includegraphics[scale=0.8]{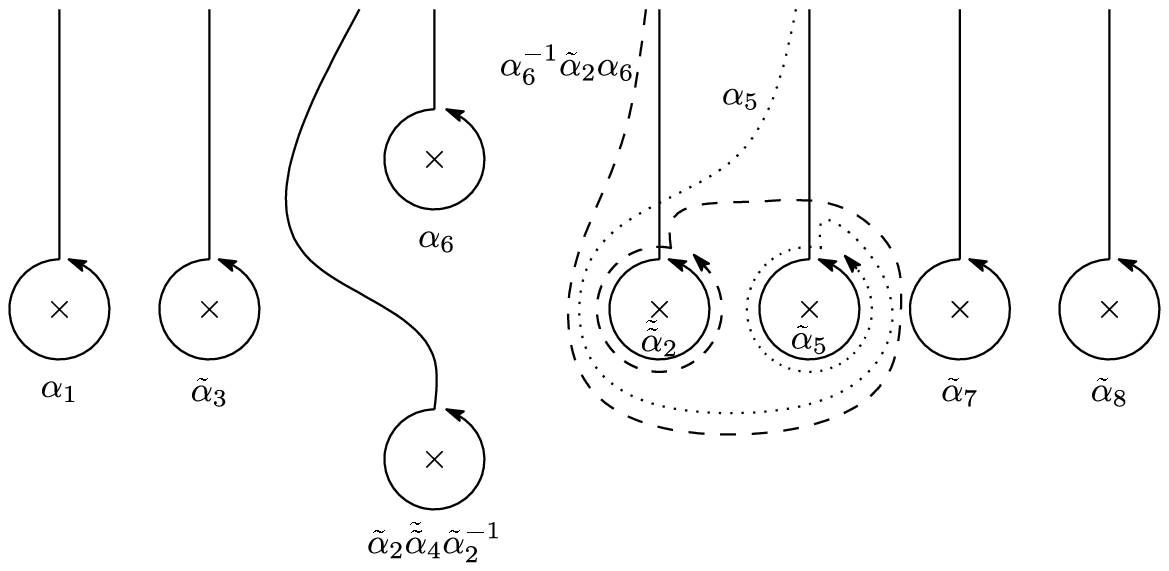} 
    }
    \caption{Loops in $\CL_{\la}$ with $a_8 < \la <a_9$}
    \label{fig-a89}
  \end{figure}
\item At $\la =a_9$, $\CL_{a_9}$ is tangent to $Q$. 
  If we move $\la$ around $a_9$, then $h_5$ and $h_7$ interchange
  counterclockwisely. 
  This implies a monodromy relation 
  \begin{align*}
    \tal_5 =\tal_7 ,\ \textrm{that is, } 
    \al_2 \al_4 \al_2^{-1}=\al_8^{-1} \al_7 \al_8.
  \end{align*}
  By considering the half-turn of this move, 
  we obtain a picture of $X^{(3)}\cap H \cap \CL_{\la}$ with $a_9 < \la <a_{10}$; 
  see Figure \ref{fig-a9A}. 
  Note that in $\CL_{\la}$ with $\la >a_9$, two points $h_5$ and $h_7$ are
  not in the real axis, and satisfy $\im (h_5)<0<\im (h_7)$.
  \begin{figure}[h]
    \centering{
      \includegraphics[scale=0.8]{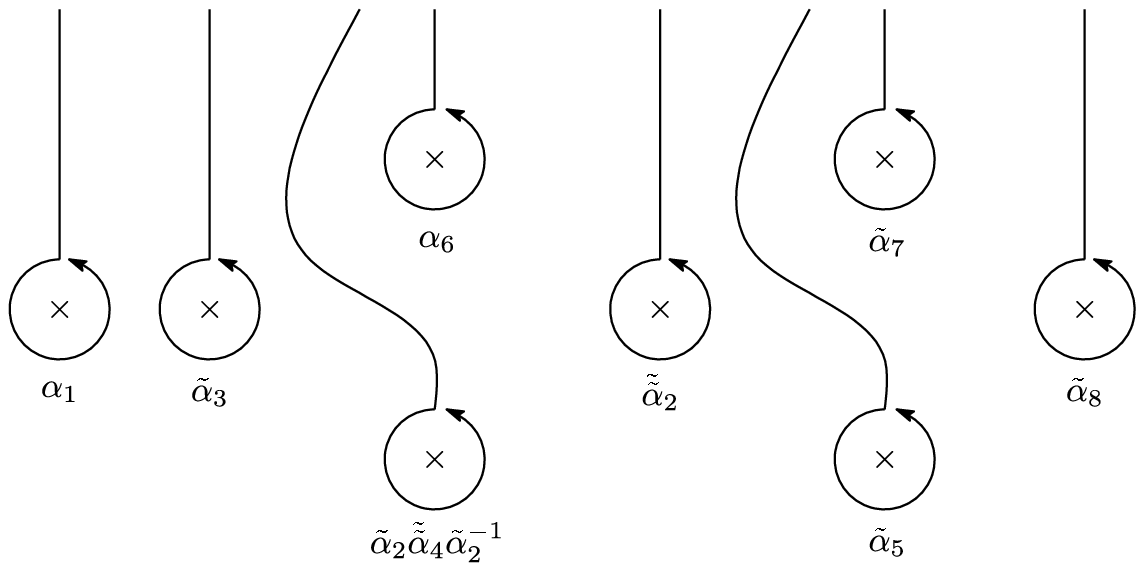} 
    }
    \caption{Loops in $\CL_{\la}$ with $a_9 < \la <a_{10}$}
    \label{fig-a9A}
  \end{figure}
\item At $\la =a_{10}$, $\CL_{a_{10}}$ passes through the intersection point $(4, 4)= L_0 \cap L_2$. 
  When the line $\CL_{\la}$ approaches $\CL_{a_{10}}$, 
  the points $h_2$ and $h_8$ merge together. 
  To write down a monodromy relation, we retake loops around $h_2$ and $h_5$ by
  $\tal_7^{-1} \ttal_2 \tal_7$ and $\ttal_2 \tal_5 \ttal_2^{-1}$, respectively 
  (see Figure \ref{fig-aA}). 
  \begin{figure}[h]
    \centering{
      \includegraphics[scale=0.8]{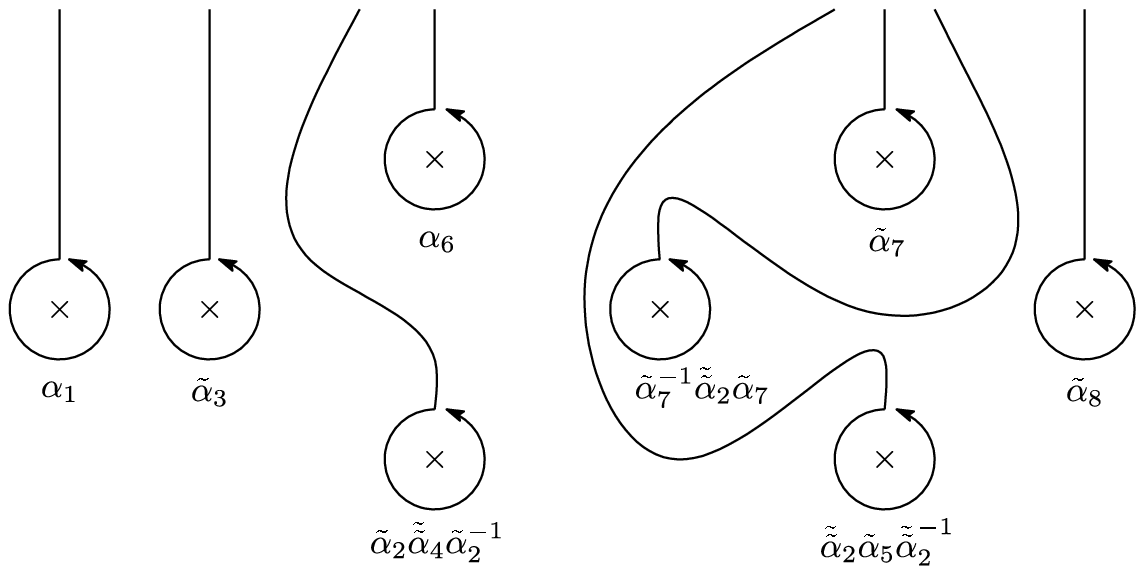} 
    }
    \caption{Retaking loops at $\CL_{\la}$ with $a_9 < \la <a_{10}$.}
    \label{fig-aA}
  \end{figure}

  By using these generators, we obtain a monodromy relation 
  \begin{align*}
    [\tal_7^{-1} \ttal_2 \tal_7, \tal_8 ]=1.
  \end{align*}
  By considering the half-turn of this move, 
  we obtain a picture of $X^{(3)}\cap H \cap \CL_{\la}$ with $a_{10} < \la <a_{11}$; 
  see Figure \ref{fig-aAB}. 
  \begin{figure}[h]
    \centering{
      \includegraphics[scale=0.8]{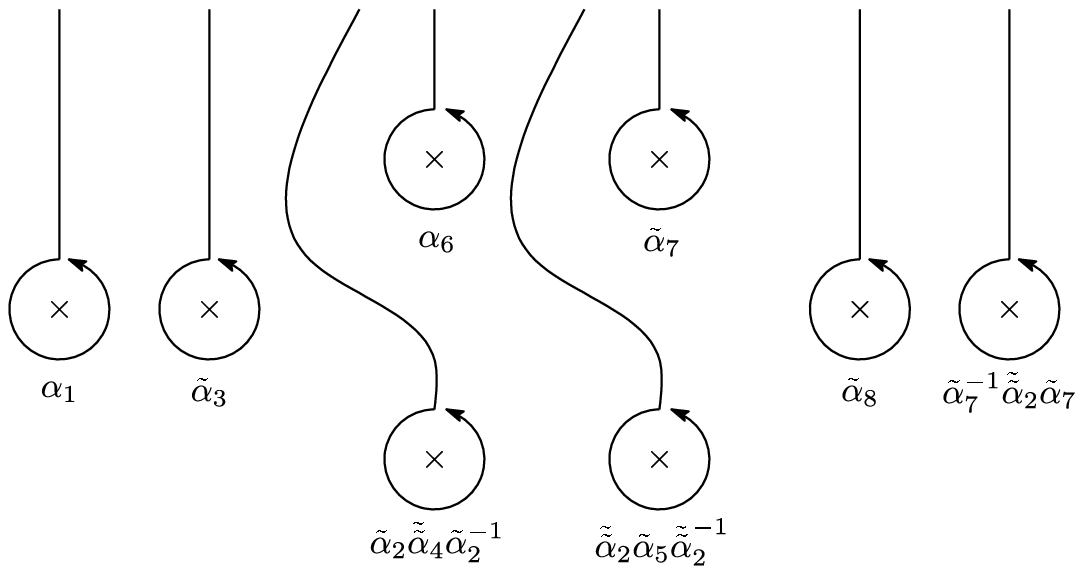} 
    }
    \caption{Loops in $\CL_{\la}$ with $a_{10} < \la <a_{11}$}
    \label{fig-aAB}
  \end{figure}
\item At $\la =a_{11}$, $\CL_{a_{11}}$ passes through the intersection point $L_0 \cap L_3$. 
  When the line $\CL_{\la}$ approaches $\CL_{a_{11}}$, 
  the points $h_3$ and $h_8$ merge together. 
  To write down a monodromy relation, 
  we redraw a picture of $\CL_{a_{11}} \simeq \P^1$ so that 
  $h_8$ is leftmost, and 
  we retake a loop around $h_8$ by
  $(\tal_7^{-1} \ttal_2 \tal_7 \al_1)^{-1} \tal_8 (\tal_7^{-1} \ttal_2 \tal_7 \al_1)$ 
  (see Figure \ref{fig-aB}). 
  \begin{figure}[h]
    \centering{
      \includegraphics[scale=0.8]{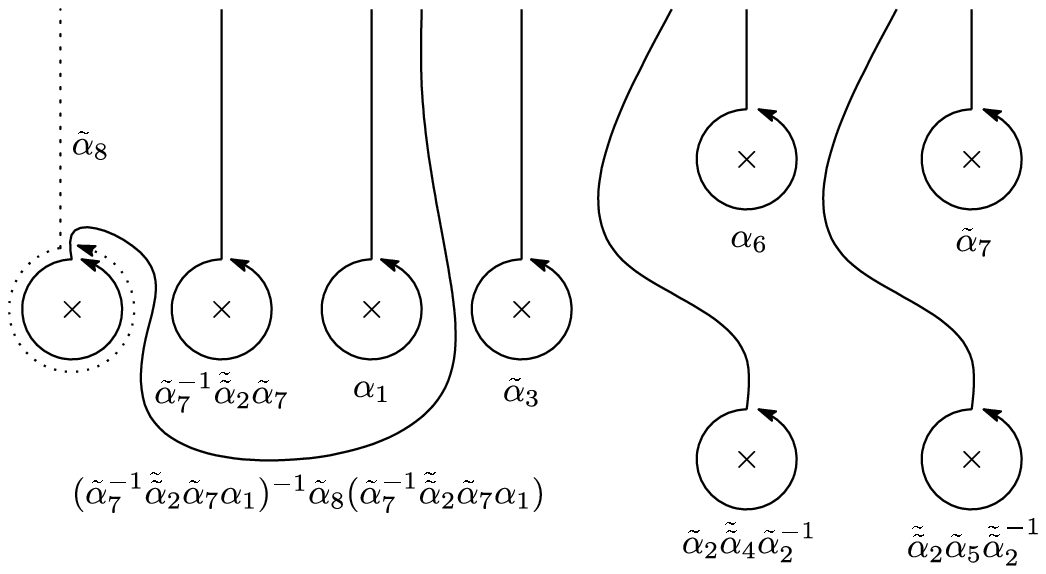} 
    }
    \caption{Retaking loops at $\CL_{\la}$ with $a_{10} < \la <a_{11}$.}
    \label{fig-aB}
  \end{figure}

  By using these generators, we obtain a monodromy relation 
  \begin{align*}
    [(\tal_7^{-1} \ttal_2 \tal_7 \al_1)^{-1} \tal_8 (\tal_7^{-1} \ttal_2 \tal_7 \al_1), 
    \tal_3 ]=1.
  \end{align*}
\end{enumerate}
Therefore, we obtain the all relations for $\la >0$. 
We list the relations obtained in $\la >0$:
\begin{enumerate}[(1)]
\setcounter{enumi}{-1}
\item\label{rel-0} $\al_1 \al_2 \dots \al_8 =1$; 
\item\label{rel-1} $\al_4 =\al_5$; 
\item\label{rel-2} $[\al_4, \al_6]=1$; 
\item\label{rel-3} $(\al_3 \al_4)^2 =(\al_4 \al_3)^2$;
\item\label{rel-4} $[\al_2, \al_4^{-1} \al_3 \al_4]=1$; 
\item\label{rel-5} $(\al_2 \al_3 \al_4 \al_3^{-1})^2 =(\al_3 \al_4 \al_3^{-1} \al_2)^2$;
\item\label{rel-6} $\al_6 =\al_2 \al_3 \al_4 (\al_2 \al_3)^{-1}$; 
\item\label{rel-7} $(\al_7 \al_8)^2 =(\al_8 \al_7)^2$;
\item\label{rel-8} $(\tal_2 \al_5)^2 =(\al_5 \tal_2)^2$; 
\item\label{rel-9} $\al_2 \al_4 \al_2^{-1}=\al_8^{-1} \al_7 \al_8$; 
\item\label{rel-10} $[\tal_7^{-1} \ttal_2 \tal_7 ,\tal_8]=1$; 
\item\label{rel-11} $[(\tal_7^{-1} \ttal_2 \tal_7 \al_1)^{-1} \tal_8 (\tal_7^{-1} \ttal_2 \tal_7 \al_1),\tal_3]=1$,  
\end{enumerate}
where
\begin{align*}
  \tal_2 &=(\al_3 \al_4)^{-1} \al_2 (\al_3 \al_4) ,\quad  
  \ttal_2 =(\al_3 \al_4 \al_5 \al_6)^{-1} \al_2 (\al_3 \al_4 \al_5 \al_6)  , \\
  \tal_3 &=\al_4^{-1} \al_3 \al_4 ,\quad  
  \tal_7 =\al_8^{-1} \al_7 \al_8 ,\quad  
  \tal_8 =\al_7 \al_8 \al_7^{-1}.
\end{align*}
Note that the relation (\ref{rel-11}) obtained as the monodromy around $a_{11}=\infty$ is not needed
(see also Remark \ref{rem:a11}).

Next, we move $\la$ from a positive number to a negative one around $\la=0$. 
Then we have two interchanges $\al_1 \leftrightarrow \al_2$ and 
$\al_6 \leftrightarrow \al_7$, and obtain the monodromy relations (\ref{rel-0'}') below. 
We obtain the monodromy relations around $-a_1$, $\ldots$, $-a_{10}$ as follows
(recall that (\ref{rel-1}) $\al_4 =\al_5$): 
\begin{enumerate}[(1')]
\setcounter{enumi}{-1}
\item\label{rel-0'} $[\al_1, \al_2]=1$, $[\al_6, \al_7]=1$; 
\item same as (1); 
\item\label{rel-2'} $[\al_4, \al_7]=1$; 
\item same as (3);
\item\label{rel-4'} $[\al_1, \al_4^{-1} \al_3 \al_4]=1$; 
\item\label{rel-5'} $(\al_1 \al_3 \al_4 \al_3^{-1})^2 =(\al_3 \al_4 \al_3^{-1} \al_1)^2$;
\item\label{rel-6'} $\al_7 =\al_1 \al_3 \al_4 (\al_1 \al_3)^{-1}$; 
\item\label{rel-7'} $(\al_6 \al_8)^2 =(\al_8 \al_6)^2$;
\item\label{rel-8'} $(\tal'_1 \al_5)^2 =(\al_5 \tal'_1)^2$;
\item\label{rel-9'} $\al_1 \al_4 \al_1^{-1}=\al_8^{-1} \al_6 \al_8$; 
\item\label{rel-10'} $[\talp_6{}^{-1} \ttalp_1 \talp_6,\talp_8]=1$; 
\item\label{rel-11'} $[(\talp_6{}^{-1} \ttalp_1 \talp_6 \al_2)^{-1} \tal_8
  (\talp_6{}^{-1} \ttalp_1 \talp_6 \al_2),\talp_3]=1$;   
\end{enumerate}
here 
\begin{align*}
  \talp_1 &=(\al_3 \al_4)^{-1} \al_1 (\al_3 \al_4) ,\quad  
  \ttalp_1 =(\al_3 \al_4 \al_5 \al_7)^{-1} \al_1 (\al_3 \al_4 \al_5 \al_7)  , \\
  \talp_3 &=\al_4^{-1} \al_3 \al_4 ,\quad  
  \talp_7 =\al_8^{-1} \al_6 \al_8 ,\quad  
  \talp_8 =\al_6 \al_8 \al_6^{-1}.
\end{align*}

By using the relations (\ref{rel-0}), (\ref{rel-1}), (\ref{rel-6}) and (\ref{rel-6'}'), 
we have 
\begin{align*}
  \pi_1(X^{(3)}\cap H)=\left\langle \al_1,\al_2,\al_3, \al_4  \left| 
      \begin{array}{l}
        \textrm{(\ref{rel-2}), (\ref{rel-3}), (\ref{rel-4}), (\ref{rel-5}),
          (\ref{rel-0'}'), (\ref{rel-2'}'), (\ref{rel-4'}'), (\ref{rel-5'}')} \\
        \textrm{(\ref{rel-7}), (\ref{rel-8}), (\ref{rel-9}), (\ref{rel-10}),
          (\ref{rel-7'}'), (\ref{rel-8'}'), (\ref{rel-9'}'), (\ref{rel-10'}')}         
      \end{array}
    \right. \right\rangle .
\end{align*}
We put 
\begin{align}
  \label{eq-beta-alpha}
  \be_1=\al_1 ,\quad 
  \be_2=\al_2 ,\quad 
  \be_3=\tal_3 =\al_4^{-1}\al_3 \al_4 ,\quad 
  \be_4=\tal_4 =\al_3 \al_4 \al_3^{-1}.
\end{align}
By the relation (\ref{rel-3}), $\al_i$'s are written as 
\begin{align}
  \label{eq-alpha-beta}
  \al_1=\be_1 ,\quad 
  \al_2=\be_2 ,\quad 
  \al_3=\be_4^{-1}\be_3 \be_4 ,\quad 
  \al_4=\be_3 \be_4 \be_3^{-1}.
\end{align}
Thus, $\be_1$, $\be_2$, $\be_3$ and $\be_4$ form a generator of $\pi_1(X^{(3)}\cap H)$:  
\begin{align*}
  \pi_1(X^{(3)}\cap H)=\left\langle \be_1,\be_2,\be_3, \be_4  \left| 
      \begin{array}{l}
        \textrm{(\ref{rel-2}), (\ref{rel-3}), (\ref{rel-4}), (\ref{rel-5}),
          (\ref{rel-0'}'), (\ref{rel-2'}'), (\ref{rel-4'}'), (\ref{rel-5'}')} \\
        \textrm{(\ref{rel-7}), (\ref{rel-8}), (\ref{rel-9}), (\ref{rel-10}),
          (\ref{rel-7'}'), (\ref{rel-8'}'), (\ref{rel-9'}'), (\ref{rel-10'}')}         
      \end{array}
    \right. \right\rangle .
\end{align*}

\begin{Lem}
  \label{lem-ABC}
  The relations 
  (\ref{rel-2}), (\ref{rel-3}), (\ref{rel-4}), (\ref{rel-5}), 
  (\ref{rel-0'}'), (\ref{rel-2'}'), (\ref{rel-4'}'), (\ref{rel-5'}')
  are equivalent to 
  \begin{enumerate}[{\rm (A)}]
  \item $[\be_i,\be_j]=1 \quad (1\leq i<j \leq 3)$; 
  \item $[\be_i \be_4 \be_i^{-1}, \be_j \be_4 \be_j^{-1}]=1 \quad (1\leq i<j \leq 3)$;
  \item $(\be_4\be_k)^2=(\be_k\be_4)^2 \quad (1\leq k \leq 3)$. 
  \end{enumerate}
\end{Lem}
\begin{proof}
  Note that 
  \begin{align*}
    \be_2 \be_4 \be_2^{-1}
    =\al_2 \al_3 \al_4 \al_3^{-1} \al_2^{-1}
    =\al_6 ,\quad 
    \be_1 \be_4 \be_1^{-1}
    =\al_1 \al_3 \al_4 \al_3^{-1} \al_1^{-1}
    =\al_7 .
  \end{align*}
  The lemma is proved by straightforward calculations. 
\end{proof}
By this lemma we obtain 
\begin{align*}
  \pi_1(X^{(3)}\cap H)=\left\langle \be_1,\be_2,\be_3, \be_4  \left| 
      \begin{array}{l}
        \textrm{(A), (B), (C)} \\
        \textrm{(\ref{rel-7}), (\ref{rel-8}), (\ref{rel-9}), (\ref{rel-10}),
          (\ref{rel-7'}'), (\ref{rel-8'}'), (\ref{rel-9'}'), (\ref{rel-10'}')}         
      \end{array}
    \right. \right\rangle .
\end{align*}
Note that by $(\be_3 \be_4)^2 =(\be_4 \be_3)^2$, we have 
$$
\al_3 \al_4 = (\be_4^{-1}\be_3 \be_4) (\be_3 \be_4 \be_3^{-1}) =\be_3 \be_4 . 
$$
Recall that
\begin{align*}
  &\al_1=\be_1 ,\quad 
  \al_2=\be_2 ,\quad 
  \al_3=\be_4^{-1}\be_3 \be_4 ,\quad 
  \al_4=\al_5=\be_3 \be_4 \be_3^{-1}, \\
  &\al_6 =\be_2 \be_4 \be_2^{-1} ,\quad 
  \al_7 =\be_1 \be_4 \be_1^{-1}, 
\end{align*}
and 
\begin{align*}
  \al_8^{-1} &=\al_1 \al_2 \al_3 \al_4 \al_5 \al_6 \al_7 
  =\be_1 \be_2 \cdot \be_3 \be_4 \cdot \be_3 \be_4 \be_3^{-1} 
  \cdot \be_2 \be_4 \be_2^{-1} \cdot \be_1 \be_4 \be_1^{-1} \\
  &=\be_1 \be_2 \be_4 \be_3 \be_4 \be_2 \be_4 \be_2^{-1} \be_1 \be_4 \be_1^{-1} .
\end{align*}

\begin{Lem}
  The relations (\ref{rel-7})--(\ref{rel-10}) 
  and (\ref{rel-7'}')--(\ref{rel-10'}') follow from (A)--(C). 
\end{Lem}
\begin{proof}
  We show the lemma only for (\ref{rel-7})--(\ref{rel-10}), 
  because the others are shown in a similar way. 
  We assume (A)--(C). 
  First, we rewrite the relations (\ref{rel-7}), (\ref{rel-8}), (\ref{rel-9}),
  by using $\be_i$'s:
  \begin{align*}
    \textrm{(\ref{rel-7})} 
    \Leftrightarrow &\al_7^{-1} \al_8^{-1}\al_7^{-1} \al_8^{-1} =\al_8^{-1} \al_7^{-1} \al_8^{-1} \al_7^{-1} \\
    \Leftrightarrow &\be_1 \be_4^{-1} \be_2 \be_4 \be_3 \be_4 \be_2 \be_4 
    \be_1 \be_4 \be_3 \be_4 \be_2 \be_4 \be_2^{-1} \be_1 \be_4 \be_1^{-1} \\
    &=\be_1 \be_2 \be_4 \be_3 \be_4 \be_2 \be_4 \be_1 \be_4 \be_3 \be_4 \be_2 \be_4 \be_2^{-1} \\
    \Leftrightarrow &\be_3 \be_4 \be_2 \be_4 
    \be_1 \be_4 \be_3  \be_1^{-1} \be_4 \be_1
    =\be_2 \be_4 \be_2^{-1} \be_3 \be_4 \be_2 \be_4 \be_1 
    \be_4 \be_3 , 
  \end{align*}
  \begin{align*}
    \textrm{(\ref{rel-8})} 
    \Leftrightarrow &((\al_3 \al_4)^{-1} \al_2 (\al_3 \al_4) \cdot \al_5)^2 
    =(\al_5 \cdot (\al_3 \al_4)^{-1} \al_2 (\al_3 \al_4))^2 \\
    \Leftrightarrow &(\be_4^{-1}\be_3^{-1} \be_2 \be_3 \be_4 \be_3 \be_4 \be_3^{-1})^2 
    =(\be_3 \be_4 \be_3^{-1} \be_4^{-1} \be_3^{-1} \be_2 \be_3 \be_4)^2 \\
    \Leftrightarrow &\be_2 \be_4 \be_2 \be_4
    =\be_4 \be_2 \be_4 \be_2 
    \quad \textrm{ (this is a relation in (C))}, 
  \end{align*}
  \begin{align*}
    \textrm{(\ref{rel-9})} 
    \Leftrightarrow & \be_2 \be_3 \be_4 \be_3^{-1} \be_2^{-1}
    =\be_1 \be_2 \be_4 \be_3 \be_4 \be_2 \be_4 \be_2^{-1} \be_1 \be_4 \be_1^{-1}
    (\be_1 \be_2 \be_4 \be_3 \be_4 \be_2 \be_4 \be_2^{-1})^{-1} \\
    \Leftrightarrow & \be_4 
    =\be_3^{-1}\be_1 \be_4 \be_3 \be_4  \be_1 \be_4 
    \be_1^{-1} \be_4^{-1} \be_3^{-1} \be_4^{-1} \be_1^{-1} \be_3 . 
  \end{align*}
  Next, we show (\ref{rel-7}) and (\ref{rel-9}), 
  since (\ref{rel-8}) is already proved. 
  Note that $[\be_i ,\be_j]=1$ and $[\be_i \be_4 \be_i^{-1}, \be_j \be_4 \be_j^{-1}]=1$ 
  imply $[\be_i^{-1} \be_4 \be_i, \be_j^{-1} \be_4 \be_j]=1$. 
  The left-hand side of (\ref{rel-7}) is 
  \begin{align*}
    &\be_3 \be_4 \be_2 \be_4 \be_1 (\be_3 \be_3^{-1}) \be_4 \be_3  \be_1^{-1} \be_4 \be_1 
    =\be_3 \be_4 \be_2 \be_4 \be_1 \be_3 \be_1^{-1} \be_4 \be_1 \be_3^{-1} \be_4 \be_3 \\
    &=\be_3 \be_4 \be_2 \be_4 \be_3 \be_4 \be_1 \be_3^{-1}\be_4 \be_3 
    =\be_3 \be_4 \be_2 \be_4 \be_3 \be_4 \be_3^{-1} \be_1 \be_4 \be_3 ,
  \end{align*}
  and the right-hand side is 
  \begin{align*}
    &\be_2 \be_4 \be_2^{-1} \be_3 \be_4 (\be_3^{-1} \be_3) \be_2 \be_4 \be_1 
    \be_4 \be_3 
    =\be_3 \be_4 \be_3^{-1} \be_2 \be_4 \be_2^{-1} \be_3 \be_2 \be_4 \be_1 
    \be_4 \be_3 \\
    &=\be_3 \be_4 \be_3^{-1} \be_2 \be_4 \be_3 \be_4 \be_1 
    \be_4 \be_3
    =\be_3 \be_4 \be_2 \be_3^{-1} \be_4 \be_3 \be_4 \be_1 
    \be_4 \be_3 .
  \end{align*}
  Thus, (\ref{rel-7}) is equivalent to 
  $\be_4 \be_3 \be_4 \be_3^{-1} = \be_3^{-1} \be_4 \be_3 \be_4$ 
  which is nothing but a relation in (C). 
  The right-hand side of (\ref{rel-9}) is 
  \begin{align*}
    &\be_1 \be_3^{-1} \be_4 \be_3 \be_4  \be_1 \be_4 
    \be_1^{-1} \be_4^{-1} \be_3^{-1} \be_4^{-1} \be_3 \be_1^{-1} \\
    &=\be_1 \be_4 
    \be_3 \be_4 \be_3^{-1} \be_1 \be_4 \be_1^{-1} 
    \be_4^{-1} \be_3^{-1} \be_4^{-1} \be_3 \be_1^{-1} \\
    &=\be_1 \be_4 \be_1 \be_4  
    \be_1^{-1} \be_3 \be_4 
    \be_3^{-1} \be_4^{-1} \be_3^{-1} \be_4^{-1} \be_3 \be_1^{-1} \\
    &=\be_4 \be_1 \be_4 \be_1 
    \be_1^{-1} \be_3 \be_4 
    \be_4^{-1} \be_3^{-1} \be_4^{-1} \be_3^{-1}  \be_3 \be_1^{-1} 
    =\be_4 ,
  \end{align*}
  and hence (\ref{rel-9}) is proved. 
  Finally, we show (\ref{rel-10}). 
  By using (\ref{rel-7}), we have
  $$
  \tal_7 \tal_8 \tal_7^{-1}
  =\al_8^{-1} \al_7 \al_8 \al_7 \al_8 \al_7^{-1}\al_8^{-1} \al_7^{-1} \al_8
  =\al_8 . 
  $$
  Thus, the (\ref{rel-10}) is rewritten by $\be_i$'s as follows: 
  \begin{align*}
    \textrm{(\ref{rel-10})} 
    \Leftrightarrow & [\ttal_2 , \tal_7 \tal_8 \tal_7^{-1}]=1 
    \ \Leftrightarrow [ \ttal_2 ,\al_8^{-1}]=1 \\
    \Leftrightarrow & 
    [(\be_3 \be_4 \cdot \be_3 \be_4 \be_3^{-1} \cdot \be_2 \be_4 \be_2^{-1})^{-1} \be_2 
    (\be_3 \be_4 \cdot \be_3 \be_4 \be_3^{-1} \cdot \be_2 \be_4 \be_2^{-1}), \\
    & \quad \be_1 \be_2 \be_4 \be_3 \be_4 \be_2 \be_4 \be_2^{-1} \be_1 \be_4 \be_1^{-1}]=1 \\
    \Leftrightarrow & 
    [ \be_2,
    \be_4 \be_3 \be_4 \be_2 \be_4 
    \be_1 \be_4 \be_3 \be_4  \be_1 \be_4 \be_1^{-1}
    \be_4^{-1} \be_3^{-1} \be_4^{-1}]=1.  
  \end{align*}
  Since 
  \begin{align*}
    &\be_4 \be_3 \be_4 \be_2 \be_4 
    \be_1 \be_4 \be_3 \be_4 \be_1 \be_4 \be_1^{-1}
    \be_4^{-1} \be_3^{-1} \be_4^{-1} 
    =\be_4 \be_3 \be_4 \be_2 \be_4 \be_3 \be_4 \be_3^{-1} \be_1  
  \end{align*}
  and $[\be_2 ,\be_3^{-1} \be_1 ]=1$, (\ref{rel-10}) is equivalent to 
  \begin{align*}
     \be_2 \cdot \be_4 \be_3 \be_4 \be_2 \be_4 \be_3 \be_4 
     =\be_4 \be_3 \be_4 \be_2 \be_4 \be_3 \be_4 \cdot \be_2 .
  \end{align*}
  This is shown as 
  \begin{align*}
    &\be_2 \cdot \be_4 \be_3 \be_4 \be_2 \be_4 \be_3 \be_4 \cdot \be_2^{-1}
    =\be_2 \be_4 (\be_2^{-1} \be_2) \be_3 \be_4 \be_2 \be_4 \be_3 \be_4 \be_2^{-1} \\
    &=\be_2 \be_4 \be_2^{-1} \be_3 \be_2 \be_4 \be_2 \be_4 \be_3 \be_4 \be_2^{-1}
    =\be_2 \be_4 \be_2^{-1} \be_3 \be_4 (\be_3^{-1}  \be_3 ) 
    \be_2 \be_4 \be_2 \be_3 \be_4 \be_2^{-1} \\
    &=\be_3 \be_4 \be_3^{-1}  \be_2 \be_4 \be_2^{-1} \be_3  
    \be_2 \be_4 \be_2 \be_3 \be_4 \be_2^{-1}
    =\be_3 \be_4 \be_3^{-1}  \be_2 \be_4 \be_3 \be_4 \be_3 \be_2 \be_4 \be_2^{-1} \\
    &=\be_3 \be_4 \be_3^{-1} \be_2 \be_3 \be_4 \be_3 \be_4 \be_2 \be_4 \be_2^{-1}
    =\be_3 \be_4 \be_2 \be_3 \be_3^{-1} \be_4 \be_3 \be_2^{-1} \be_4 \be_2 \be_4 \\
    &=\be_3 \be_4 \be_2 \be_3  \be_2^{-1} \be_4 \be_2 \be_3^{-1} \be_4 \be_3 \be_4 
    =\be_3 \be_4 \be_3 \be_4 \be_2 \be_3^{-1} \be_4 \be_3 \be_4  \\
    &=\be_4 \be_3 \be_4 \be_3 \be_2 \be_3^{-1} \be_4 \be_3 \be_4
    =\be_4 \be_3 \be_4 \be_2  \be_4 \be_3 \be_4 .
  \end{align*}
  Therefore, the proof is completed. 
\end{proof}

\begin{Rem}
  \label{rem:a11}
  Note that the relations (\ref{rel-11}) and (\ref{rel-11'}') 
  follow from others. 
  Indeed, by (\ref{rel-10}), we have 
  \begin{align*}
    \textrm{(\ref{rel-11})} 
    \Leftrightarrow & [\al_1^{-1} \tal_8 \al_1 ,\tal_3]=1 
    \Leftrightarrow [\al_1^{-1} \tal_8^{-1} \al_1 ,\tal_3]=1 \\
    \Leftrightarrow &[\al_1^{-1} \cdot \al_7 \al_8^{-1} \al_7^{-1} \cdot \al_1 ,\tal_3]=1 \\
    \Leftrightarrow & 
    [\be_1^{-1} \cdot \be_1 \be_4 \be_2 \be_4 \be_3 \be_4 \be_2 \be_4 \be_2^{-1} \cdot \be_1,\be_3]=1 \\
    \Leftrightarrow & 
    [\be_4  \be_2 \be_4 \be_3 \be_4 \be_2 \be_4 ,\be_3]=1 ,
  \end{align*}
and this follows from (A)--(C). 
\end{Rem}

Summarizing the above arguments, we obtain the following theorem: 
\begin{Th}\label{th-3dim-plane}
  \begin{align*}
    \pi_1(X^{(3)}\cap H)=\left\langle \be_4,\be_1,\be_2,\be_3  \left| 
        \begin{array}{l}
          [\be_i,\be_j]=1, \ [\be_i \be_4 \be_i^{-1}, \be_j \be_4 \be_j^{-1}]=1 
          \ (1\leq i<j \leq 3) \\
          (\be_4\be_k)^2=(\be_k\be_4)^2 \ (1\leq k \leq 3)
        \end{array}
      \right. \right\rangle .
  \end{align*} 
\end{Th}

\subsection{Correspondence between $\be_i$'s and $\ga_j$'s}
To complete the proof of Theorem \ref{th-3dim}, 
we give relations in $\pi_1(X^{(3)})$ between the loops 
$\be_i$'s and $\ga_j$'s. 

By the parametrization (\ref{eq-plane-parameter}) of the plane $H$, 
we have 
$$
x_0+4x_3=-3x,\quad x_1+x_2=-2x ,
$$
and hence the defining equation of $H$ is given as 
\begin{align*}
  3x_1+3x_2-8x_3=2x_0 .
\end{align*}

We fix a sufficiently small positive number $\vep$. 
We consider a line $\CL'$ in $H$ defined as 
\begin{align}
  \label{eq-Lprime}
  y=-\frac{1}{3-\vep}(x-4)
\end{align}
which passes through $(x,y)=(1+\vep,1)$. 
\begin{figure}[h]
  \centering{
  \includegraphics[scale=0.8]{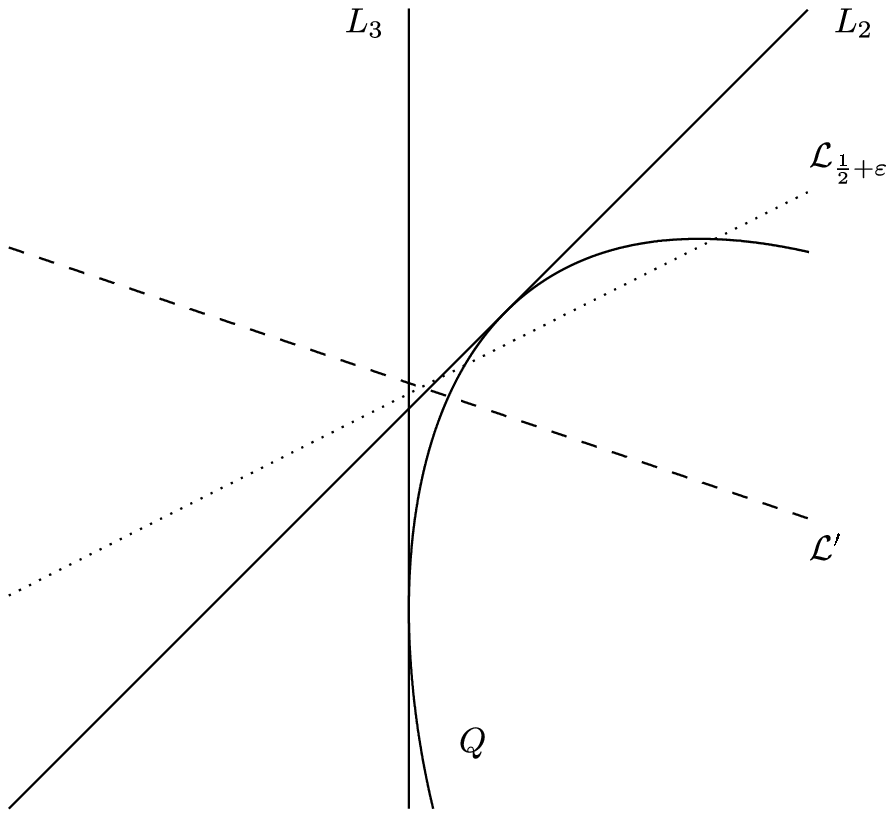} 
  }
  \caption{$\CL_{\frac{1}{2}+\vep}$ and $\CL'$ around $(1,1)\in \R^2$}
  \label{fig-sing3_LL}
\end{figure}
The loops $\al_1 ,\tal_3 ,\al_2 ,\tal_4$ in $\CL_{\frac{1}{2}+\vep}$ 
(see Figure \ref{fig-a45}) 
naturally define those in $\CL'$ (we use same notations). 
Since (\ref{eq-Lprime}) is expressed as 
$$
\frac{1}{2}(x_2-x_1) =-\frac{1}{3-\vep}x_0
$$
by (\ref{eq-plane-parameter}), 
the line $\CL' \subset \C^3$ is defined by 
\begin{align*}
  3x_1+3x_2-8x_3=2, \quad 
  x_1-x_2=\frac{2}{3-\vep} .
\end{align*}
By straightforward calculation, this line parametrized 
by $t\in \C$ as follows: 
\begin{align}
  \label{eq-Lprime-parameter}
  (x_1, x_2, x_3)=\left( \frac{6-\vep}{9-3\vep},\frac{-\vep}{9-3\vep},0 \right)
  +t\cdot \left( \frac{4}{3},\frac{4}{3},1  \right) .
\end{align}
If we identify $\CL'$ with $\C$ by $t$, then the intersection points 
$\CL' \cap (x_1=0)$, $\CL' \cap (x_2=0)$, $\CL' \cap (x_3=0)$ and $\CL' \cap S^{(3)}$ 
correspond to 
\begin{align*}
  t=-\frac{6-\vep}{4(3-\vep)},\quad 
  t=\frac{\vep}{4(3-\vep)},\quad  t=0,\quad 
  t=t'_1,t'_2,t'_3,t'_4 ,
\end{align*}
respectively, where $0<t'_1<t'_2<t'_3<t'_4$.
By definition of $\ga_i$'s and commutativity among $\ga_1, \ga_2, \ga_3$, 
the loop $\ga_0$ (resp. $\ga_1, \ga_2, \ga_3$) coincides with 
a loop which goes once around $t=t'_1$ 
(resp. $t=-\frac{6-\vep}{4(3-\vep)}$, $t=\frac{\vep}{4(3-\vep)}$, $t=0$)
approaching this point through the upper half-plane of the $t$-space. 

The loops $\al_1 ,\al_2 ,\tal_3 ,\tal_4$ in $\CL'$ (or $\CL_{\frac{1}{2}+\vep}$)
are defined under the parametrization by $x$. 
We should compare the parametrization by $x$ with that by $t$, 
and relate $\al_1 ,\al_2 ,\tal_3 ,\tal_4$ to $\ga_0, \ga_1, \ga_2 ,\ga_3$. 
The correspondence between the $x$-space and $t$-space is given by 
\begin{align}
  \label{eq-x-t}
  t=\frac{-x+1}{x-4} \ \left( =-1-\frac{3}{x-4}\right) .
\end{align}
Indeed, by (\ref{eq-Lprime}), we have
\begin{align*}
  x_1 =\frac{-x-y}{x-4}
  =\frac{6-\vep}{9-3\vep}+\frac{4}{3}t , \quad
  x_2 =\frac{y-x}{x-4}
  =\frac{-\vep}{9-3\vep}+\frac{4}{3}t  , \quad
  x_3 =t,
\end{align*}
and these expressions are coincide with (\ref{eq-Lprime-parameter}). 
The M\"obius transformation (\ref{eq-x-t}) is decomposed into
four elementary transformations 
$$
w=x-4,\quad v=\frac{1}{w} ,\quad u=-3v ,\quad t=u-1. 
$$
We see the change of $\al_1 ,\al_2 ,\tal_3 ,\tal_4$ under each transformations. 
In the $x$-space, they are drawn as follows.
\begin{align*}
  \includegraphics[scale=0.8]{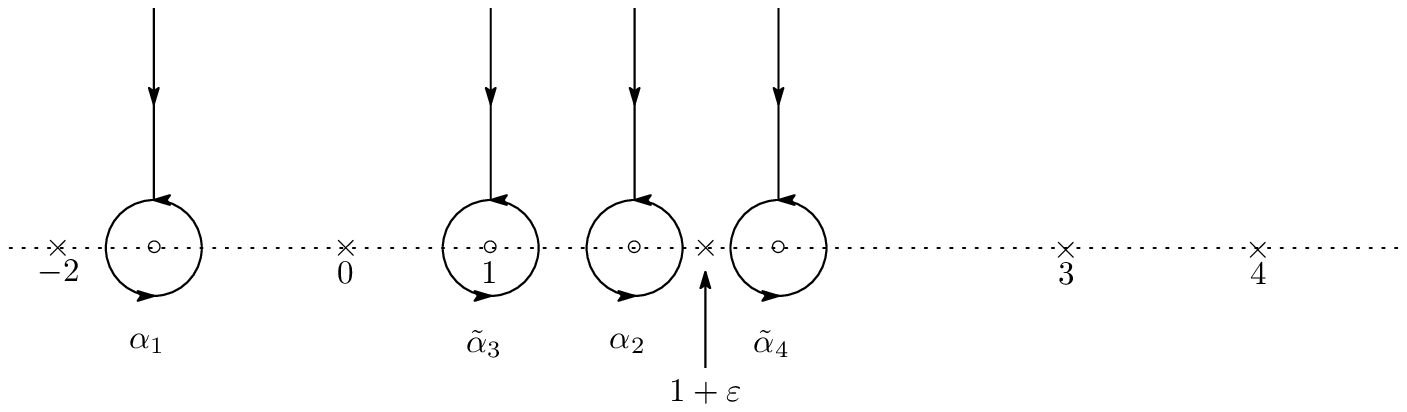} 
\end{align*}
\begin{enumerate}[(i)]
\item $w=x-4$; the changes are trivial.
  \begin{align*}
    \includegraphics[scale=0.8]{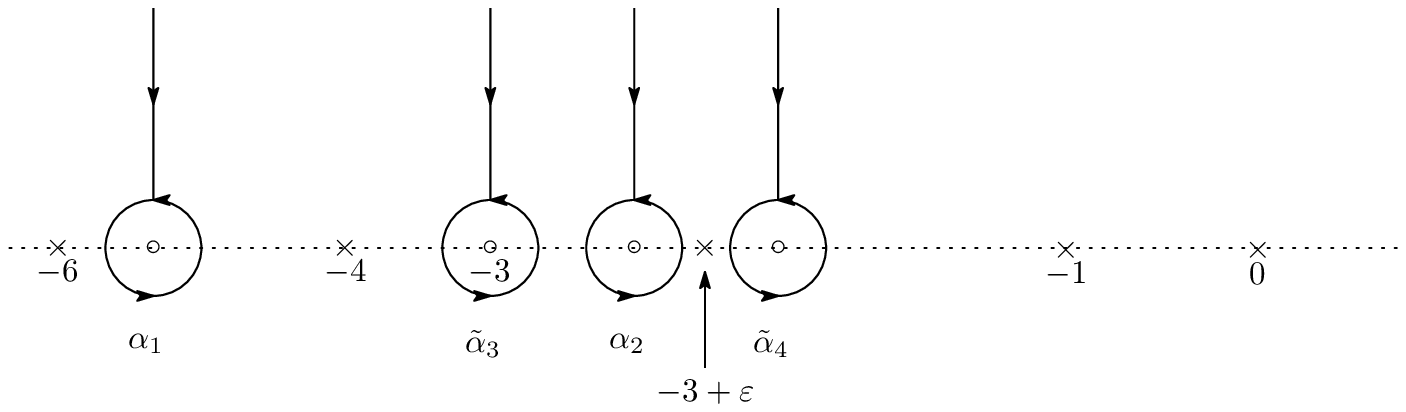} 
  \end{align*}
\item $v=\frac{1}{w}$; the approach to each circle is through 
  the lower half-plane in the $v$-space. 
  \begin{align*}
    \includegraphics[scale=0.8]{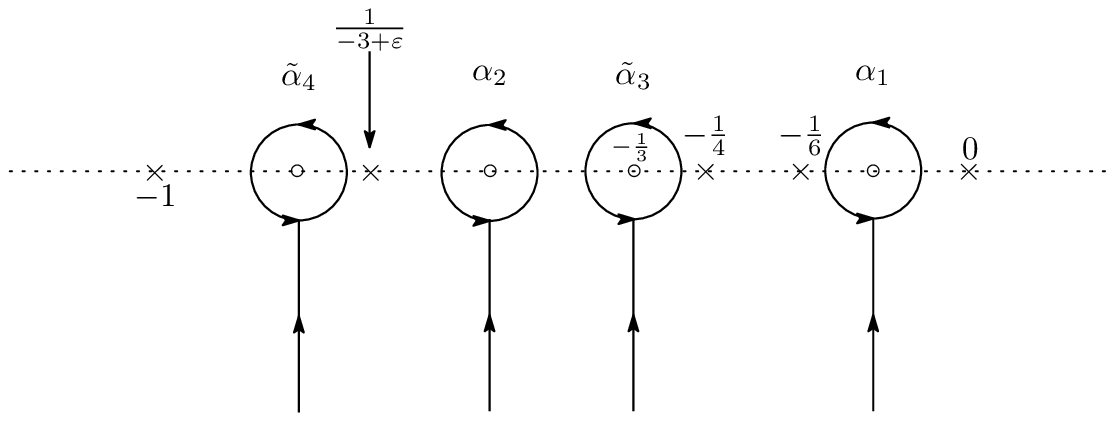} 
  \end{align*}
\item $u=-3v$; the approaches are changed again. 
  \begin{align*}
    \includegraphics[scale=0.8]{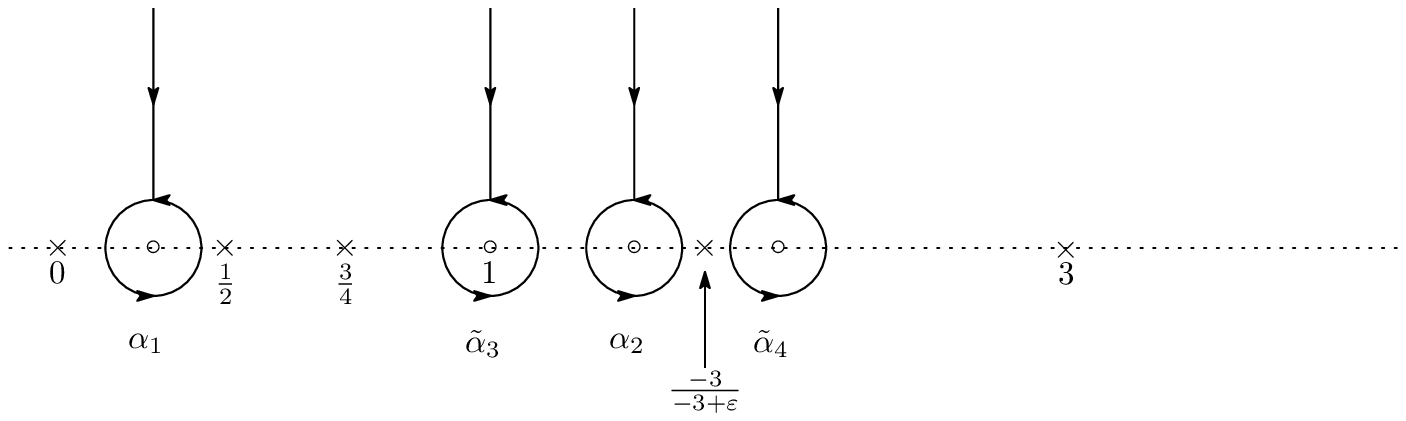} 
  \end{align*}
\item $t=u-1$; the changes are trivial.
  \begin{align*}
    \includegraphics[scale=0.8]{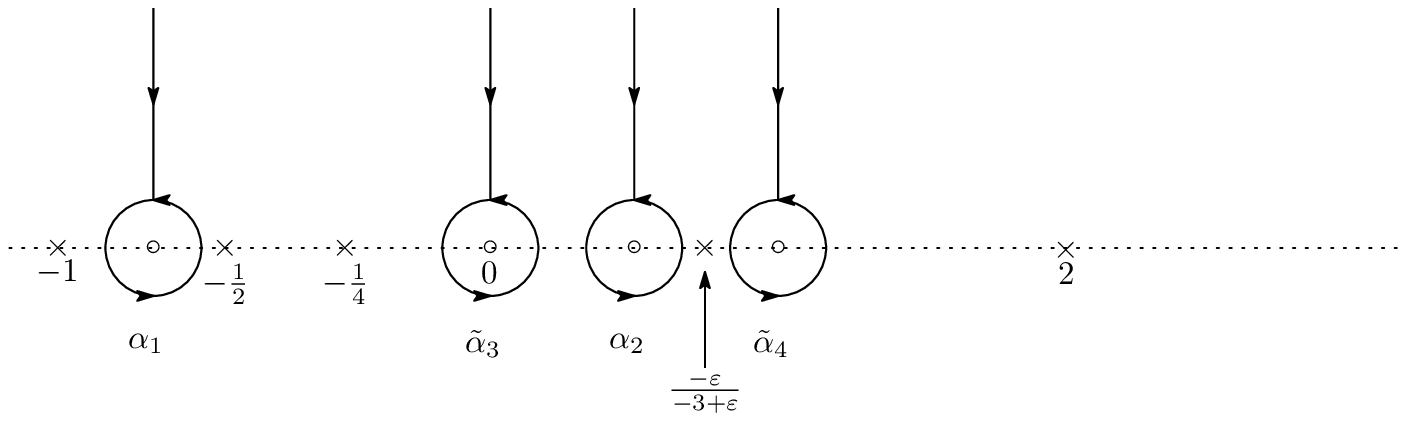} 
  \end{align*}
\end{enumerate}
As mentioned above, the loops $\al_1 ,\tal_3 ,\al_2 ,\tal_4$ 
in the last picture coincide with $\ga_1 ,\ga_3 ,\ga_2 ,\ga_0$, respectively.  
Therefore, we obtain the following lemma. 
\begin{Lem}
  The loops $\al_1 ,\al_2 ,\tal_3 ,\tal_4$ in $\CL'$ coincide with 
  $\ga_1, \ga_2 ,\ga_3, \ga_0$ as elements in $\pi_1(X^{(3)})$, respectively. 
  Therefore, $\be_1(=\al_1),\be_2(=\al_2) ,\be_3(=\tal_3) ,\be_4(=\tal_4) \in \pi_1 (X^{(3)}\cap H)$ are 
  mapped into $\ga_1, \ga_2 ,\ga_3, \ga_0$ under the isomorphism (\ref{eq-isom-plane}), 
  respectively. 
\end{Lem}
By this lemma and Theorem \ref{th-3dim-plane}, we obtain Theorem \ref{th-3dim}.

\section{The covering space ---the complement of hyperplane arrangement}
\subsection{Covering spaces}\label{subsection-covering}
We consider a branched $2^n$-covering
$$
\phi : \C^n \to \C^n ;\quad 
(\xi_1 ,\ldots ,\xi_n) \mapsto 
(x_1 ,\ldots ,x_n) =(\xi_1^2 ,\ldots ,\xi_n^2)
$$
of $\C^n$, and we put $\tilde{S}^{(n)}=\phi^{-1} (S^{(n)})$. 
The pull-back of $F_n (x)$ by $\phi$ 
is decomposed into the product of linear forms in $\xi_i$'s: 
$$
F_n(\phi(\xi))=\prod _{(a_1 ,\ldots ,a_n) \in \ztwo^n} 
\left( 1-\sum_{k=1}^n (-1)^{a_k} \xi_k \right) .
$$
Thus, $\tilde{S}^{(n)}$ is the union of hyperplanes in $\C^n$: 
\begin{align*}
  \tilde{S}^{(n)}
  =\bigcup_{(a_1 ,\ldots ,a_n) \in \ztwo^n } H(a_1 ,\ldots ,a_n) ,\quad 
  H(a_1 ,\ldots ,a_n) =\left( 1-\sum_{k=1}^n (-1)^{a_k} \xi_k =0 \right) .
\end{align*}
In this section, we consider the fundamental group of 
\begin{align*}
  \tilde{X}^{(n)}=\phi^{-1} (X^{(n)})=\C^n - \left( \bigcup_{k=1}^n H_k \cup 
    \bigcup_{(a_1 ,\ldots ,a_n) \in \ztwo^n } H(a_1 ,\ldots ,a_n) \right) ,
\end{align*}
where $H_k=(\xi_k=0)$. 
The restriction 
$$
\phi :\tilde{X}^{(n)} \longrightarrow X^{(n)}
$$
of $\phi$ to $\tilde{X}^{(n)}$ is a $(\Z/ 2\Z)^n$-Galois covering. 
Hence, we have a short exact sequence 
\begin{align*}
  1 \longrightarrow \pi_1 (\tilde{X}^{(n)}) \overset{\phi_*}{\longrightarrow} 
  \pi_1 (X^{(n)}) \longrightarrow (\Z / 2\Z)^n \longrightarrow 1 .
\end{align*}

For $(a_1 ,\ldots ,a_n) \in \ztwo^n (=(\Z / 2\Z)^n)$, let 
$$
\xi_{a_1 \cdots a_n} =\left( \frac{(-1)^{a_1}}{\sqrt{2}n},\ldots ,\frac{(-1)^{a_n}}{\sqrt{2}n} \right)
\in \phi^{-1}(\dot{x}). 
$$
We define a path $\tga_k^{(a_1 \cdots a_n)}$ and a loop $\tga_0^{(a_1 \cdots a_n)}$ in $\tilde{X}^{(n)}$
as lifts of $\ga_k$ and $\ga_0$ such that
$\tga_k^{(a_1 \cdots a_n)} (0) =\tga_0^{(a_1 \cdots a_n)} (0) =\xi_{a_1 \cdots a_n}$, respectively.  
Note that 
$\tga_k^{(a_1 \cdots a_n)} (1)=\xi_{a_1 ,\dots, a_k+1 ,\ldots, a_n}$ and 
$\tga_0^{(a_1 \cdots a_n)} (1) =\xi_{a_1 \cdots a_n}$. 
Of course, we have 
$\phi_{*} (\tga_k^{(a_1 \cdots a_n)})=\ga_k$ and 
$\phi_{*} (\tga_0^{(a_1 \cdots a_n)}) =\ga_0$.
For $1\leq i_1 <i_2 <\dots <i_k \leq n$, 
we put 
\begin{align*}
  \tau^{(i_1 \cdots i_k)}
  =\tga_{i_1}^{(0 \cdots 0)}\tga_{i_2}^{(0 \cdots \overset{i_1}{1} \cdots 0)}
  \tga_{i_3}^{(0 \cdots \overset{i_1}{1} \cdots \overset{i_2}{1} \cdots 0)}
  \cdots \tga_{i_k}^{(0 \cdots \overset{i_1}{1} \cdots \overset{i_2\cdots}{1\cdots} 
    \overset{i_{k-1}}{1} \cdots 0)} .
\end{align*}
We consider loops in $\pi_1 (\tilde{X}^{(n)} ,\xi_{0\cdots 0})$: 
\begin{align*}
  &\la_0 = \tga_0^{(0 \cdots 0)}, \\
  &\la_k =\tga_k^{(0 \cdots 0 \cdots 0)}\tga_k^{(0 \cdots 1 \cdots 0)} \quad 
  (k=1,\dots ,n), \\
  &\la_0^{(i_1 \cdots i_k)} 
  =\tau^{(i_1 \cdots i_k)} \ga_0^{(0 \cdots \overset{i_1}{1} \cdots \overset{i_2\cdots}{1\cdots} 
    \overset{i_k}{1} \cdots 0)} \overline{\tau^{(i_1 \cdots i_k)}} \quad
  (1\leq i_1 <i_2 <\dots <i_k \leq n).
\end{align*}
Figure \ref{fig-lines2dim} shows some loops and paths in $\tilde{X}^{(2)}$. 
For example, $\la_0^{(12)}$ is defined as 
$\la_0^{(12)}=\tga_1^{(00)}\tga_2^{(10)} \cdot \tga_0^{(11)} 
\cdot \overline{\tga_1^{(00)}\tga_2^{(10)}}$. 
\begin{figure}[h]
  \centering{
  \includegraphics[scale=0.8]{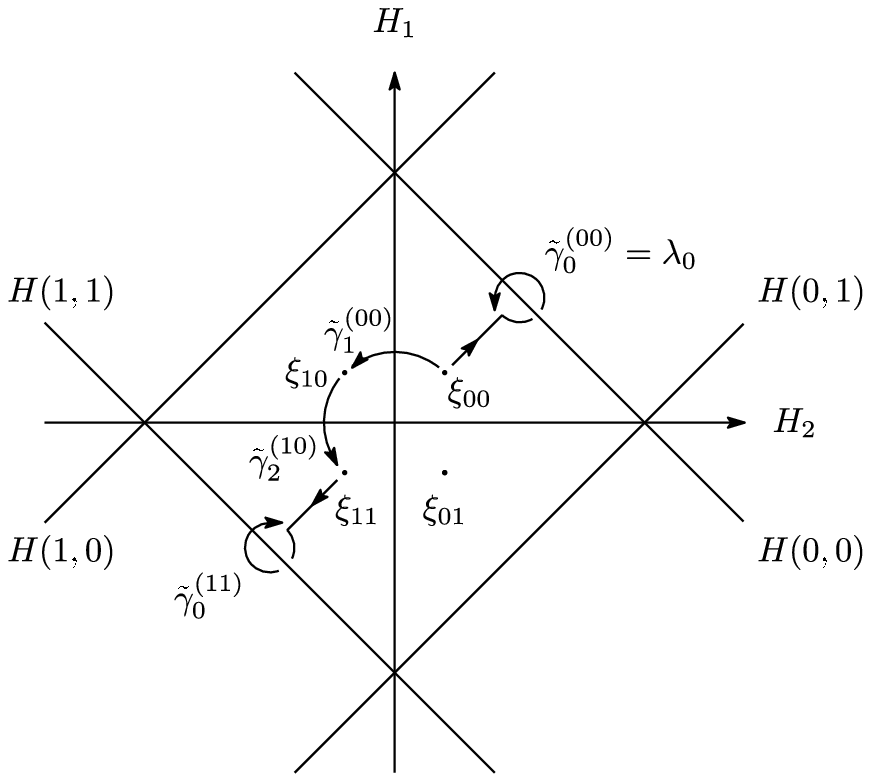} 
  }
  \caption{Some loops and paths in $\tilde{X}^{(2)}$.}
  \label{fig-lines2dim}
\end{figure}

By the definition, we obtain the following. 
\begin{Lem}
  \label{lem-covering}
  We have 
  \begin{align*}
    &\phi_{*} (\la_0 )=\ga_0 ,\quad
    \phi_{*} (\la_k )=\ga_k^2 ,\\ 
    &\phi_{*} (\la_0^{(i_1 \cdots i_k)})
    =(\ga_{i_1}\ga_{i_2}\cdots \ga_{i_k}) \ga_0  (\ga_{i_1}\ga_{i_2}\cdots \ga_{i_k})^{-1} .
  \end{align*}
\end{Lem}


\subsection{The case of $n=2$}
By using the Reidemeister-Schreier method, 
we obtain a presentation of $\pi_1 (\tilde{X}^{(2)})$. 
Since computations are similar to that in the next subsection, 
we do not give precise computations. 
\begin{Prop}
  The fundamental group 
  $\pi_1 (\tilde{X}^{(2)})$ has a presentation by 6 generators
  \begin{align*}
    \la_1,\la_2,\la_0, 
    \la_0^{(1)},\la_0^{(2)},\la_0^{(12)},
  \end{align*}
  and 5 defining relations
  \begin{align*}
    &[\la_1 ,\la_2]=1 ,  \\
    &\la_0^{(i)}\la_i\la_0=\la_0\la_0^{(i)}\la_i=\la_i\la_0\la_0^{(i)}  \quad (i=1,2), \\
    &\la_i\la_0^{(j)}\la_0^{(12)}=\la_0^{(j)}\la_0^{(12)}\la_i=\la_0^{(12)}\la_i\la_0^{(j)} 
    \quad (\{i,j\}=\{1,2\}) .
  \end{align*}
\end{Prop}
\begin{proof}[Sketch of Proof]
  Let $K$ be the free group generated by $\ga_0$, $\ga_1$, $\ga_2$, 
  and $\vph : K\to \pi_1(X^{(2)})$ be the natural epimorphism. 
  The subgroup $K_1 =\vph^{-1}(\pi_1(\tilde{X}^{(2)}))$ of $K$ is also free, 
  and the set 
  \begin{align*}
    T=\{ 1,\ga_1, \ga_2,\ga_1 \ga_2 \} \subset K
  \end{align*}
  is a Schreier transversal for $K_1$ in $K$. 
  By the Reidemeister-Schreier method, we obtain 
  \begin{align*}
    &\ga_0 ,\ \ga_1^2 , \ \ga_2^2 , \ 
    \ga_1  \ga_0 \ga_1^{-1} ,\ \ga_2 \ga_0 \ga_2^{-1} , \
    \ga_1 \ga_2 \ga_0 (\ga_1 \ga_2)^{-1} ,\\
    & \ga_2 \ga_1 \ga_2^{-1} \ga_1^{-1} ,\ 
    \ga_1 \ga_2 \ga_1 \ga_2^{-1} ,\ 
    \ga_1 \ga_2^2 \ga_1^{-1} 
  \end{align*}
  as generators of $K_1$, and 
  we also obtain 12 relations in $\pi_1 (\tilde{X}^{(2)})$. 
  To determine the generator, imitate (i) and (ii) in the proof of Lemma \ref{lem-RS-generator}.  
  We obtain similar relations to 
  (\ref{eq-RS-rel-1-0}), 
  (\ref{eq-RS-rel-1-1-1}), 
  (\ref{eq-RS-rel-1-1-2}), 
  (\ref{eq-RS-rel-1-2-1}), 
  (\ref{eq-RS-rel-3-0}), 
  (\ref{eq-RS-rel-3-1-1}), 
  (\ref{eq-RS-rel-3-1-2}), 
  (\ref{eq-RS-rel-3-1-3}), 
  (\ref{eq-RS-rel-3-2-1}), 
  (\ref{eq-RS-rel-3-2-2}). 
  Using the correspondence in Lemma \ref{lem-covering}, 
  we obtain the proposition. 
\end{proof}

\subsection{The case of $n=3$}
\begin{figure}[h]
  \centering{
  \includegraphics[scale=0.8]{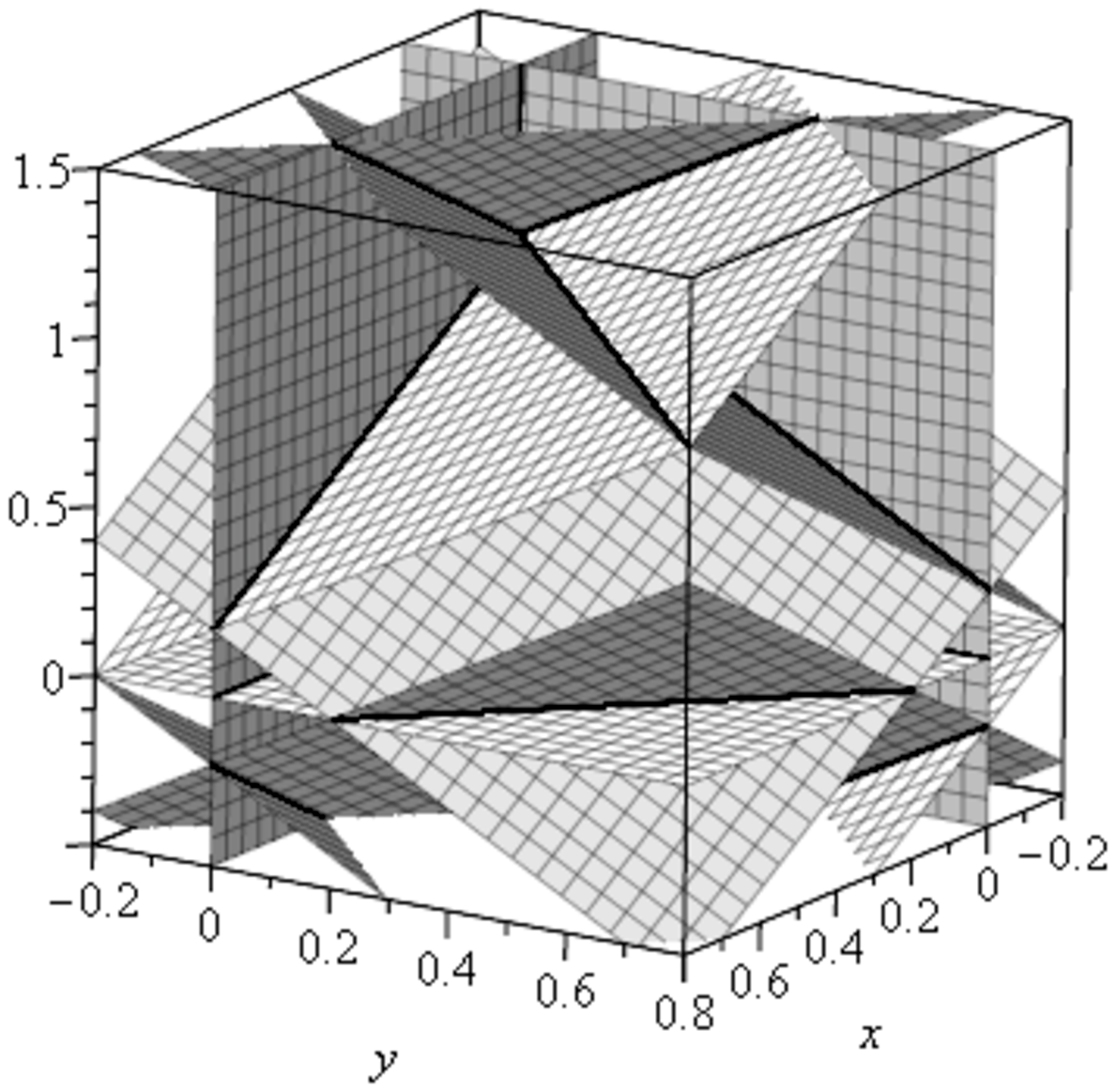} 
  }
  \caption{A part of $\tilde{X}^{(3)}$.}
  \label{fig-planes3dim}
\end{figure}
By using Theorem \ref{th-3dim}, 
we now compute the fundamental group $\pi_1 (\tilde{X}^{(3)})$. 
As in Subsection \ref{subsection-covering}, we consider the 11 loops
\begin{align}
  \label{eq-11-loops}
  \la_1, \ \la_2, \ \la_3 ,\ \la_0,\ 
  \la_0^{(1)}, \ \la_0^{(2)}, \ \la_0^{(3)}, \ \la_0^{(12)}, \ \la_0^{(13)}, \ \la_0^{(23)}, \ \la_0^{(123)}.
\end{align}
By Lemma \ref{lem-covering}, we have
\begin{align*}
  &\phi_{*} (\la_0)=\ga_0 , \ 
  \phi_{*} (\la_1)=\ga_1^2, \ 
  \phi_{*} (\la_2)=\ga_2^2, \ 
  \phi_{*} (\la_3)=\ga_3^2,\\ 
  &\phi_{*} (\la_0^{(1)})=\ga_1\ga_0\ga_1^{-1}, \ 
  \phi_{*} (\la_0^{(2)})=\ga_2\ga_0\ga_2^{-1}, \ 
  \phi_{*} (\la_0^{(3)})=\ga_3\ga_0\ga_3^{-1},  \\
  &\phi_{*} (\la_0^{(12)})=\ga_1\ga_2\ga_0(\ga_1\ga_2)^{-1}, \ 
  \phi_{*} (\la_0^{(13)})=\ga_1\ga_3\ga_0(\ga_1\ga_3)^{-1}, \\ 
  &\phi_{*} (\la_0^{(23)})=\ga_2\ga_3\ga_0(\ga_2\ga_3)^{-1}, \ 
  \phi_{*} (\la_0^{(123)})=\ga_1\ga_2\ga_3\ga_0(\ga_1\ga_2\ga_3)^{-1}.
\end{align*}
We put $G=\pi_1 (X^{(3)})$ and $G_1=\pi_1 (\tilde{X}^{(3)})$. 
Recall the short exact sequence 
\begin{align*}
  1 \longrightarrow G_1 \overset{\phi_*}{\longrightarrow} 
  G \overset{p}{\longrightarrow} (\Z / 2\Z)^3 \longrightarrow 1 .
\end{align*}
It is easy to see that 
this sequence is realized as 
\begin{align*}
  &p(\ga_0) =0 ,\ 
  p(\ga_1) =(1,0,0) ,\ 
  p(\ga_2) =(0,1,0) ,\ 
  p(\ga_3) =(0,0,1) ,\\
  &G_1 =\{ g\in G  \mid 
  \textrm{the sum of the exponents of }\ga_i \textrm{ is even for each }
  i=1,2,3 \} .
\end{align*}
If we put 
$$
q:(\Z / 2\Z)^3 \to G;\quad
q(b_1 ,b_2 ,b_3) =\ga_1^{b_1} \ga_2^{b_2} \ga_3^{b_3}, 
$$
then we have $q\circ p =\textrm{id}$. 
Thus, the above exact sequence is split one, and 
$G$ is a semidirect product of $G_1$ by $(\Z / 2\Z)^3$, 
that is, $G=G_1 \rtimes (\Z / 2\Z)^3$. 

We determine the generators and relations of $G_1$ by 
the Reidemeister-Schreier method (see, e.g., \cite[Chapter 2]{Bogopolski}). 

Let $K$ be the free group generated by $\ga_0$, $\ga_1$, $\ga_2$, $\ga_3$, 
and $\vph : K\to G$ be the natural epimorphism. 
Note that the subgroup $K_1 =\vph^{-1}(G_1)$ of $K$ is also free. 
The set 
\begin{align*}
  T=\{ 1,\ga_1, \ga_2, \ga_3 ,\ga_1 \ga_2 ,\ga_1 \ga_3 ,\ga_2 \ga_3 ,\ga_1 \ga_2 \ga_3 \}
  \subset K
\end{align*}
is a Schreier transversal for $K_1$ in $K$. 
For $g\in K$, denote by $\ol{g}$ the unique element of $T$ such that $K_1 g=K_1 \ol{g}$. 

\begin{Lem}
  \label{lem-RS-generator}
  The following 25 elements form a generator of the free group $K_1$: 
  \begin{align}
    \label{eq-25-generator}
    & \ga_0 ,\ \ga_1^2 , \ \ga_2^2 , \ \ga_3^2 , \ 
    \ga_1  \ga_0 \ga_1^{-1} ,\ \ga_2 \ga_0 \ga_2^{-1} ,\ \ga_3 \ga_0 \ga_3^{-1} ,\\
    \nonumber
    & \ga_j \ga_i \ga_j^{-1} \ga_i^{-1} ,\ 
    \ga_i \ga_j \ga_i \ga_j^{-1} ,\ 
    \ga_i \ga_j^2 \ga_i^{-1} , \
    \ga_i \ga_j \ga_0 (\ga_i \ga_j)^{-1}, \\
    \nonumber
    & \ga_1 \ga_3 \ga_2 (\ga_1 \ga_2 \ga_3)^{-1}, \ 
    \ga_2 \ga_3 \ga_1 (\ga_1 \ga_2 \ga_3)^{-1}, \\
    \nonumber
    & \ga_1 \ga_2 \ga_3 \ga_1 (\ga_2 \ga_3)^{-1},\ 
    \ga_1 \ga_2 \ga_3 \ga_2 (\ga_1 \ga_3)^{-1},\ 
    \ga_1 \ga_2 \ga_3^2 (\ga_1 \ga_2)^{-1},\\ 
    \nonumber
    &\ga_1 \ga_2 \ga_3 \ga_0 (\ga_1 \ga_2 \ga_3)^{-1}, 
  \end{align}
  where $1\leq i < j \leq 3$. 
\end{Lem}
\begin{proof}
  We put $B=\{ \ga_0,\ga_1,\ga_2,\ga_3 \}$ which is a generator of $K$. 
  A generator of $K_1$ is given by 
  \begin{align*}
    \{ (tb)(\ol{tb})^{-1} \mid t\in T, b\in B , (tb)(\ol{tb})^{-1}\neq 1 \} .
  \end{align*}
  It is sufficient to compute all $(tb)(\ol{tb})^{-1}$. 
  \begin{enumerate}[(i)]
  \item In the case $t=1$, since $(tb)(\ol{tb})^{-1}$'s are
    \begin{align*}
      \ga_0 \ol{\ga_0}^{-1}=\ga_0 \cdot 1 =\ga_0 ,\  
      \ga_1 \ol{\ga_1}^{-1}=1,\ 
      \ga_2 \ol{\ga_2}^{-1}=1,\ 
      \ga_3 \ol{\ga_3}^{-1}=1,
    \end{align*}
    we obtain a generator $\ga_0$. 
  \item In the case $t=\ga_i$ $(1\leq i \leq 3)$,  
    $(tb)(\ol{tb})^{-1}$'s are 
    \begin{align*}
      \ga_i^2 (\ol{\ga_i^2})^{-1} , \ 
      \ga_i \ga_j (\ol{\ga_i \ga_j})^{-1},\ 
      \ga_i \ga_0 (\ol{\ga_i \ga_0})^{-1}, 
    \end{align*}
    where $j\neq i$. 
    Since $\ga_i^2, \ga_i \ga_0^{-1} \ga_i^{-1}\in K_1$, we have
    $\ol{\ga_i^2}=1$ and 
    $\ol{\ga_i \ga_0} =\ol{\ga_i \ga_0^{-1} \ga_i^{-1} \cdot \ga_i \ga_0} =\ga_i$. 
    Thus we obtain 
    \begin{align*}
      \ga_i^2 (\ol{\ga_i^2})^{-1} =\ga_i^2 , \ 
      \ga_i \ga_0 (\ol{\ga_i \ga_0})^{-1} =\ga_i \ga_0 \ga_i^{-1}. 
    \end{align*}
    We compute $\ga_i \ga_j (\ol{\ga_i \ga_j})^{-1}$. 
    If $i<j$, then $\ol{\ga_i \ga_j}=\ga_i \ga_j$. 
    If $i>j$, then $\ol{\ga_i \ga_j}=\ol{\ga_j \ga_i \ga_j^{-1} \ga_i^{-1} \cdot \ga_i \ga_j}=\ga_j \ga_i$. 
    We thus have 
    \begin{align*}
      \ga_i \ga_j (\ol{\ga_i \ga_j})^{-1} =\left\{
        \begin{array}{ll}
          1 & (i<j) \\
          \ga_i \ga_j \ga_i^{-1} \ga_j^{-1} & (i>j) .
        \end{array}
      \right.
    \end{align*}
    Therefore, we obtain generators $\ga_i^2$, $\ga_i \ga_0 \ga_i^{-1}$ $(1\leq i \leq 3)$ 
    and $\ga_i \ga_j \ga_i^{-1} \ga_j^{-1}$ $(1\leq j<i \leq 3)$. 
  \item In the case $t=\ga_1 \ga_2$, $\ol{tb}$'s are
    \begin{align*}
      &\ol{\ga_1 \ga_2 \ga_0}=\ol{\ga_1 \ga_2 \ga_0 (\ga_1 \ga_2)^{-1} \cdot \ga_1 \ga_2} =\ga_1 \ga_2 , \\
      &\ol{\ga_1 \ga_2 \ga_1}=\ol{\ga_1 \ga_2 \ga_1 \ga_2^{-1} \cdot \ga_2} =\ga_2 ,\quad  
      \ol{\ga_1 \ga_2 \ga_2}=\ol{\ga_1 \ga_2^2 \ga_1^{-1} \cdot \ga_1} =\ga_1 ,\\
      &\ol{\ga_1 \ga_2 \ga_3}=\ga_1 \ga_2 \ga_3 ,
    \end{align*}
    and hence we obtain generators
    \begin{align*}
      &\ga_1 \ga_2 \ga_0 (\ol{\ga_1 \ga_2 \ga_0})^{-1} = \ga_1 \ga_2 \ga_0 (\ga_1 \ga_2)^{-1}, \\
      &\ga_1 \ga_2 \ga_1 (\ol{\ga_1 \ga_2 \ga_1})^{-1} = \ga_1 \ga_2 \ga_1 \ga_2^{-1} ,\quad 
      \ga_1 \ga_2 \ga_2 (\ol{\ga_1 \ga_2 \ga_2})^{-1} = \ga_1 \ga_2^2 \ga_1^{-1} .
    \end{align*}
  \end{enumerate}
  The following generators are obtained 
  in the same way as above. 
  \begin{enumerate}[(i)]
    \setcounter{enumi}{3}
  \item In the case $t=\ga_1 \ga_3$, 
    we obtain generators
    \begin{align*}
      &\ga_1 \ga_3 \ga_0 (\ol{\ga_1 \ga_3 \ga_0})^{-1} = \ga_1 \ga_3 \ga_0 (\ga_1 \ga_3)^{-1}, \quad 
      \ga_1 \ga_3 \ga_1 (\ol{\ga_1 \ga_3 \ga_1})^{-1} = \ga_1 \ga_3 \ga_1 \ga_3^{-1} ,\\ 
      &\ga_1 \ga_3 \ga_2 (\ol{\ga_1 \ga_3 \ga_2})^{-1} = \ga_1 \ga_3 \ga_2 (\ga_1 \ga_2 \ga_3)^{-1}  \quad 
      \ga_1 \ga_3 \ga_3 (\ol{\ga_1 \ga_3 \ga_3})^{-1} = \ga_1 \ga_3^2 \ga_1^{-1} .
    \end{align*} 
  \item In the case $t=\ga_2 \ga_3$, 
    we obtain generators
    \begin{align*}
      &\ga_2 \ga_3 \ga_0 (\ol{\ga_2 \ga_3 \ga_0})^{-1} = \ga_2 \ga_3 \ga_0 (\ga_2 \ga_3)^{-1}, \quad 
      \ga_2 \ga_3 \ga_1 (\ol{\ga_2 \ga_3 \ga_1})^{-1} = \ga_2 \ga_3 \ga_1 (\ga_1 \ga_2 \ga_3)^{-1},\\ 
      &\ga_2 \ga_3 \ga_2 (\ol{\ga_2 \ga_3 \ga_2})^{-1} =  \ga_2 \ga_3 \ga_2 \ga_3^{-1} \quad 
      \ga_2 \ga_3 \ga_3 (\ol{\ga_2 \ga_3 \ga_3})^{-1} = \ga_2 \ga_3^2 \ga_2^{-1} .
    \end{align*}  
  \item In the case $t=\ga_1 \ga_2 \ga_3$, 
    we obtain generators
    \begin{align*}
      &\ga_1 \ga_2 \ga_3 \ga_0 (\ol{\ga_1 \ga_2 \ga_3 \ga_0})^{-1} 
      =\ga_1 \ga_2 \ga_3 \ga_0 (\ga_1 \ga_2 \ga_3)^{-1}, \\ 
      &\ga_1 \ga_2 \ga_3 \ga_1 (\ol{\ga_1 \ga_2 \ga_3 \ga_1})^{-1} 
      =\ga_1 \ga_2 \ga_3 \ga_1 (\ga_2 \ga_3)^{-1},\\ 
      &\ga_1 \ga_2 \ga_3 \ga_2 (\ol{\ga_1 \ga_2 \ga_3 \ga_2})^{-1} 
      =\ga_1 \ga_2 \ga_3 \ga_2 (\ga_1 \ga_3)^{-1} \\ 
      &\ga_1 \ga_2 \ga_3 \ga_3 (\ol{\ga_1 \ga_2 \ga_3 \ga_3})^{-1} 
      =\ga_1 \ga_2 \ga_3^2 (\ga_1 \ga_2)^{-1} .
    \end{align*}   
  \end{enumerate}
  Therefore, we obtain the 25 generators. 
\end{proof}
We put 
\begin{align*}
  R=\left\{
    \begin{array}{cl}
      \ga_i \ga_j \ga_i^{-1} \ga_j^{-1} &(1\leq i<j\leq 3) , \\
      \ga_i \ga_0 \ga_i^{-1} \ga_j \ga_0 \ga_j^{-1} 
      \ga_i \ga_0^{-1} \ga_i^{-1} \ga_j \ga_0^{-1} \ga_j^{-1} 
      &(1\leq i<j\leq 3) , \\
      \ga_i \ga_0 \ga_i \ga_0 \ga_i^{-1} \ga_0^{-1} \ga_i^{-1} \ga_0^{-1} 
      &(1\leq i \leq 3)
    \end{array}
  \right\} 
\end{align*}
which generates the relations of $G$, that is, 
$G=\langle \ga_0,\ga_1,\ga_2,\ga_3 \mid R\rangle$. 
By the Reidemeister-Schreier method, $G_1$ is presented by  
the 25 generators (\ref{eq-25-generator}) 
with relations of the form 
\begin{align*}
  t r t^{-1} ,\quad t\in T,\ r\in R .
\end{align*}
Note that to obtain relations among the generators (\ref{eq-25-generator}), 
we need to rewrite these relations by them. 

We write down the $72(=8\cdot 9)$ relations. 
\begin{enumerate}[(i)]
  \setlength{\itemsep}{40pt}
\item $r=\ga_i \ga_j \ga_i^{-1} \ga_j^{-1}$ 
  ($1\leq i<j\leq 3$). 
  \begin{enumerate}[(a)]
  \item $t=1$. We obtain a relation
    \begin{align}
      \label{eq-RS-rel-1-0}
      \ga_j \ga_i \ga_j^{-1} \ga_i^{-1}=1. 
    \end{align}
  \item $t=\ga_k$. 
    We obtain the following relations
    \begin{align}
      \label{eq-RS-rel-1-1-1}
      1=\ga_i \cdot \ga_i \ga_j \ga_i^{-1} \ga_j^{-1} \cdot \ga_i^{-1}
      =\ga_i^2 (\ga_i\ga_j\ga_i\ga_j^{-1})^{-1}. 
    \end{align}
    \begin{align}
      \label{eq-RS-rel-1-1-2}
      1=\ga_j \cdot \ga_i \ga_j \ga_i^{-1} \ga_j^{-1} \cdot \ga_j^{-1}
      =\ga_j \ga_i \ga_j^{-1} \ga_i^{-1} \cdot \ga_i \ga_j^2 \ga_i^{-1}  \cdot \ga_j^{-2}. 
    \end{align}
    \begin{align}
      \label{eq-RS-rel-1-1-3}
      1&=\ga_3 \cdot \ga_1 \ga_2 \ga_1^{-1} \ga_2^{-1} \cdot \ga_3^{-1} \\
      \nonumber
      &=\ga_3 \ga_1 \ga_3^{-1} \ga_1^{-1} \cdot 
      \ga_1 \ga_3 \ga_2 (\ga_1 \ga_2 \ga_3)^{-1} \\
      \nonumber
      &\quad \cdot 
      \left( \ga_3 \ga_2 \ga_3^{-1} \ga_2^{-1} 
        \cdot \ga_2 \ga_3 \ga_1 (\ga_1 \ga_2 \ga_3)^{-1} \right)^{-1} ,
    \end{align}
    \begin{align}
      \label{eq-RS-rel-1-1-4}
      1&=\ga_2 \cdot \ga_1 \ga_3 \ga_1^{-1} \ga_3^{-1} \cdot \ga_2^{-1} \\
      \nonumber
      &=\ga_2 \ga_1 \ga_2^{-1} \ga_1^{-1} 
      \cdot \left(  \ga_2 \ga_3 \ga_1 (\ga_1 \ga_2 \ga_3)^{-1} \right)^{-1} ,
    \end{align}
    \begin{align}
      \label{eq-RS-rel-1-1-5}
      1=\ga_1 \cdot \ga_2 \ga_3 \ga_2^{-1} \ga_3^{-1} \cdot \ga_1^{-1} 
      =\left( \ga_1 \ga_3 \ga_2 (\ga_1 \ga_2 \ga_3)^{-1} \right)^{-1} .
    \end{align}
    In the following, we write only the results. 
  \item $t=\ga_k \ga_l$. 
    \begin{align}
      \label{eq-RS-rel-1-2-1}
      1
      =\ga_i \ga_j \ga_i \ga_j^{-1} \cdot \ga_j^2 \cdot \ga_i^{-2} 
      \cdot (\ga_i \ga_j^2 \ga_i^{-1})^{-1}. 
    \end{align}
    \begin{align}
      \label{eq-RS-rel-1-2-2}
      1
      &=\ga_1 \ga_3 \ga_1 \ga_3^{-1} \cdot \ga_3 \ga_2 \ga_3^{-1} \ga_2^{-1} \\
      \nonumber
      &\quad 
      \cdot \left( \ga_1 \ga_3 \ga_2 (\ga_1 \ga_2 \ga_3)^{-1} 
        \cdot \ga_1 \ga_2 \ga_3 \ga_1 (\ga_2 \ga_3)^{-1} \right)^{-1} , 
    \end{align}
    \begin{align}
      \label{eq-RS-rel-1-2-3}
      1
      &=\ga_2 \ga_3 \ga_1 (\ga_1 \ga_2 \ga_3)^{-1} 
      \cdot \ga_1 \ga_2 \ga_3 \ga_2 (\ga_1 \ga_3)^{-1} \\
      \nonumber &\qquad 
      \cdot \ga_1 \ga_3 \ga_1^{-1} \ga_3^{-1} 
      \cdot (\ga_2 \ga_3 \ga_2 \ga_3^{-1})^{-1}. 
    \end{align}
    \begin{align}
      \label{eq-RS-rel-1-2-4}
      1
      =\ga_1 \ga_2 \ga_1 \ga_2^{-1} 
      \cdot \left( \ga_1 \ga_2 \ga_3 \ga_1 (\ga_2 \ga_3)^{-1} \right)^{-1} , 
    \end{align}
    \begin{align}
      \label{eq-RS-rel-1-2-5}
      1
      &=\ga_2 \ga_3 \ga_1 (\ga_1 \ga_2 \ga_3)^{-1}
      \cdot \ga_1 \ga_2 \ga_3^2 (\ga_1 \ga_2)^{-1} \\
      \nonumber &\qquad
      \cdot \left( \ga_2 \ga_3^2 \ga_2^{-1} \cdot \ga_2 \ga_1 \ga_2^{-1} \ga_1^{-1} \right)^{-1}. 
    \end{align}
    \begin{align}
      \label{eq-RS-rel-1-2-6}
      1
      =\ga_1 \ga_2^2 \ga_1^{-1} 
      \cdot \left( \ga_1 \ga_2 \ga_3 \ga_2 (\ga_1 \ga_3)^{-1} \right)^{-1}, 
    \end{align}
    \begin{align}
      \label{eq-RS-rel-1-2-7}
      1
      =\ga_1 \ga_3 \ga_2 (\ga_1 \ga_2 \ga_3)^{-1}
      \cdot \ga_1 \ga_2 \ga_3^2 (\ga_1 \ga_2)^{-1}
      \cdot \left( \ga_1 \ga_3^2 \ga_1^{-1} \right)^{-1}. 
    \end{align}
  \item $t=\ga_1 \ga_2 \ga_3$. 
    \begin{align}
      \label{eq-RS-rel-1-3-1}
      1
      &=\ga_1 \ga_2 \ga_3 \ga_1 (\ga_2 \ga_3)^{-1}
      \cdot \ga_2 \ga_3 \ga_2 \ga_3^{-1} \\
      \nonumber & \qquad 
      \cdot \left( \ga_1 \ga_2 \ga_3 \ga_2 (\ga_1 \ga_3)^{-1}
        \cdot \ga_1 \ga_3 \ga_1 \ga_3^{-1} \right)^{-1}, 
    \end{align}
    \begin{align}
      \label{eq-RS-rel-1-3-2}
      1
      &=\ga_1 \ga_2 \ga_3 \ga_1 (\ga_2 \ga_3)^{-1}
      \cdot \ga_2 \ga_3^2 \ga_2^{-1} 
      \\ \nonumber & \qquad 
      \cdot \left( \ga_1 \ga_2 \ga_3^2 (\ga_1 \ga_2)^{-1}
        \cdot \ga_1 \ga_2 \ga_1 \ga_2^{-1} \right)^{-1}, 
    \end{align}
    \begin{align}
      \label{eq-RS-rel-1-3-3}
      1
      =\ga_1 \ga_2 \ga_3 \ga_2 (\ga_1 \ga_3)^{-1}
      \cdot \ga_1 \ga_3^2 \ga_1^{-1}
      \cdot \left( \ga_1 \ga_2 \ga_3^2 (\ga_1 \ga_2)^{-1}
        \cdot \ga_1 \ga_2^2 \ga_1^{-1} \right)^{-1}. 
    \end{align}
  \end{enumerate}
\item $r=\ga_i \ga_0 \ga_i^{-1} \ga_j \ga_0 \ga_j^{-1} 
  \ga_i \ga_0^{-1} \ga_i^{-1} \ga_j \ga_0^{-1} \ga_j^{-1}$
  ($1\leq i<j\leq 3$). 
  \begin{enumerate}[(a)]
  \item $t=1$. 
    \begin{align}
      \label{eq-RS-rel-2-0}
      \ga_i \ga_0 \ga_i^{-1} \cdot \ga_j \ga_0 \ga_j^{-1} 
      \cdot \ga_i \ga_0^{-1} \ga_i^{-1} \cdot \ga_j \ga_0^{-1} \ga_j^{-1}=1. 
    \end{align}
  \item $t=\ga_k$. 
    \begin{align}
      \label{eq-RS-rel-2-1-1}
      1
      &=\ga_i^2 \cdot \ga_0 \cdot \ga_i^{-2} 
      \cdot \ga_i \ga_j \ga_0 (\ga_i \ga_j)^{-1} 
      \\ \nonumber & \qquad 
      \cdot \ga_i^2 \cdot \ga_0^{-1} \cdot \ga_i^{-2}  
      \cdot \left( \ga_i \ga_j \ga_0 (\ga_i \ga_j)^{-1} \right)^{-1} .
    \end{align}
    \begin{align}
      \label{eq-RS-rel-2-1-2}
      1
      &=\ga_j \ga_i \ga_j^{-1} \ga_i^{-1} 
      \cdot \ga_i \ga_j \ga_0 (\ga_i \ga_j)^{-1} 
      \cdot (\ga_j \ga_i \ga_j^{-1} \ga_i^{-1})^{-1} 
      \cdot \ga_j^2 \cdot \ga_0 \cdot \ga_j^{-2} \\
      \nonumber
      &\qquad
      \cdot \ga_j \ga_i \ga_j^{-1} \ga_i^{-1}
      \cdot \left( \ga_i \ga_j \ga_0 (\ga_i \ga_j)^{-1} \right)^{-1} 
      \\ \nonumber & \qquad 
      \cdot (\ga_j \ga_i \ga_j^{-1} \ga_i^{-1})^{-1}
      \cdot \ga_j^2 \cdot \ga_0^{-1} \cdot \ga_j^{-2} . 
    \end{align}
    \begin{align}
      \label{eq-RS-rel-2-1-3}
      1
      &
      =\ga_3 \ga_1 \ga_3^{-1} \ga_1^{-1}
      \cdot \ga_1 \ga_3 \ga_0 (\ga_1 \ga_3)^{-1} 
      \cdot (\ga_3 \ga_1 \ga_3^{-1} \ga_1^{-1})^{-1}
      \cdot \ga_3 \ga_2 \ga_3^{-1} \ga_2^{-1} \\
      \nonumber
      &\qquad 
      \cdot \ga_2 \ga_3 \ga_0 (\ga_2 \ga_3)^{-1} 
      \cdot (\ga_3 \ga_2 \ga_3^{-1} \ga_2^{-1})^{-1}
      \cdot \ga_3 \ga_1 \ga_3^{-1} \ga_1^{-1} \\
      \nonumber
      &\qquad 
      \cdot \left( \ga_1 \ga_3 \ga_0 (\ga_1 \ga_3)^{-1} \right)^{-1} 
      \cdot (\ga_3 \ga_1 \ga_3^{-1} \ga_1^{-1})^{-1}
      \cdot \ga_3 \ga_2 \ga_3^{-1} \ga_2^{-1} \\
      \nonumber
      &\qquad 
      \cdot \left( \ga_2 \ga_3 \ga_0 (\ga_2 \ga_3)^{-1} \right)^{-1} 
      \cdot (\ga_3 \ga_2 \ga_3^{-1} \ga_2^{-1})^{-1},
    \end{align}
    \begin{align}
      \label{eq-RS-rel-2-1-4}
      1
      &=\ga_2 \ga_1 \ga_2^{-1} \ga_1^{-1}
      \cdot \ga_1 \ga_2 \ga_0 (\ga_1 \ga_2)^{-1} 
      \cdot (\ga_2 \ga_1 \ga_2^{-1} \ga_1^{-1})^{-1}
      \\ \nonumber & \qquad 
      \cdot \ga_2 \ga_3 \ga_0 (\ga_2 \ga_3)^{-1} 
      \cdot \ga_2 \ga_1 \ga_2^{-1} \ga_1^{-1}
      \cdot \left( \ga_1 \ga_2 \ga_0 (\ga_1 \ga_2)^{-1} \right)^{-1} 
      \\ \nonumber & \qquad 
      \cdot (\ga_2 \ga_1 \ga_2^{-1} \ga_1^{-1})^{-1}
      \cdot \left( \ga_2 \ga_3 \ga_0 (\ga_2 \ga_3)^{-1} \right)^{-1} ,
    \end{align}
    \begin{align}
      \label{eq-RS-rel-2-1-5}
      1
      &=\ga_1 \ga_2 \ga_0 (\ga_1 \ga_2)^{-1}
      \cdot \ga_1 \ga_3 \ga_0 (\ga_1 \ga_3)^{-1} 
      \\ \nonumber & \qquad 
      \cdot \left( \ga_1 \ga_2 \ga_0 (\ga_1 \ga_2)^{-1} \right)^{-1} 
      \cdot \left( \ga_1 \ga_3 \ga_0 (\ga_1 \ga_3)^{-1} \right)^{-1} .
    \end{align}
  \item $t=\ga_k \ga_l$. 
    \begin{align}
      \label{eq-RS-rel-2-2-1}
      1
      &=\ga_i \ga_j \ga_i \ga_j^{-1} 
      \cdot \ga_j \ga_0 \ga_j^{-1}
      \cdot (\ga_i \ga_j \ga_i \ga_j^{-1})^{-1}
      \cdot \ga_i \ga_j^2 \ga_i^{-1} \\
      \nonumber
      &\qquad
      \cdot \ga_i \ga_0 \ga_i^{-1}
      \cdot (\ga_i \ga_j^2 \ga_i^{-1})^{-1}
      \cdot \ga_i \ga_j \ga_i \ga_j^{-1} 
      \cdot (\ga_j \ga_0 \ga_j^{-1})^{-1} \\
      \nonumber
      &\qquad
      \cdot (\ga_i \ga_j \ga_i \ga_j^{-1})^{-1}
      \cdot \ga_i \ga_j^2 \ga_i^{-1}
      \cdot (\ga_i \ga_0 \ga_i^{-1})^{-1}
      \cdot (\ga_i \ga_j^2 \ga_i^{-1})^{-1}. 
    \end{align}
    \begin{align}
      \label{eq-RS-rel-2-2-2}
      1
      &=\ga_1 \ga_3 \ga_1 \ga_3^{-1}
      \cdot \ga_3 \ga_0 \ga_3^{-1} 
      \cdot (\ga_1 \ga_3 \ga_1 \ga_3^{-1})^{-1} 
      \cdot \ga_1 \ga_3 \ga_2 (\ga_1 \ga_2 \ga_3)^{-1}  \\
      \nonumber
      &\qquad
      \cdot \ga_1 \ga_2 \ga_3 \ga_0 (\ga_1 \ga_2 \ga_3)^{-1}
      \cdot \left( \ga_1 \ga_3 \ga_2 (\ga_1 \ga_2 \ga_3)^{-1} \right)^{-1}
      \cdot \ga_1 \ga_3 \ga_1 \ga_3^{-1}\\
      \nonumber
      &\qquad
      \cdot (\ga_3 \ga_0 \ga_3^{-1} )^{-1}
      \cdot (\ga_1 \ga_3 \ga_1 \ga_3^{-1})^{-1}
      \cdot \ga_1 \ga_3 \ga_2 (\ga_1 \ga_2 \ga_3)^{-1} \\
      \nonumber
      &\qquad
      \cdot \left( \ga_1 \ga_2 \ga_3 \ga_0 (\ga_1 \ga_2 \ga_3)^{-1} \right)^{-1}
      \cdot \left( \ga_1 \ga_3 \ga_2 (\ga_1 \ga_2 \ga_3)^{-1} \right)^{-1} , 
    \end{align}
    \begin{align}
      \label{eq-RS-rel-2-2-3}
      1
      &=\ga_2 \ga_3 \ga_1 (\ga_1 \ga_2 \ga_3)^{-1}
      \cdot \ga_1 \ga_2 \ga_3 \ga_0 (\ga_1 \ga_2 \ga_3)^{-1}
      \\ \nonumber & \qquad 
      \cdot \left( \ga_2 \ga_3 \ga_1 (\ga_1 \ga_2 \ga_3)^{-1} \right)^{-1} 
      \cdot \ga_2 \ga_3 \ga_2 \ga_3^{-1}
      \cdot \ga_3 \ga_0 \ga_3^{-1} 
      \\ \nonumber & \qquad 
      \cdot (\ga_2 \ga_3 \ga_2 \ga_3^{-1})^{-1} 
      \cdot \ga_2 \ga_3 \ga_1 (\ga_1 \ga_2 \ga_3)^{-1}
      \\ \nonumber & \qquad 
      \cdot \left( \ga_1 \ga_2 \ga_3 \ga_0 (\ga_1 \ga_2 \ga_3)^{-1} \right)^{-1}
      \cdot \left( \ga_2 \ga_3 \ga_1 (\ga_1 \ga_2 \ga_3)^{-1} \right)^{-1} 
      \\ \nonumber & \qquad 
      \cdot \ga_2 \ga_3 \ga_2 \ga_3^{-1}
      \cdot (\ga_3 \ga_0 \ga_3^{-1} )^{-1}
      \cdot (\ga_2 \ga_3 \ga_2 \ga_3^{-1})^{-1} .
    \end{align}
    \begin{align}
      \label{eq-RS-rel-2-2-4}
      1
      &=\ga_1 \ga_2 \ga_1 \ga_2^{-1} 
      \cdot \ga_2 \ga_0 \ga_2^{-1} 
      \cdot (\ga_1 \ga_2 \ga_1 \ga_2^{-1})^{-1}
      \cdot \ga_1 \ga_2 \ga_3 \ga_0 (\ga_1 \ga_2 \ga_3)^{-1} \\
      \nonumber
      &\qquad
      \cdot \ga_1 \ga_2 \ga_1 \ga_2^{-1} 
      \cdot (\ga_2 \ga_0 \ga_2^{-1} )^{-1}
      \cdot (\ga_1 \ga_2 \ga_1 \ga_2^{-1})^{-1}\\
      \nonumber
      &\qquad
      \cdot \left( \ga_1 \ga_2 \ga_3 \ga_0 (\ga_1 \ga_2 \ga_3)^{-1} \right)^{-1} ,
    \end{align}
    \begin{align}
      \label{eq-RS-rel-2-2-5}
      1
      &=\ga_2 \ga_3 \ga_1 (\ga_1 \ga_2 \ga_3)^{-1}
      \cdot \ga_1 \ga_2 \ga_3 \ga_0 (\ga_1 \ga_2 \ga_3)^{-1} 
      \\ \nonumber & \qquad 
      \cdot \left( \ga_2 \ga_3 \ga_1 (\ga_1 \ga_2 \ga_3)^{-1} \right)^{-1} 
      \cdot \ga_2 \ga_3^2 \ga_2^{-1}
      \cdot \ga_2 \ga_0 \ga_2^{-1} 
      \\ \nonumber & \qquad 
      \cdot (\ga_2 \ga_3^2 \ga_2^{-1})^{-1}
      \cdot \ga_2 \ga_3 \ga_1 (\ga_1 \ga_2 \ga_3)^{-1} 
      \\ \nonumber & \qquad 
      \cdot \left( \ga_1 \ga_2 \ga_3 \ga_0 (\ga_1 \ga_2 \ga_3)^{-1} \right)^{-1}
      \cdot \left( \ga_2 \ga_3 \ga_1 (\ga_1 \ga_2 \ga_3)^{-1} \right)^{-1} 
      \\ \nonumber & \qquad 
      \cdot \ga_2 \ga_3^2 \ga_2^{-1}
      \cdot (\ga_2 \ga_0 \ga_2^{-1} )^{-1}
      \cdot (\ga_2 \ga_3^2 \ga_2^{-1})^{-1} . 
    \end{align}
    \begin{align}
      \label{eq-RS-rel-2-2-6}
      1
      &=\ga_1 \ga_2^2 \ga_1^{-1} 
      \cdot \ga_1 \ga_0 \ga_1^{-1} 
      \cdot (\ga_1 \ga_2^2 \ga_1^{-1})^{-1}
      \\ \nonumber & \qquad 
      \cdot \ga_1 \ga_2 \ga_3 \ga_0 (\ga_1 \ga_2 \ga_3)^{-1}  
      \cdot \ga_1 \ga_2^2 \ga_1^{-1} 
      \cdot (\ga_1 \ga_0 \ga_1^{-1} )^{-1}
      \\ \nonumber & \qquad 
      \cdot (\ga_1 \ga_2^2 \ga_1^{-1})^{-1}
      \cdot \left( \ga_1 \ga_2 \ga_3 \ga_0 (\ga_1 \ga_2 \ga_3)^{-1} \right)^{-1} ,
    \end{align}
    \begin{align}
      \label{eq-RS-rel-2-2-7}
      1
      &=\ga_1 \ga_3 \ga_2 (\ga_1 \ga_2 \ga_3)^{-1}
      \cdot \ga_1 \ga_2 \ga_3 \ga_0 (\ga_1 \ga_2 \ga_3)^{-1}
      \\ \nonumber & \qquad 
      \cdot \left( \ga_1 \ga_3 \ga_2 (\ga_1 \ga_2 \ga_3)^{-1} \right)^{-1} 
      \cdot \ga_1 \ga_3^2 \ga_1^{-1} 
      \cdot \ga_1 \ga_0 \ga_1^{-1} 
      \\ \nonumber & \qquad 
      \cdot (\ga_1 \ga_3^2 \ga_1^{-1})^{-1}
      \cdot \ga_1 \ga_3 \ga_2 (\ga_1 \ga_2 \ga_3)^{-1} 
      \\ \nonumber & \qquad 
      \cdot \left( \ga_1 \ga_2 \ga_3 \ga_0 (\ga_1 \ga_2 \ga_3)^{-1} \right)^{-1}
      \cdot \left( \ga_1 \ga_3 \ga_2 (\ga_1 \ga_2 \ga_3)^{-1} \right)^{-1}  
      \\ \nonumber & \qquad 
      \cdot \ga_1 \ga_3^2 \ga_1^{-1} 
      \cdot (\ga_1 \ga_0 \ga_1^{-1} )^{-1}
      \cdot (\ga_1 \ga_3^2 \ga_1^{-1})^{-1} .
    \end{align}
  \item $t=\ga_1 \ga_2 \ga_3$. 
    \begin{align}
      \label{eq-RS-rel-2-3-1}
      1
      &=\ga_1 \ga_2 \ga_3 \ga_1 (\ga_2 \ga_3)^{-1}
      \cdot \ga_2 \ga_3 \ga_0 (\ga_2 \ga_3)^{-1} 
      \cdot \left( \ga_1 \ga_2 \ga_3 \ga_1 (\ga_2 \ga_3)^{-1} \right)^{-1} 
      \\ \nonumber & \qquad 
      \cdot \ga_1 \ga_2 \ga_3 \ga_2 (\ga_1 \ga_3)^{-1}
      \cdot \ga_1 \ga_3 \ga_0 (\ga_1 \ga_3)^{-1} 
      \\ \nonumber & \qquad 
      \cdot \left( \ga_1 \ga_2 \ga_3 \ga_2 (\ga_1 \ga_3)^{-1} \right)^{-1} 
      \cdot \ga_1 \ga_2 \ga_3 \ga_1 (\ga_2 \ga_3)^{-1}
      \\ \nonumber & \qquad 
      \cdot \left( \ga_2 \ga_3 \ga_0 (\ga_2 \ga_3)^{-1} \right)^{-1} 
      \cdot \left( \ga_1 \ga_2 \ga_3 \ga_1 (\ga_2 \ga_3)^{-1} \right)^{-1} 
      \\ \nonumber & \qquad 
      \cdot \ga_1 \ga_2 \ga_3 \ga_2 (\ga_1 \ga_3)^{-1}
      \cdot \left( \ga_1 \ga_3 \ga_0 (\ga_1 \ga_3)^{-1} \right)^{-1} 
      \\ \nonumber & \qquad 
      \cdot \left( \ga_1 \ga_2 \ga_3 \ga_2 (\ga_1 \ga_3)^{-1} \right)^{-1} ,
    \end{align}
    \begin{align}
      \label{eq-RS-rel-2-3-2}
      1
      &=\ga_1 \ga_2 \ga_3 \ga_1 (\ga_2 \ga_3)^{-1}
      \cdot \ga_2 \ga_3 \ga_0 (\ga_2 \ga_3)^{-1} 
      \cdot \left( \ga_1 \ga_2 \ga_3 \ga_1 (\ga_2 \ga_3)^{-1} \right)^{-1} \\
      \nonumber
      &\qquad 
      \cdot \ga_1 \ga_2 \ga_3^2 (\ga_1 \ga_2)^{-1}
      \cdot \ga_1 \ga_2 \ga_0 (\ga_1 \ga_2)^{-1} 
      \cdot \left( \ga_1 \ga_2 \ga_3^2 (\ga_1 \ga_2)^{-1} \right)^{-1} \\
      \nonumber
      &\qquad 
      \cdot \ga_1 \ga_2 \ga_3 \ga_1 (\ga_2 \ga_3)^{-1}
      \cdot \left( \ga_2 \ga_3 \ga_0 (\ga_2 \ga_3)^{-1} \right)^{-1} 
      \\ \nonumber & \qquad 
      \cdot \left( \ga_1 \ga_2 \ga_3 \ga_1 (\ga_2 \ga_3)^{-1} \right)^{-1} 
      \cdot \ga_1 \ga_2 \ga_3^2 (\ga_1 \ga_2)^{-1}
      \\ \nonumber & \qquad 
      \cdot \left( \ga_1 \ga_2 \ga_0 (\ga_1 \ga_2)^{-1} \right)^{-1} 
      \cdot \left( \ga_1 \ga_2 \ga_3^2 (\ga_1 \ga_2)^{-1} \right)^{-1} ,
    \end{align}
    \begin{align}
      \label{eq-RS-rel-2-3-3}
      1
      &=\ga_1 \ga_2 \ga_3 \ga_2 (\ga_1 \ga_3)^{-1}
      \cdot \ga_1 \ga_3 \ga_0 (\ga_1 \ga_3)^{-1} 
      \cdot \left( \ga_1 \ga_2 \ga_3 \ga_2 (\ga_1 \ga_3)^{-1} \right)^{-1} \\
      \nonumber
      &\qquad 
      \cdot \ga_1 \ga_2 \ga_3^2 (\ga_1 \ga_2)^{-1}
      \cdot \ga_1 \ga_2 \ga_0 (\ga_1 \ga_2)^{-1} 
      \cdot \left( \ga_1 \ga_2 \ga_3^2 (\ga_1 \ga_2)^{-1} \right)^{-1} \\
      \nonumber
      &\qquad 
      \cdot \ga_1 \ga_2 \ga_3 \ga_2 (\ga_1 \ga_3)^{-1}
      \cdot \left( \ga_1 \ga_3 \ga_0 (\ga_1 \ga_3)^{-1} \right)^{-1} 
      \\ \nonumber & \qquad 
      \cdot \left( \ga_1 \ga_2 \ga_3 \ga_2 (\ga_1 \ga_3)^{-1} \right)^{-1} 
      \cdot \ga_1 \ga_2 \ga_3^2 (\ga_1 \ga_2)^{-1}
      \\ \nonumber & \qquad 
      \cdot \left( \ga_1 \ga_2 \ga_0 (\ga_1 \ga_2)^{-1} \right)^{-1} 
      \cdot \left( \ga_1 \ga_2 \ga_3^2 (\ga_1 \ga_2)^{-1} \right)^{-1} .
    \end{align}
  \end{enumerate}
\item $r=\ga_i \ga_0 \ga_i \ga_0 \ga_i^{-1} \ga_0^{-1} \ga_i^{-1} \ga_0^{-1}$
  ($1\leq i \leq 3$). 
  \begin{enumerate}[(a)]
  \item $t=1$. 
    \begin{align}
      \label{eq-RS-rel-3-0}
      1
      =\ga_i \ga_0 \ga_i^{-1} \cdot \ga_i^2 \cdot \ga_0 \cdot \ga_i^{-2} \cdot 
      (\ga_i \ga_0 \ga_i^{-1})^{-1} \cdot \ga_0^{-1} .
    \end{align}
  \item $t=\ga_k$. 
    \begin{align}
      \label{eq-RS-rel-3-1-1}
      1
      &=\ga_k \ga_i \ga_0 (\ga_k \ga_i)^{-1}
      \cdot \ga_k \ga_i^2 \ga_k^{-1}
      \cdot \ga_k \ga_0 \ga_k^{-1} 
      \\
      \nonumber
      &\qquad
      \cdot \left( \ga_k \ga_0 \ga_k^{-1}
        \cdot \ga_k \ga_i \ga_0 (\ga_k \ga_i)^{-1}
        \cdot \ga_k \ga_i^2 \ga_k^{-1} \right)^{-1} 
      \quad (k<i). 
    \end{align}
    \begin{align}
      \label{eq-RS-rel-3-1-2}
      1
      =\ga_i^2 \cdot \ga_0 
      \cdot \ga_i \ga_0 \ga_i^{-1} 
      \cdot \ga_0^{-1} \cdot \ga_i^{-2}
      \cdot (\ga_i \ga_0 \ga_i^{-1})^{-1} .
    \end{align}
    \begin{align}
      \label{eq-RS-rel-3-1-3}
      1
      &=\ga_k \ga_i \ga_k^{-1} \ga_i^{-1}
      \cdot \ga_i \ga_k \ga_0 (\ga_i \ga_k)^{-1}
      \cdot \ga_i \ga_k \ga_i \ga_k^{-1}
      \cdot \ga_k \ga_0 \ga_k^{-1} \\
      \nonumber
      &\qquad
      \cdot \big( \ga_k \ga_0 \ga_k^{-1}
        \cdot \ga_k \ga_i \ga_k^{-1} \ga_i^{-1}
        \cdot \ga_i \ga_k \ga_0 (\ga_i \ga_k)^{-1}
        \cdot \ga_i \ga_k \ga_i \ga_k^{-1} \big)^{-1}  
      \\ \nonumber  & \qquad
      \phantom{\ga_k \ga_0 \ga_k^{-1}
        \cdot \ga_k \ga_i \ga_k^{-1} \ga_i^{-1}
        \cdot \ga_i \ga_k \ga_0 (\ga_i \ga_k)^{-1}
        \cdot \ga_i \ga_k}
      \quad (k>i). 
    \end{align}
  \item $t=\ga_k \ga_l$. 
    \begin{align}
      \label{eq-RS-rel-3-2-1}
      1
      &=\ga_i \ga_l \ga_i \ga_l^{-1}
      \cdot \ga_l \ga_0 \ga_l^{-1} 
      \cdot \ga_l \ga_i \ga_l^{-1} \ga_i^{-1}
      \cdot \ga_i \ga_l \ga_0 (\ga_i \ga_l)^{-1} \\
      \nonumber
      &\qquad
      \cdot \left( \ga_i \ga_l \ga_0 (\ga_i \ga_l)^{-1}
        \cdot \ga_i \ga_l \ga_i \ga_l^{-1}
        \cdot \ga_l \ga_0 \ga_l^{-1} 
        \cdot \ga_l \ga_i \ga_l^{-1} \ga_i^{-1}
      \right)^{-1}  
      \\ \nonumber  & \qquad
      \phantom{\ga_i \ga_l \ga_0 (\ga_i \ga_l)^{-1}
        \cdot \ga_i \ga_l \ga_i \ga_l^{-1}
        \cdot \ga_l \ga_0 \ga_l^{-1} 
        \cdot \ga_l \ga_i \ga_l^{-1}}
      \quad (i<l). 
    \end{align}
    \begin{align}
      \label{eq-RS-rel-3-2-2}
      1
      &=\ga_k \ga_i^2 \ga_k^{-1}
      \cdot \ga_k \ga_0 \ga_k^{-1} 
      \cdot \ga_k \ga_i \ga_0 (\ga_k \ga_i)^{-1} \\
      \nonumber
      &\qquad
      \cdot \left( \ga_k \ga_i \ga_0 (\ga_k \ga_i)^{-1}
        \cdot \ga_k \ga_i^2 \ga_k^{-1}
        \cdot \ga_k \ga_0 \ga_k^{-1} 
      \right)^{-1}  
      \quad (k<i). 
    \end{align}
    \begin{align}
      \label{eq-RS-rel-3-2-3}
      1
      &=\ga_2 \ga_3 \ga_1 (\ga_1 \ga_2 \ga_3)^{-1}
      \cdot \ga_1 \ga_2 \ga_3 \ga_0 (\ga_1 \ga_2 \ga_3)^{-1}
      \\ \nonumber  & \qquad
      \cdot \ga_1 \ga_2 \ga_3 \ga_1 (\ga_2 \ga_3)^{-1}
      \cdot \ga_2 \ga_3 \ga_0 (\ga_2 \ga_3)^{-1}  
      \\ \nonumber  & \qquad
      \cdot \big( 
        \ga_2 \ga_3 \ga_0 (\ga_2 \ga_3)^{-1} 
        \cdot \ga_2 \ga_3 \ga_1 (\ga_1 \ga_2 \ga_3)^{-1}  
      \\ \nonumber  & \qquad \qquad
        \cdot \ga_1 \ga_2 \ga_3 \ga_0 (\ga_1 \ga_2 \ga_3)^{-1} 
        \cdot \ga_1 \ga_2 \ga_3 \ga_1 (\ga_2 \ga_3)^{-1} 
      \big)^{-1} ,
    \end{align}
    \begin{align}
      \label{eq-RS-rel-3-2-4}
      1
      &=\ga_1 \ga_3 \ga_2 (\ga_1 \ga_2 \ga_3)^{-1}
      \cdot \ga_1 \ga_2 \ga_3 \ga_0 (\ga_1 \ga_2 \ga_3)^{-1}
      \\ \nonumber  & \qquad
      \cdot \ga_1 \ga_2 \ga_3 \ga_2 (\ga_1 \ga_3)^{-1}
      \cdot \ga_1 \ga_3 \ga_0 (\ga_1 \ga_3)^{-1}  
      \\ \nonumber  & \qquad
      \cdot \big( 
        \ga_1 \ga_3 \ga_0 (\ga_1 \ga_3)^{-1} 
        \cdot \ga_1 \ga_3 \ga_2 (\ga_1 \ga_2 \ga_3)^{-1}  
      \\ \nonumber  & \qquad \qquad 
        \cdot \ga_1 \ga_2 \ga_3 \ga_0 (\ga_1 \ga_2 \ga_3)^{-1} 
        \cdot \ga_1 \ga_2 \ga_3 \ga_2 (\ga_1 \ga_3)^{-1} 
      \big)^{-1} ,
    \end{align}
    \begin{align}
      \label{eq-RS-rel-3-2-5}
      1
      &=\ga_1 \ga_2 \ga_3 \ga_0 (\ga_1 \ga_2 \ga_3)^{-1}
      \cdot \ga_1 \ga_2 \ga_3^2 (\ga_1 \ga_2)^{-1}
      \cdot \ga_1 \ga_2 \ga_0 (\ga_1 \ga_2)^{-1}  \\
      \nonumber
      &\qquad 
      \cdot \big( 
        \ga_1 \ga_2 \ga_0 (\ga_1 \ga_2)^{-1} 
        \cdot \ga_1 \ga_2 \ga_3 \ga_0 (\ga_1 \ga_2 \ga_3)^{-1} 
      \\ \nonumber  & \qquad \qquad 
        \cdot \ga_1 \ga_2 \ga_3^2 (\ga_1 \ga_2)^{-1} 
      \big)^{-1} .
    \end{align}
  \item $t=\ga_1 \ga_2 \ga_3$. 
    We take $k,l$ such that $\{ i,k,l \} =\{ 1,2,3 \}$ and $k<l$. 
    Then we obtain a relation
    \begin{align}
      \label{eq-RS-rel-3-3}
      1
      &=\ga_1 \ga_2 \ga_3 \ga_i (\ga_k \ga_l)^{-1}
      \cdot \ga_k \ga_l \ga_0 (\ga_k \ga_l)^{-1}
      \\ \nonumber  & \qquad
      \cdot \ga_k \ga_l \ga_i (\ga_1 \ga_2 \ga_3)^{-1}
      \cdot \ga_1 \ga_2 \ga_3 \ga_0 (\ga_1 \ga_2 \ga_3)^{-1}  \\
      \nonumber
      &\qquad 
      \cdot \big( 
        \ga_1 \ga_2 \ga_3 \ga_0 (\ga_1 \ga_2 \ga_3)^{-1}
        \cdot \ga_1 \ga_2 \ga_3 \ga_i (\ga_k \ga_l)^{-1}
      \\ \nonumber  & \qquad \qquad 
        \cdot \ga_k \ga_l \ga_0 (\ga_k \ga_l)^{-1}
        \cdot \ga_k \ga_l \ga_i (\ga_1 \ga_2 \ga_3)^{-1}
      \big)^{-1} .
    \end{align}
  \end{enumerate}
\end{enumerate}
Therefore, $G_1$ is presented by  
the 25 generators (\ref{eq-25-generator}) and 
the 72 relations (\ref{eq-RS-rel-1-0})--(\ref{eq-RS-rel-3-3}). 
We reduce this presentation to a simpler one. 
\begin{Cor}
  The 11 elements   
  \begin{align}
    \label{eq-11-generator}
    & \ga_0 ,\ \ga_1^2 , \ \ga_2^2 , \ \ga_3^2 , \ 
    \ga_1  \ga_0 \ga_1^{-1} ,\ \ga_2 \ga_0 \ga_2^{-1} ,\ \ga_3 \ga_0 \ga_3^{-1} ,\\
    \nonumber
    &\ga_1 \ga_2 \ga_0 (\ga_1 \ga_2)^{-1},\ 
    \ga_1 \ga_3 \ga_0 (\ga_1 \ga_3)^{-1}, \
    \ga_2 \ga_3 \ga_0 (\ga_2 \ga_3)^{-1},\\
    \nonumber
    &\ga_1 \ga_2 \ga_3 \ga_0 (\ga_1 \ga_2 \ga_3)^{-1} 
  \end{align}
  in (\ref{eq-25-generator}) generate $G_1$. 
\end{Cor}
\begin{proof}
  To prove this, it suffices to show the following relations in $G_1$:
  \begin{align}
    \label{eq-cor-generator1}
    & \ga_j \ga_i \ga_j^{-1} \ga_i^{-1}=1 ,\\
    \label{eq-cor-generator2}
    & \ga_i \ga_j \ga_i \ga_j^{-1}=\ga_i^2 ,\\ 
    \label{eq-cor-generator3}
    & \ga_i \ga_j^2 \ga_i^{-1} =\ga_j^2, \\
    \label{eq-cor-generator4}
    & \ga_1 \ga_3 \ga_2 (\ga_1 \ga_2 \ga_3)^{-1} =1, \ 
    \ga_2 \ga_3 \ga_1 (\ga_1 \ga_2 \ga_3)^{-1} =1, \\
    \label{eq-cor-generator5}
    & \ga_1 \ga_2 \ga_3 \ga_1 (\ga_2 \ga_3)^{-1} =\ga_1^2,\\
    \nonumber
    &\ga_1 \ga_2 \ga_3 \ga_2 (\ga_1 \ga_3)^{-1}=\ga_2^2,\ 
    \ga_1 \ga_2 \ga_3^2 (\ga_1 \ga_2)^{-1}=\ga_3^2 , 
  \end{align}
  where $1\leq i < j \leq 3$. 
  (\ref{eq-cor-generator1}) is same as (\ref{eq-RS-rel-1-0}). 
  (\ref{eq-cor-generator2}) is equivalent to (\ref{eq-RS-rel-1-1-1}). 
  (\ref{eq-cor-generator3}) follows from (\ref{eq-RS-rel-1-1-2}) and (\ref{eq-cor-generator1}). 
  The first relation of (\ref{eq-cor-generator4}) is equivalent to (\ref{eq-RS-rel-1-1-5}), 
  and the second one follows from (\ref{eq-RS-rel-1-1-4}) and (\ref{eq-cor-generator1}).
  Three relations (\ref{eq-cor-generator5}) follow from 
  (\ref{eq-RS-rel-1-2-4}), (\ref{eq-RS-rel-1-2-6}), (\ref{eq-RS-rel-1-2-7}) 
  and above relations. 
\end{proof}
Under the inclusion $\phi_* :G_1 \to G$,  
the 11 loops (\ref{eq-11-loops}) coincide with the generator (\ref{eq-11-generator})
of $G_1$, 
so we use the notations
\begin{align*}
  & \ga_i^2 =\la_i ,\quad \ga_0 =\la_0 ,\quad 
  \ga_i \ga_0 \ga_i^{-1}=\la_0^{(i)} ,\\
  &\ga_j \ga_k \ga_0 (\ga_j \ga_k)^{-1}=\la_0^{(jk)},\quad 
  \ga_1 \ga_2 \ga_3 \ga_0 (\ga_1 \ga_2 \ga_3)^{-1} =\la_0^{(123)},
\end{align*}
where $1\leq i \leq 3$ and $1\leq j<k \leq 3$. 

By using these generators and relations (\ref{eq-cor-generator1})--(\ref{eq-cor-generator5}), 
we rewrite the relations (\ref{eq-RS-rel-1-0})--(\ref{eq-RS-rel-3-3}). 
The relations (\ref{eq-RS-rel-1-0})--(\ref{eq-RS-rel-1-1-5}) and 
(\ref{eq-RS-rel-1-2-2})--(\ref{eq-RS-rel-1-2-7}) become trivial ones. 
(\ref{eq-RS-rel-1-2-1}) implies
\begin{align*}
  [\la_i ,\la_j]=1 \quad (1\leq i < j \leq 3), 
\end{align*}
which is equivalent to (\ref{eq-RS-rel-1-3-1})--(\ref{eq-RS-rel-1-3-3}). 
(\ref{eq-RS-rel-2-0}) implies 
\begin{align*}
  [\la_0^{(i)} ,\la_0^{(j)}]=1 \quad (1\leq i < j \leq 3). 
\end{align*}
(\ref{eq-RS-rel-2-1-1}) and (\ref{eq-RS-rel-2-1-2}) imply 
\begin{align}
  \label{eq-H-rel4-all}
  [\la_0^{(ij)} , \la_i \la_0 \la_i^{-1}]=1 , \ 
  [\la_0^{(ij)} , \la_j \la_0 \la_j^{-1}]=1 \quad (1\leq i < j \leq 3) .
\end{align}
(\ref{eq-RS-rel-2-1-3})--(\ref{eq-RS-rel-2-1-5}) imply 
\begin{align*}
  [\la_0^{(12)} ,\la_0^{(13)}]=1 ,\ 
  [\la_0^{(12)} ,\la_0^{(23)}]=1 ,\ 
  [\la_0^{(13)} ,\la_0^{(23)}]=1 .
\end{align*}
(\ref{eq-RS-rel-2-2-1}) implies 
\begin{align}
  \label{eq-H-rel6}
  [\la_i \la_0^{(j)} \la_i^{-1} , \la_j \la_0^{(i)} \la_j^{-1}]=1 
  \quad (1\leq i < j \leq 3) .  
\end{align}
(\ref{eq-RS-rel-2-2-2})--(\ref{eq-RS-rel-2-2-7}) imply
\begin{align}
  \label{eq-H-rel5-all}
  [\la_0^{(123)} , \la_i \la_0^{(j)} \la_i^{-1}]=1 
  \quad (1\leq i \neq j \leq 3) .    
\end{align}
(\ref{eq-RS-rel-2-3-1})--(\ref{eq-RS-rel-2-3-3}) imply
\begin{align}
  \label{eq-H-rel7}
  &[\la_1 \la_0^{(23)} \la_1^{-1} , \la_2 \la_0^{(13)} \la_2^{-1}]=1, \\ 
  \nonumber
  &[\la_1 \la_0^{(23)} \la_1^{-1} , \la_3 \la_0^{(12)} \la_3^{-1}]=1, \ 
  [\la_2 \la_0^{(13)} \la_2^{-1} , \la_3 \la_0^{(12)} \la_3^{-1}]=1.
\end{align}
(\ref{eq-RS-rel-3-0}) and (\ref{eq-RS-rel-3-1-2}) imply 
\begin{align*}
  \la_0^{(i)}\la_i\la_0=\la_0\la_0^{(i)}\la_i=\la_i\la_0\la_0^{(i)}  \quad (1\leq i \leq 3). 
\end{align*}
(\ref{eq-RS-rel-3-1-1}) and (\ref{eq-RS-rel-3-2-2}) imply
\begin{align*}
  \la_j\la_0^{(i)}\la_0^{(ij)}=\la_0^{(i)}\la_0^{(ij)}\la_j=\la_0^{(ij)}\la_j\la_0^{(i)} 
  \quad (1\leq i < j \leq 3) . 
\end{align*}
(\ref{eq-RS-rel-3-1-3}) and (\ref{eq-RS-rel-3-2-1}) imply
\begin{align*}
  \la_i\la_0^{(j)}\la_0^{(ij)}=\la_0^{(j)}\la_0^{(ij)}\la_i=\la_0^{(ij)}\la_i\la_0^{(j)} 
  \quad (1\leq i < j \leq 3) .
\end{align*}
(\ref{eq-RS-rel-3-2-3})--(\ref{eq-RS-rel-3-3}) imply 
\begin{align*}
  \la_i\la_0^{(jk)}\la_0^{(123)}=\la_0^{(jk)}\la_0^{(123)}\la_i=\la_0^{(123)}\la_i\la_0^{(jk)} ,
\end{align*}
where $\{ i,j,k\}=\{ 1,2,3\}$ and $j<k$.
\begin{Th}
  The fundamental group 
  $\pi_1 (\tilde{X}^{(3)}) =G_1$ has a presentation by 11 generators
  \begin{align*}
    \la_1,\la_2,\la_3,\la_0, 
    \la_0^{(1)},\la_0^{(2)},\la_0^{(3)},  
    \la_0^{(12)},\la_0^{(13)},\la_0^{(23)}, \la_0^{(123)} ,
  \end{align*}
  and 27 defining relations
  \begin{align}
    \label{eq-H-rel1}
    &[\la_i ,\la_j]=1 \quad (1\leq i < j \leq 3),  \\
    \label{eq-H-rel2}
    &[\la_0^{(i)} ,\la_0^{(j)}]=1 \quad (1\leq i < j \leq 3),  \\
    \label{eq-H-rel3}
    &[\la_0^{(12)} ,\la_0^{(13)}]=1 ,\ 
    [\la_0^{(12)} ,\la_0^{(23)}]=1 ,\ 
    [\la_0^{(13)} ,\la_0^{(23)}]=1 ,  \\
    \label{eq-H-rel4}
    &[\la_0^{(ij)} , \la_i \la_0 \la_i^{-1}]=1 \quad (1\leq i < j \leq 3),  \\
    \label{eq-H-rel5}
    &[\la_0^{(123)} , \la_1 \la_0^{(2)} \la_1^{-1}]=1, \\ 
    \nonumber
    &[\la_0^{(123)} , \la_2 \la_0^{(3)} \la_2^{-1}]=1, \ 
    [\la_0^{(123)} , \la_3 \la_0^{(1)} \la_3^{-1}]=1, \\
    \label{eq-H-rel8}
    &\la_0^{(i)}\la_i\la_0=\la_0\la_0^{(i)}\la_i=\la_i\la_0\la_0^{(i)}  \quad (1\leq i \leq 3), \\
    \nonumber
    &\la_i\la_0^{(j)}\la_0^{(ij)}=\la_0^{(j)}\la_0^{(ij)}\la_i=\la_0^{(ij)}\la_i\la_0^{(j)} 
    \quad (1\leq i < j \leq 3), \\
    \nonumber
    &\la_j\la_0^{(i)}\la_0^{(ij)}=\la_0^{(i)}\la_0^{(ij)}\la_j=\la_0^{(ij)}\la_j\la_0^{(i)} 
    \quad (1\leq i < j \leq 3), \\
    \nonumber
    &\la_i\la_0^{(jk)}\la_0^{(123)}=\la_0^{(jk)}\la_0^{(123)}\la_i=\la_0^{(123)}\la_i\la_0^{(jk)} \\
    \nonumber
    &\qquad (\{ i,j,k\}=\{ 1,2,3\}, \ j<k).
  \end{align}
\end{Th}
\begin{proof}
  We need to show that the relations 
  \begin{itemize}
  \item the second relation of (\ref{eq-H-rel4-all}), 
  \item (\ref{eq-H-rel6}), 
  \item (\ref{eq-H-rel5-all}) for $(i,j)=(3,2),(1,3),(2,1)$, and 
  \item (\ref{eq-H-rel7}) 
  \end{itemize}
  follow from (\ref{eq-H-rel1})--(\ref{eq-H-rel8}). 
  We only consider the second relation of (\ref{eq-H-rel4-all}), 
  since the others are also proved similarly. 
  By (\ref{eq-H-rel8}), we have 
  \begin{align*}
    \la_0^{(i)}\la_i\la_0=\la_0\la_0^{(i)}\la_i ,\quad 
    \la_j\la_0^{(i)}\la_0^{(ij)}=\la_0^{(ij)}\la_j\la_0^{(i)} .
  \end{align*}
  Then (\ref{eq-H-rel4}) implies
  \begin{align*}
    &[\la_0^{(ij)}, \la_j \la_0 \la_j^{-1}]
    =\la_j \cdot [ \la_j^{-1} \la_0^{(ij)}\la_j, \la_0] \cdot \la_j^{-1} \\
    &=\la_j \cdot [\la_0^{(i)} \la_0^{(ij)} {\la_0^{(i)}}^{-1}, \la_0] \cdot \la_j^{-1} 
    =\la_j \la_0^{(i)} \cdot [\la_0^{(ij)}, {\la_0^{(i)}}^{-1} \la_0\la_0^{(i)}] \cdot (\la_j \la_0^{(i)})^{-1} \\
    &=\la_j \la_0^{(i)} \cdot [\la_0^{(ij)}, \la_i \la_0 \la_i^{-1}] \cdot (\la_j \la_0^{(i)})^{-1}
    =\la_j \la_0^{(i)} \cdot 1 \cdot (\la_j \la_0^{(i)})^{-1} =1. 
  \end{align*}
\end{proof}
\begin{Rem}
  We can interpret that these relations come from lines which are
  intersections of the planes $H_k$, $H(a_1,a_2,a_3)$,
  as Table \ref{table:intersection}. 
  For example, the loop $\la_0^{(12)}$ turns 
  the hyperplane $H(1,1,0)$. 
  \begin{table}
    \begin{tabular}{|l|l|}
      \hline
      (\ref{eq-H-rel1})
      & $H_1\cap H_2$, $H_1\cap H_3$, $H_2\cap H_3$ \\ \hline
      (\ref{eq-H-rel2})
      & $H(1,0,0) \cap H(0,1,0)$, $H(1,0,0) \cap H(0,0,1)$, $H(0,1,0) \cap H(0,0,1)$  \\ \hline  
      (\ref{eq-H-rel3})
      & $H(1,1,0) \cap H(1,0,1)$, $H(0,1,1) \cap H(1,1,0)$, $H(1,0,1) \cap H(0,1,1)$  \\ \hline  
      (\ref{eq-H-rel4})
      & $H(0,0,0) \cap H(1,1,0)$, $H(0,0,0) \cap H(1,0,1)$, $H(0,0,0) \cap H(0,1,1)$ \\ \hline  
      (\ref{eq-H-rel5})
      & $H(1,1,1) \cap H(0,1,0)$, $H(1,1,1) \cap H(0,0,1)$, $H(1,1,1) \cap H(1,0,0)$  \\ \hline  
      (\ref{eq-H-rel8})
      & $H(0,0,0)\cap H(1,0,0) \cap H_1$, 
      $H(0,0,0)\cap H(0,1,0) \cap H_2$, \\
      &$H(0,0,0)\cap H(0,0,1) \cap H_3$  \\ \cline{2-2}
      & $H(0,1,0)\cap H(1,1,0) \cap H_1$, 
      $H(0,0,1)\cap H(1,0,1) \cap H_1$, \\
      &$H(0,0,1)\cap H(0,1,1) \cap H_2$  \\ \cline{2-2}
      & $H(1,0,0)\cap H(1,1,0) \cap H_2$, 
      $H(1,0,0)\cap H(1,0,1) \cap H_3$, \\
      &$H(0,1,0)\cap H(0,1,1) \cap H_3$  \\ \cline{2-2}
      & $H(0,1,1)\cap H(1,1,1) \cap H_1$, 
      $H(1,0,1)\cap H(1,1,1) \cap H_2$, \\
      &$H(1,1,0)\cap H(1,1,1) \cap H_3$ 
      \\ \hline  
    \end{tabular}
    \caption{Relations and intersections.}\label{table:intersection}
  \end{table}
\end{Rem}

\end{document}